\documentclass[a4paper]{article}
\usepackage[left=3cm,right=3cm,top=2.5cm,bottom=2.5cm]{geometry}
\usepackage{amsmath,amsthm,amssymb,bm}
\usepackage{graphicx,subfigure} 
\usepackage{authblk}  
\usepackage{cite}     
\usepackage{booktabs} 
\usepackage{algorithm}
\usepackage{enumitem} 
\usepackage{rotating} 
\usepackage{color}
\newtheorem{mytheorem}{Theorem}[section] 
\newtheorem{mylemma}{Lemma}[section]
\newtheorem{mycorollary}{Corollary}[section]
\theoremstyle{definition}

\newtheorem{myremark}{Remark}[section]
\numberwithin{equation}{section}  %
\numberwithin{figure}{section}    %
\numberwithin{table}{section}     %
\numberwithin{algorithm}{section}
\newcommand{\myvec}[1]{\textup{\textbf{#1}}}

\newcommand{\RS}{\mathbb{R}}
\newcommand{\NS}{\mathbb{N}}
\newcommand{\diag}{\mathrm{diag}\,} 
\newcommand{\ball}{\bm{\mathrm{B}}} 
\newcommand{\mylim}{\lim\limits}
\newcommand{\mysum}{\sum\limits}
\newcommand{\range}{\mathrm{range}}
\begin{document}

\title{On the convergence of two-step modified Newton method for 
  nonsymmetric algebraic Riccati equations from transport theory}

\author{
    Juan Liang and
    Yonghui Ling\thanks{Corresponding author. \textit{E-mail address:} \texttt{yhling@mnnu.edu.cn}}
    \and
    Department of Mathematics, Minnan Normal University, Zhangzhou 363000, China
    }

\maketitle
\begin{abstract}
  This paper is concerned with the convergence of a two-step modified Newton method for 
  solving the nonlinear system arising from the minimal nonnegative solution of
  nonsymmetric algebraic Riccati equations from neutron transport theory.
  We show the monotonic convergence of the two-step modified Newton method under mild assumptions.
  When the Jacobian of the nonlinear operator at the minimal positive solution is singular,
  we present a convergence analysis of the two-step modified Newton method in this context.
  Numerical experiments are conducted to demonstrate that the proposed method 
  yields comparable results to several existing Newton-type methods
  and that it brings a significant reduction in computation time for nearly singular and large-scale problems.

  \textbf{Keywords:} Nonsymmetric algebraic Riccati equation, minimal positive solution, 
  two-step modified Newton method, monotone convergence, singular problems
\end{abstract}
\section{Introduction}
\label{sec:Introduction}

Our aim in this paper is to study effective solutions of 
nonsymmetric algebraic Riccati equation (NARE) from neutron transport theory as follows form:
\begin{equation}
  \label{eq:NARE}
  XCX - XD - AX + B = 0,
\end{equation}
where $X \in \RS^{n \times n}$ is an unknown matrix, 
and $A, B, C, D \in \RS^{n \times n}$ are known matrices given by 
\begin{equation}
  \label{mat:ABCD}
  A = \Delta - \myvec{e}\myvec{q}^{\top}, \quad
  B = \myvec{e}\myvec{e}^{\top}, \quad
  C = \myvec{q}\myvec{q}^{\top}, \quad
  D = \Gamma - \myvec{q}\myvec{e}^{\top},
\end{equation}
with
$$
\left\{\begin{aligned}
  \Delta &= \diag(\delta_1,\delta_2,\ldots,\delta_n), & \delta_i &= \frac{1}{c\omega_i(1+\alpha)} > 0, \\
  \Gamma &= \diag(\gamma_1,\gamma_2,\ldots,\gamma_n), & \gamma_i &= \frac{1}{c\omega_i(1-\alpha)} > 0, \\
  \myvec{q} &= (q_1,q_2,\ldots,q_n)^{\top}, & q_i &= \frac{c_i}{2\omega_i} > 0, \\
  \myvec{e} &= (1,1,\ldots,1)^{\top}.
\end{aligned}\right.
$$
The matrices and vectors above depend on the two parameters
\begin{equation}
  \label{cons:alpha_c}
  c \in (0,1] \quad \text{and} \quad \alpha \in [0,1).
\end{equation}
Moreover, $\{\omega_i\}_{i=1}^n$ and $\{c_i\}_{i=1}^n$ are the sets of the Gauss-Legendre nodes and weights, 
respectively, on the interval $[0,1]$, and satisfy 
$$
0 < \omega_n < \cdots < \omega_2 < \omega_1 < 1 \ \text{and} \ \sum_{i=1}^n c_i = 1 \ \text{with} \ c_i > 0.
$$
Clearly, $\{\delta_i\}_{i=1}^n$ and $\{\gamma_i\}_{i=1}^n$ are
strictly monotonically increasing, and
$$
\begin{cases}
\delta_i = \gamma_i, & \mbox{when $\alpha = 0$},\\
\delta_i \neq \gamma_i, & \mbox{when $\alpha \neq 0$},
\end{cases}
\quad i = 1,2,\ldots,n.
$$
The NARE \eqref{eq:NARE} is obtained by a discretization of an integrodifferential equation
describing neutron transport during a collision process.
The solution of interest from a physical perspective is the minimal nonnegative solution 
\cite{Juang1995,JuangL1998,Guo2001,BiniIM2012}.

Most of developed numerical methods in the last two decades for solving NARE \eqref{eq:NARE} 
fall into one of three categories: Newton-type methods \cite{GuoL2000b,Guo2001,BiniIP2008,LinBW2008,LinB2008,HuangKH2010,LinBJ2010,BentbibJS2015}, 
fixed-point methods \cite{Lu2005a,BaoLW2006,BaiGL2008,Lin2008,GuoL2010,LinBW2011,HuangM2014,HuangM2020,SunGW2023}
and the structure-preserving doubling methods \cite{GuoIM2008,LiCKL2013,GuoG2014,Guo2016}.
In the present paper, we are concerned with the algorithms based on Newton-type iterations.
Lu \cite{Lu2005a} first proved that the solution of \eqref{eq:NARE} must have the following form:
$$
X = T \circ (\myvec{u} \myvec{v}^{\top}) = (\myvec{u} \myvec{v}^{\top}) \circ T,
$$
where $\circ$ denotes the Hadamard product,
$T = (t_{ij})_{n \times n} = \left(\frac{1}{\delta_i + \gamma_j}\right)_{n\times n}$,
$\myvec{u}$ and $\myvec{v}$ are vectors satisfying
\begin{equation}
\label{eq:VecEq_uv}
\left\{
\begin{aligned}
\myvec{u} &= \myvec{u} \circ (P\myvec{v}) + \myvec{e},\\
\myvec{v} &= \myvec{v} \circ (\widetilde{P}\myvec{u}) + \myvec{e},
\end{aligned}
\right.
\end{equation}
with
\begin{equation}
\label{mat:P_Ptilde}
  P
    = (p_{ij})_{n\times n}
    = \left(\frac{q_j}{\delta_i + \gamma_j}\right)_{n\times n}, \quad 
  \widetilde{P}
    = (\widetilde{p}_{ij})_{n\times n}
    = \left(\frac{q_j}{\gamma_i + \delta_j}\right)_{n\times n}.
\end{equation}
We set $\myvec{x} = [\myvec{u}^{\top},\myvec{v}^{\top}]^{\top} \in \RS^{2n}$.
Then the objective of finding the minimal nonnegative solution of \eqref{eq:NARE} is equivalent to
finding solutions for the nonlinear system
\begin{equation}
\label{eq:f(x)=0}
\myvec{f}(\myvec{x}) = 
\myvec{f}(\myvec{u},\myvec{v}) \stackrel{\text{def}}{=}
\begin{bmatrix} 
  \myvec{u} - \myvec{u} \circ (P\myvec{v}) - \myvec{e} \\ 
  \myvec{v} - \myvec{v} \circ (\widetilde{P}\myvec{u}) - \myvec{e}
\end{bmatrix} = \myvec{0},
\end{equation}
where $\myvec{f}: \RS^{2n} \to \RS^{2n}$.
The advantage in representing \eqref{eq:VecEq_uv} as the nonlinear system \eqref{eq:f(x)=0} is that
we now can use the Newton-type methods to solve it.
It is worth highlighting that the Jacobian of $\myvec{f}$ 
at the minimal positive solution $\myvec{x}^* \in \RS^{2n}$ of \eqref{eq:f(x)=0} 
is a singular $M$-matrix if and only if $\alpha = 0$ and $c = 1$.
For further details, refer to \cite{JuangL1998,GuoL2000b}.

Lu \cite{Lu2005b} investigated the monotone convergence of the standard Newton method for solving the nonlinear system \eqref{eq:f(x)=0},
and obtained an iterative algorithm in combination with the fixed-point iteration.
To accelerate the convergence of the Newton method,
Lin et al. \cite{LinBW2008} applied the two-step Newton method 
$$
\left\{\begin{aligned}
  \myvec{y}_k &= \myvec{x}_k + \myvec{f}'(\myvec{x}_{k})^{-1}\myvec{f}(\myvec{x}_k), \\
  \myvec{x}_{k+1} &= \myvec{y}_{k} - \myvec{f}'(\myvec{x}_{k})^{-1}\myvec{f}(\myvec{y}_{k}),
\end{aligned}\right. \quad k = 0,1,2,\ldots
$$
to solve \eqref{eq:f(x)=0}. 
It is worth noting that the scalar form of this method is a special case ($\beta = 1$) of 
one-parameter family of two-step Newton methods with the third order iteration function,
as described in Traub's book \cite[p. 181]{Traub1964}:
$$
\phi(x) = x - \frac{\beta^2-\beta-1}{\beta^2} \frac{g(x)}{g'(x)} - \frac{1}{\beta^2}\frac{g(x+\beta g(x)/g'(x))}{g'(x)}, \quad \beta \neq 0,
$$
and was further rediscovered and studied by Kou et al. \cite{KouLW2006},
where $g: \mathbb{D} \subset \RS \to \RS$ is a continuously differentiable function with $\mathbb{D}$ an open interval.
Subsequently, Ling and Xu \cite{LingX2017} proved the monotone convergence of this one-parameter family of two-step Newton methods.
For $\beta = -1$ in the finite dimensional case, that is, the classical two-step Newton method,
$$
\left\{\begin{aligned}
  \myvec{y}_k &= \myvec{x}_k - \myvec{f}'(\myvec{x}_{k})^{-1}\myvec{f}(\myvec{x}_k), \\
  \myvec{x}_{k+1} &= \myvec{y}_{k} - \myvec{f}'(\myvec{x}_{k})^{-1}\myvec{f}(\myvec{y}_{k}),
\end{aligned}\right. \quad k = 0,1,2,\ldots,
$$
Ling et al. \cite{LingLL2022} performed a semilocal convergence analysis under some mild generalized Lipschitz conditions,
and applied the results to solve the nonlinear system \eqref{eq:f(x)=0}.
Both two-step Newton methods require one evaluation of the Jacobian 
and two evaluations of the function per iteration. 
In comparison, they require one additional function evaluation per iteration than the standard Newton method.
However, they demonstrate faster convergence, 
which may result in improved computational performance in specific nonlinear problems \cite{ChenWS2018,MaC2020,ChenO2024}.

Although many other iterative methods for solving nonlinear operator equations are cubically convergent 
(see for example \cite{LingH2017,LingX2014,EzquerroH2009b} and references therein), 
they typically require a higher computational cost per iteration than Newton's method.
To accelerate the convergence of the Newton method, 
while maintaining the same computational cost per iteration,
a two-step modified Newton method \cite{McDougallW2014} has recently been proposed.
Starting from an initial point $\myvec{x}_0 \in \RS^{2n}$, 
the two-step modified Newton method in multidimensional form is defined iteratively by
\begin{equation}
  \label{it:TSMNM}
  \left\{\begin{aligned}
    \myvec{y}_{k} &= \myvec{x}_k - \myvec{f}'(\myvec{z}_{k-1})^{-1}\myvec{f}(\myvec{x}_k), \\
    \myvec{x}_{k+1} &= \myvec{x}_{k} - \myvec{f}'(\myvec{z}_{k})^{-1}\myvec{f}(\myvec{x}_{k}),
  \end{aligned}\right.\quad k = 0,1,2,\ldots,
\end{equation}
where $\myvec{z}_{-1} = \myvec{x}_0$ and $\myvec{z}_k = (\myvec{x}_k + \myvec{y}_k)/2$ for $k \geq 0$.
We note that at iteration $k$, the two-step modified Newton method requires only one Jacobian evaluation
and one function evaluation, since the Jacobian $\myvec{f}'(\myvec{z}_{k-1})$ has already been evaluated at the previous iteration.
Therefore, the computational cost per iteration of the two-step modified Newton method is comparable to that of Newton's method,
except for the first iteration, which requires an additional Jacobian evaluation.
The first work on the convergence theory for the two-step modified Newton method was developed 
by Potra in \cite{Potra2017}, where a rigorous and comprehensive convergence analysis, 
including both semilocal and local convergence, was provided.
It was shown in \cite{Potra2017} that the two-step modified Newton method exhibits 
locally superquadratic convergence under the assumptions that the derivatives of the function
satisfy the Lipschitz conditions.
Using the majorizing function technique, 
which is extensively used in the convergence analysis of Newton-type methods (see for example \cite{LingLL2022} and references therein), 
Cárdenas et al. \cite{CardenasCS2020,CardenasCS2022} recently established new semilocal convergence 
under some assumptions on the second derivative of the function.

To the best of our knowledge, there are no results available on the convergence of 
the two-step modified Newton method \eqref{it:TSMNM} under the assumption that 
the Jacobian $\myvec{f}'(\myvec{x}^*)$ is singular.
In contrast, 
singular problems for other Newton-type methods have been extensively studied,
including those for Newton's method in references 
\cite{Reddien1978,DeckerK1980a,DeckerK1980b,DeckerKK1983,OberlinW2009,IzmailovKS2018a,IzmailovKS2018b,FischerIS2021}, 
inexact Newton methods in \cite{KelleyX1993,Argyros1999}, 
and quasi-Newton methods in \cite{DeckerK1985,BuhmilerKL2010,Mannel2023}.

Motivated by the potential and advantages of the two-step modified Newton method \eqref{it:TSMNM},
in this paper we investigate its convergence behavior for solving the nonlinear system \eqref{eq:f(x)=0}.
Specifically, we show that the sequence generated by the two-step modified Newton method \eqref{it:TSMNM}
with zero initial guess or some other suitable initial guess is well-defined 
and converges monotonically to the minimal positive solution of the system \eqref{eq:f(x)=0}.
When the Jacobian $\myvec{f}'(\myvec{x}^*)$ is nonsingular (i.e., $\alpha \neq 0$ or $c \neq 1$),
and the convergence criterion given by Potra in \cite{Potra2017} is satisfied,
we can obtain the local quadratic convergence of the two-step modified Newton method \eqref{it:TSMNM}.
For the case when the Jacobian $\myvec{f}'(\myvec{x}^*)$ is singular (i.e., $\alpha = 0$ and $c = 1$),
we consider two classes of assumptions on the singularity of $\myvec{f}'(\myvec{x}^*)$,
and establish the local convergence of the two-step modified Newton method \eqref{it:TSMNM}.
The underlying approach in our convergence analysis is based on a technique for 
approximating the inverse of the derivative near a given point as developed in \cite{DeckerK1980a,DeckerK1980b,DeckerKK1983}.
We implement the two-step modified Newton method \eqref{it:TSMNM} to solve the nonlinear system \eqref{eq:f(x)=0}. 
Our preliminary numerical results exhibit the superiority of the proposed method over other Newton-type methods. 
In particular, the experiments show that the two-step modified Newton method leads to 
a significant reduction in computation time for nearly singular and large-scale problems.

The rest of this paper is organized as follows. 
In Section \ref{sec:Preliminaries}, we present some preliminaries that will be used in the convergence analysis.
We give the iterative algorithm based on the two-step modified Newton method \eqref{it:TSMNM} 
for solving the nonlinear system \eqref{eq:f(x)=0} in Section \ref{sec:TwoStepModifiedNewton}.
In Section \ref{sec:ConvergenceAnalysis}, we analyze the convergence of the two-step modified Newton method \eqref{it:TSMNM}.
Numerical experiments are presented in Section \ref{sec:NumericalExperiments} to illustrate the effectiveness of the proposed algorithm.
Finally, we conclude the paper in Section \ref{sec:Conclusions}.

\section{Preliminaries}
\label{sec:Preliminaries}

Throughout this paper, vectors are columns by default and are denoted by bold lowercase letters, e.g., $\myvec{v}$,
while matrices are denoted by regular uppercase letters, e.g., $V$, which is clear from the context.
We use $\diag(\myvec{v})$ to denote the diagonal matrix with the vector $\myvec{v}$ on its diagonal,
and use $I$ to denote the identity matrix with proper dimension.
If there is potential confusion, we will use $I_n$ to denote the identity matrix of dimension $n$.
The symbol $\myvec{e}_i = (0,\ldots,0,\underset{i}{1},0,\ldots,0)^{\top} \in \RS^n$ is $i$th column of the identity matrix $I_n$.
For any two nonnegative numbers $\mu$ and $\nu$, 
we write $\mu = \mathcal{O}(\nu)$ if there exists a positive constant $M$ such that $\mu \leq M\nu$.

For any real matrices $A = (a_{ij})_{m \times n}$ and $B = (b_{ij})_{m \times n}$,
we write $A \geq B \, (A > B)$ if $a_{ij} \geq b_{ij} \, (a_{ij} > b_{ij})$ for all $i = 1,2,\ldots,m$ and $j = 1,2,\ldots,n$.
We call real matrix $A = (a_{ij})_{m \times n}$ a positive matrix (nonnegative matrix)
if $a_{ij} > 0 \, (a_{ij} \geq 0)$ hold for all $i = 1,2,\ldots,m$ and $j = 1,2,\ldots,n$,
and we write $A > 0 \,(A \geq 0)$.
We denote by $A \circ B = (a_{ij}\cdot b_{ij})_{m \times n}$ the Hadamard product of $A$ and $B$.
Moreover, for any real vectors $\myvec{a} = (a_1,a_2,\ldots,a_n)^{\top}$
and $\myvec{b} = (b_1,b_2,\ldots,b_n)^{\top}$, 
we write $\myvec{a} \geq \myvec{b} \,(\myvec{a} > \myvec{b})$ if $a_i \geq b_i \,(a_i > b_i)$ for all $i = 1,2,\ldots,n$.
The vector of all zero components is denoted by $\myvec{0}$.
If all the components of a vector $\myvec{v} \in \RS^n$ are positive (nonnegative), 
we call $\myvec{v}$ a positive (nonnegative) vector, and we write $\myvec{v} > \myvec{0} \,(\myvec{v} \geq \myvec{0})$.
A vector sequence $\{\myvec{v}_k\}_{k=1}^{\infty} \subset \RS^n$ is called monotonic if $\myvec{v}_{k+1} \geq \myvec{v}_k$ for all $k = 1,2,\ldots$.

A real square matrix $A = (a_{ij})_{n \times n}$ is called a $Z$-matrix if $a_{ij} \leq 0$ for all $i \neq j$.
Any $Z$-matrix $A$ can be written as
$$
 A = sI - B,
$$
where $s \in \RS$ and matrix $B$ is nonnegative.
Furthermore, a $Z$-matrix $A$ is called a nonsingular $M$-matrix if $s > \rho(B)$, 
where $\rho(B)$ is the spectral radius of $B$.
The following lemma, which is taken from \cite[Theorem 2.3 in Chapter 6, p. 137]{BermanP1994}, 
gives some criteria for determining whether the $Z$-matrix is a nonsingular $M$-matrix.

\begin{mylemma}
  \label{lem:NonsingularMmatrix}
For a $Z$-matrix $A \in \RS^{n \times n}$, the following statements are equivalent:
  \begin{enumerate}[label=\textup{(\roman*)}]
    \item $A$ is a nonsingular $M$-matrix.
    \item $A$ is inverse-positive. That is, $A$ is nonsingular and $A^{-1} \geq 0$.
    \item $A$ is semipositive. That is, $A\myvec{v} > \myvec{0}$ holds for some vector $\myvec{v} > \myvec{0}$.
  \end{enumerate}
\end{mylemma}

The next well-known lemma is a direct consequence of the equivalence of (i) and (iii) in the above lemma. 
See \cite[Lemma 1]{Guo2013} for example.

\begin{mylemma}
  \label{lem:A-1>B-1}
  Let $A, B \in \RS^{n \times n}$ be Z-matrices.
  If $A$ is an $M$-matrix and $B \geq A$, then $B$ is also an $M$-matrix and $A^{-1} \geq B^{-1} \geq 0$.
\end{mylemma}

Recall that the Jacobian matrix of a continuously differentiable nonlinear operator 
$\myvec{g}: \RS^{n} \to \RS^{n}$ at point $\myvec{x} \in \RS^n$ is represented by $\myvec{g}'(\myvec{x})$.
If $\myvec{g}$ is twice continuously differentiable, 
then the Hessian matrix of $\myvec{g}$ at point $\myvec{x} \in \RS^n$ is denoted by $\myvec{g}''(\myvec{x})$,
and can be viewed as a bilinear mapping from $\RS^n \times \RS^n$ to $\RS^n$.
For convenience, for any $\myvec{u},\myvec{v} \in \RS^n$, 
we use the notation $\myvec{g}''(\myvec{x})\myvec{u}\myvec{v}$ to denote the element
$\myvec{g}''(\myvec{x})(\myvec{u},\myvec{v})$ in $\RS^n$.
It is worth noting that the Hessian matrix $\myvec{g}''(\myvec{x})$ is symmetric. 
That is,
$$
\myvec{g}''(\myvec{x})\myvec{u}\myvec{v} = \myvec{g}''(\myvec{x})\myvec{v}\myvec{u}, \quad
\forall\, \myvec{u},\myvec{v} \in \RS^n.
$$
See \cite{OrtegaR1970} for more details.
In addition, we have the following Taylor's formulas:
\begin{equation}
  \label{eq:Taylor}
  \myvec{g}(\myvec{x} + \myvec{h}) = \myvec{g}(\myvec{x}) + \myvec{g}'(\myvec{x})\myvec{h} 
  + \frac{1}{2}\myvec{g}''(\myvec{x})\myvec{h}\myvec{h} + o(\|\myvec{h}\|^2), \quad \myvec{h} \in \ball(\myvec{0},\delta),
\end{equation}
where $\ball(\myvec{0},\delta)$ is the open ball centered at $\myvec{0}$ with radius $\delta > 0$.

For any subspace $\mathcal{X} \subset \RS^n$, $\dim(\mathcal{X})$ denotes the dimension of $\mathcal{X}$.
The kernel or null space of a linear operator $A$ is denoted $\ker(A)$,
the image or range of the operator is denoted $\range(A)$.
$\range(A)$ and $\ker(A)$ are all subspaces of $\RS^n$.
Recall that a linear operator $P$ is called a projection if $P^2 = P$.
That is, projection $P$ is idempotent.
Note that if $P$ is a projection, 
then $I - P$ is also a projection, and 
$$
\range(P) = \ker{(I - P)}, \quad \ker(P) = \range(I-P), \quad \range(P) \oplus \ker{P} = \RS^n.
$$
One can see \cite{TrefethenB1997,HornJ2013,GolubV1996} for more details.
Let $P_{\mathcal{X}}$ be denoted the orthogonal projection onto the subspace $\mathcal{X}$.
Then $P_{\mathcal{X}}\myvec{x}$ must be an element of $\mathcal{X}$ for any $\myvec{x} \in \RS^n$.
When we choose $\mathcal{X} = \ker(A)$, 
for any $\myvec{x} \in \RS^n$ we have $A(P_{\mathcal{X}}\myvec{x}) = \myvec{0}$.
In addition, we use $A\big|_{\mathcal{X}}$ to denote the restriction of the operator $A$ to the subspace $\mathcal{X}$.
For any $\myvec{x} \in \mathcal{X}$, 
we remark that the norm on $\mathcal{X}$ is the same as the norm on $\RS^n$.
We conclude this section with a well-known result on the bounds of the norms of matrix-vector multiplication.
One can see \cite{Grcar2010} for more details.

\begin{mylemma}
  \label{lem:MatrixLowerBound}
  The matrix lower bound exists and is positive for any nonzero matrix. In particular,
  if $A \in \RS^{n \times n}$ is nonsingular, then we have
  $$
  \|A^{-1}\|^{-1}\|\myvec{x}\| \leq \|A\myvec{x}\| \leq \|A\|\|\myvec{x}\|, \quad \forall\, \myvec{x} \in \RS^n. 
  $$
  That is, the matrix lower bound is $\|A^{-1}\|^{-1}$. 
  Moreover, if the vector norms are $2$-norms, 
  then the matrix lower bound equals the smallest singular value of $A$.
\end{mylemma}

\section{Two-step modified Newton method}
\label{sec:TwoStepModifiedNewton}

Recall that the matrices $P$ and $\widetilde{P}$ are defined in \eqref{mat:P_Ptilde}.
Let $P = [\myvec{p}_1,\myvec{p}_2,\ldots,\myvec{p}_n] \in \RS^{n \times n}$ 
and $\widetilde{P} = [\widetilde{\myvec{p}}_1,\widetilde{\myvec{p}}_2,\ldots,\widetilde{\myvec{p}}_n] \in \RS^{n \times n}$
be column partitions. 
Clearly, $\myvec{f}$ defined by \eqref{eq:f(x)=0} is a continuously Fr\'{e}chet differentiable nonlinear operator in $\RS^{2n}$.
The Jacobican matrix of $\myvec{f}(\myvec{u},\myvec{v})$ at point $(\myvec{u},\myvec{v})$ has the following form (see \cite{Lu2005b}): 
\begin{equation}
  \label{eq:Jacobian_f}
  \myvec{f}'(\myvec{u},\myvec{v}) = I_{2n} - G(\myvec{u},\myvec{v}),
\end{equation}
where
\begin{equation}
  \label{mat:G}
  G(\myvec{u},\myvec{v}) 
  = \begin{bmatrix}
    G_1(\myvec{v}) & H_1(\myvec{u}) \\
    H_2(\myvec{v}) & G_2(\myvec{u})
  \end{bmatrix}
\end{equation}
with
\begin{align*}
  G_1(\myvec{v}) &= \diag(P\myvec{v}), \quad 
    H_1(\myvec{u}) = [\myvec{u}\circ\myvec{p}_1,\myvec{u}\circ\myvec{p}_2,\ldots,\myvec{u}\circ\myvec{p}_n], \\
  G_2(\myvec{u}) &= \diag(\widetilde{P}\myvec{u}), \quad 
    H_2(\myvec{v}) = [\myvec{v}\circ\widetilde{\myvec{p}}_1,\myvec{v}\circ\widetilde{\myvec{p}}_2,\ldots,\myvec{v}\circ\widetilde{\myvec{p}}_n].
\end{align*}
For any $\myvec{x} = [\myvec{u}^{\top},\myvec{v}^{\top}]^{\top} \in \RS^{2n}$, we have
\begin{equation}
  \label{eq:f''(x)h1h2}
  \myvec{f}''(\myvec{x})\myvec{h}_1\myvec{h}_2 
  = [\myvec{h}_1^{\top}L_1^{\top}\myvec{h}_2,\ldots,\myvec{h}_1^{\top}L_{2n}^{\top}\myvec{h}_2]^{\top} \in \RS^{2n},
    \quad \forall\, \myvec{h}_1,\myvec{h}_2 \in \RS^{2n},
\end{equation}
where
$$
L_i = \begin{bmatrix}
  O & (-\myvec{e}_iP_i^{\top}) \\
  (-\myvec{e}_iP_i^{\top})^{\top} & O 
\end{bmatrix}, \ 
L_{n+i} = \begin{bmatrix}
  O & (-\widetilde{P}_i\myvec{e}_i^{\top}) \\
  (-\widetilde{P}_i\myvec{e}_i^{\top})^{\top} & O
\end{bmatrix}, \ i = 1,2,\ldots,n,
$$
with $P_i^{\top} = (p_{i1},p_{i2},\ldots,p_{in})$ and 
$\widetilde{P}_i^{\top} = (\widetilde{p}_{i1},\widetilde{p}_{i2},\ldots,\widetilde{p}_{in})$ 
being the $i$th row of the matrices $P$ and $\widetilde{P}$, respectively.
Clearly, all the matrices $L_i$ and $L_{n+i}$ are independent of $\myvec{x}$, symmetric Z-matrices.
This implies that $\myvec{f}''(\myvec{x})\myvec{h}\myvec{h}$ is independent of $\myvec{x}$.
Moreover, $\myvec{f}'''(\myvec{x})$ is the null operator.
This allows us to use the Taylor formula \eqref{eq:Taylor} to derive the following form:
\begin{equation}
  \label{eq:TaylorExpansion}
  \myvec{f}(\myvec{x} + \myvec{h}) = \myvec{f}(\myvec{x}) + \myvec{f}'(\myvec{x})\myvec{h} 
  + \frac{1}{2}\myvec{f}''(\myvec{x})\myvec{h}\myvec{h}, \quad \myvec{h} \in \RS^{2n}.
\end{equation}
The above expansion will frequently be used in the convergence analysis of the two-step modified Newton method \eqref{it:TSMNM}.

To apply the two-step modified Newton method \eqref{it:TSMNM} to solve \eqref{eq:f(x)=0},
we choose an initial guess $\myvec{x}_0 = [\myvec{u}_0^{\top},\myvec{v}_0^{\top}]^{\top} \in \RS^{2n}$, 
and set 
$$
\myvec{x}_k = [\myvec{u}_k^{\top},\myvec{v}_k^{\top}]^{\top}, \quad
\myvec{y}_k = [\overline{\myvec{u}}_k^{\top},\overline{\myvec{v}}_k^{\top}]^{\top}, \quad
\myvec{z}_k = [\widetilde{\myvec{u}}_k^{\top},\widetilde{\myvec{v}}_k^{\top}]^{\top}.
$$
The algorithm for implementing the two-step modified Newton method \eqref{it:TSMNM} is summarized in Algorithm \ref{alg:TSMNM} as follows.

\begin{algorithm}
  \caption{Two-step modified Newton method for solving (\ref{eq:f(x)=0})} 
  \label{alg:TSMNM}
  \textit{Initialization}. Given $c \in (0, 1]$ and $\alpha \in [0,1)$. 
  Form the matrices $P$ and $\tilde{P}$ by \eqref{mat:P_Ptilde}.
  Choose an initial point $[\myvec{u}^{\top}_0,\myvec{v}^{\top}_0]^{\top} \in \RS^{2n}$.
  \begin{itemize}[leftmargin=1em,itemindent=3.5em,parsep=0em,itemsep=0em,topsep=0em]
    \item[Step 1.]
    Form the matrix $G(\myvec{u}_0,\myvec{v}_0)$ by \eqref{mat:G}. 
    Compute $\overline{\myvec{v}}_0$ from the system of linear equations below:
    \begin{align*}
      \lefteqn{\left[I_n - G_2(\myvec{u}_0) - H_2(\myvec{v}_0)\big(I_n - G_1(\myvec{v}_0)\big)^{-1}H_1(\myvec{u}_0)\right]\overline{\myvec{v}}_0 } \\
        &\quad= H_2(\myvec{v}_0)\big(I_n - G_1(\myvec{v}_0)\big)^{-1}(\myvec{e} - H_1(\myvec{u}_0)\myvec{v}_0) + \myvec{e} - H_2(\myvec{v}_0)\myvec{u}_0.
    \end{align*}
    \item[Step 2.]
    Compute $\myvec{u}_0 = \big(I_n - G_1(\myvec{v}_0)\big)^{-1}[\myvec{e} + H_1(\myvec{u}_0)(\overline{\myvec{v}}_0 - \myvec{v}_0)]$ and set
    $\widetilde{\myvec{u}}_0 = (\myvec{u}_0 + \overline{\myvec{u}}_0)/2$,
    $\widetilde{\myvec{v}}_0 = (\myvec{v}_0 + \overline{\myvec{v}}_0)/2$.
    \item[Step 3.] 
    Compute $\myvec{v}_1$ from the system of linear equations below:
    \begin{align*}
      \lefteqn{\left[I_n - G_2(\widetilde{\myvec{u}}_0) - H_2(\widetilde{\myvec{v}}_0)\big(I_n - G_1(\widetilde{\myvec{v}}_0)\big)^{-1}H_1(\widetilde{\myvec{u}}_0)\right]\myvec{v}_1 } \\
        &\quad= H_2(\widetilde{\myvec{v}}_0)\big(I_n - G_1(\widetilde{\myvec{v}}_0)\big)^{-1}
          [\myvec{e} + \myvec{u}_0\circ(P\myvec{v}_0) - G_1(\widetilde{\myvec{v}}_0)\myvec{u}_0 - H_1(\widetilde{\myvec{u}}_0)\myvec{v}_0] \\
        &\qquad + \myvec{e} + \myvec{v}_0 \circ (\widetilde{P}\myvec{u}_0) - H_2(\widetilde{\myvec{v}}_0)\myvec{u}_0 - G_2(\widetilde{\myvec{u}}_0)\myvec{v}_0.
    \end{align*}
    \item[Step 4.]
    Compute $\myvec{u}_1$ from the following formula:
    $$
    \myvec{u}_1 = \big(I_n - G_1(\widetilde{\myvec{v}}_0)\big)^{-1}
    [\myvec{e} + H_1(\widetilde{\myvec{u}}_0)(\myvec{v}_1 - \widetilde{\myvec{v}}_0) + \myvec{u}_0\circ(P\myvec{v}_0) - G_1(\widetilde{\myvec{v}}_0)\myvec{u}_0].
    $$
  \end{itemize}
  \textit{Iterative process}. For $k = 1,2,\ldots$ until convergence, do:
  \begin{itemize}[leftmargin=1em,itemindent=3.5em,parsep=0em,itemsep=0em,topsep=0em]
    \item[Step 1.]
    Form the matrix $G(\myvec{u}_k,\myvec{v}_k)$ by \eqref{mat:G}.
    Compute $\overline{\myvec{v}}_k$ from the system of linear equations below:
    \begin{align*}
      \lefteqn{\left[I_n - G_2(\widetilde{\myvec{u}}_{k-1}) - H_2(\widetilde{\myvec{v}}_{k-1})
      \big(I_n - G_1(\widetilde{\myvec{v}}_{k-1})\big)^{-1}H_1(\widetilde{\myvec{u}}_{k-1})\right]\overline{\myvec{v}}_k } \\
        &= H_2(\widetilde{\myvec{v}}_{k-1})\big(I_n - G_1(\widetilde{\myvec{v}}_{k-1})\big)^{-1}
        [\myvec{e} + \myvec{u}_k\circ(P\myvec{v}_k) - G_1(\widetilde{\myvec{v}}_{k-1})\myvec{u}_k - H_1(\widetilde{\myvec{u}}_{k-1})\myvec{v}_k] \\
        &\quad + \myvec{e} + \myvec{v}_k\circ(\widetilde{P}\myvec{u}_k) - H_2(\widetilde{\myvec{v}}_{k-1})\myvec{u}_k - G_2(\widetilde{\myvec{u}}_{k-1})\myvec{v}_k.
    \end{align*}
    \item[Step 2.]
    Compute $\overline{\myvec{u}}_k$ from the following formula:
    $$
    \overline{\myvec{u}}_k
      = \big(I_n - G_1(\widetilde{\myvec{v}}_{k-1})\big)^{-1}
      [\myvec{e} + H_1(\widetilde{\myvec{u}}_{k-1})(\overline{\myvec{v}}_k - \myvec{v}_k) + \myvec{u}_k\circ(P\myvec{v}_k) - G_1(\widetilde{\myvec{v}}_{k-1})\myvec{u}_k],
    $$
    and set $\widetilde{\myvec{u}}_k = (\myvec{u}_k + \overline{\myvec{u}}_k)/2$,
    $\widetilde{\myvec{v}}_k = (\myvec{v}_k + \overline{\myvec{v}}_k)/2$.
    \item[Step 3.]
    Compute $\myvec{v}_{k+1}$ from the system of linear equations below:
    \begin{align*}
    \lefteqn{\left[I_n - G_2(\widetilde{\myvec{u}}_k) - H_2(\widetilde{\myvec{v}}_k)
    \big(I_n - G_1(\widetilde{\myvec{v}}_k)\big)^{-1}H_1(\widetilde{\myvec{u}}_k)\right]\myvec{v}_{k+1} } \\
      &\quad= H_2(\widetilde{\myvec{v}}_k)\big(I_n - G_1(\widetilde{\myvec{v}}_k)\big)^{-1}
      [\myvec{e} + \myvec{u}_k\circ(P\myvec{v}_k) - G_1(\widetilde{\myvec{v}}_k)\myvec{u}_k - H_1(\widetilde{\myvec{u}}_k)\myvec{v}_k] \\
      &\qquad + \myvec{e} + \myvec{v}_k\circ(\widetilde{P}\myvec{u}_k) - H_2(\widetilde{\myvec{v}}_k)\myvec{u}_k - G_2(\widetilde{\myvec{u}}_k)\myvec{v}_k.
    \end{align*}
    \item[Step 4.]
    Compute $\myvec{u}_{k+1}$ from the following formula:
    $$
    \myvec{u}_{k+1} = \big(I_n - G_1(\widetilde{\myvec{v}}_k)\big)^{-1}
    [\myvec{e} + H_1(\widetilde{\myvec{u}}_k)(\myvec{v}_{k+1} - \myvec{v}_k) + \myvec{u}_k\circ(P\myvec{v}_k) - G_1(\widetilde{\myvec{v}}_k)\myvec{u}_k].
    $$
  \end{itemize}
\vspace{-2mm}
\end{algorithm}

The convergence results provided in Section \ref{sec:ConvergenceAnalysis} guarantee that 
the aforementioned algorithm is both well-defined and convergent.

\section{Convergence analysis}
\label{sec:ConvergenceAnalysis}

In this section, we first establish the monotone convergence result for 
the two-step modified Newton method \eqref{it:TSMNM},
and then analyze its convergence rates at singular roots.

\subsection{The monotone convergence}

Assume that $\myvec{x}^* \in \RS^{2n}$ is the minimal positive solution of the equation \eqref{eq:f(x)=0}.
To show the monotone convergence of the two-step modified Newton method \eqref{it:TSMNM},
we need some lemmas.
The following lemma is taken from \cite[Lemma 5]{Lu2005b}.

\begin{mylemma}[{\cite{Lu2005b}}]
  \label{lem:f''(x)hh}
  For any $\myvec{h} \in \RS^{2n}$, $\myvec{f}''(\myvec{x})\myvec{h}\myvec{h}$ is independent of $\myvec{x} \in \RS^{2n}$.
  In particular, we have $\myvec{f}''(\myvec{x})\myvec{h}\myvec{h} < \myvec{0}$ for any $\myvec{h} \in \RS^{2n}\setminus\{\myvec{0}\}$.
\end{mylemma}

Moreover, for any $\myvec{h}_1, \myvec{h}_2,\myvec{h}_3 \in \RS^{2n}$, we have
\begin{equation}
  \label{ineq:f''(x)h1h2h3}
  \begin{cases}
    \myvec{f}''(\myvec{x})\myvec{h}_2\myvec{h}_2 < \myvec{f}''(\myvec{x})\myvec{h}_1\myvec{h}_1 < \myvec{0},
      & \mbox{when $\myvec{0} < \myvec{h}_1 < \myvec{h}_2$}; \\
    \myvec{f}''(\myvec{x})\myvec{h}_1\myvec{h}_2 > \myvec{0}, 
      & \mbox{when $\myvec{h}_2 < \myvec{0} < \myvec{h}_1$}; \\
      \myvec{f}''(\myvec{x})\myvec{h}_1\myvec{h}_3 > \myvec{f}''(\myvec{x})\myvec{h}_2\myvec{h}_3,
      & \mbox{when $\myvec{h}_1 < \myvec{h}_2 < \myvec{0}$ and $\myvec{h}_3 > \myvec{0}$}.
  \end{cases}
\end{equation}

The lemma below is taken from \cite[Corollary 7]{Lu2005b}.

\begin{mylemma}[{\cite{Lu2005b}}]
  \label{lem:f'invertibility}
  Let $\myvec{x}^* \in \RS^{2n}$ be the minimal positive solution of \eqref{eq:f(x)=0}.
  If $G(\myvec{x})$ is defined by \eqref{mat:G}, then $\rho\big(G(\myvec{x}^*)\big) \leq 1$.
  That is, $\myvec{f}'(\myvec{x}^*) = I_{2n} - G(\myvec{x}^*)$ is an $M$-matrix.
  In addition, for any $\myvec{x} \in \RS^{2n}$ with $\myvec{0} \leq \myvec{x} < \myvec{x}^*$,
  $\myvec{f}'(\myvec{x})$ is a nonsingular $M$-matrix.
\end{mylemma}

Since $G(\myvec{x}) < G(\myvec{y})$ when $\myvec{x} < \myvec{y}$, 
it follows that $\myvec{f}'(\myvec{x}) > \myvec{f}'(\myvec{y})$.
By combining Lemmas \ref{lem:A-1>B-1} and \ref{lem:f'invertibility}, we have the following lemma.

\begin{mylemma}
  \label{lem:f'(x)-1<f'(y)-1}
  If $\myvec{0} < \myvec{x} < \myvec{y} < \myvec{x}^*$, 
  then $0 < \myvec{f}'(\myvec{x})^{-1} < \myvec{f}'(\myvec{y})^{-1}$.
\end{mylemma}

We now present the following monotone convergence result for the two-step modified Newton method \eqref{it:TSMNM}.

\begin{mytheorem}
  \label{th:MonotoneConvergence}
  Let $\{\myvec{x}_k\}$, $\{\myvec{y}_k\}$ and $\{\myvec{z}_k\}$ be the sequences generated 
  by the two-step modified Newton method \eqref{it:TSMNM} with an appropriate initial guess $\myvec{x}_0 \in \RS^{2n}$.
  If $\myvec{0} \leq \myvec{x}_0 < \myvec{x}^*$ and $\myvec{f}(\myvec{x}_0) < \myvec{0}$,
  then the sequences $\{\myvec{x}_k\}$, $\{\myvec{y}_k\}$ and $\{\myvec{z}_k\}$ are well-defined
  and the following statements hold:
  \begin{enumerate}[label=\textup{(\roman*)}]
    \item $\myvec{f}(\myvec{x}_k) < \myvec{0}$ and $\myvec{f}'(\myvec{z}_k)$ is a nonsingular $M$-matrix 
    and $\myvec{f}'(\myvec{z}_k) > 0$ for all $k \geq 0$.
    \item $\myvec{0} \leq \myvec{x}_k < \myvec{z}_k < \myvec{y}_k < \myvec{x}_{k+1} < \myvec{x}^*$ for all $k \geq 0$.
    \item $\mylim_{k \to \infty} \myvec{x}_k = \mylim_{k \to \infty} \myvec{y}_k = \mylim_{k \to \infty} \myvec{z}_k = \myvec{x}^*$.
  \end{enumerate}
\end{mytheorem}

\begin{proof}
  We prove the theorem by induction on $k$.
  For $k = 0$, since $\myvec{0} \leq \myvec{x}_0 < \myvec{x}^*$ and $\myvec{z}_{-1} = \myvec{x}_0$,
  it follows from Lemma \ref{lem:f'invertibility} that $\myvec{f}'(\myvec{z}_{-1})$ is a nonsingular $M$-matrix.
  Then $\myvec{f}'(\myvec{z}_{-1})^{-1} \geq 0$ by Lemma \ref{lem:NonsingularMmatrix}.
  Thanks to \eqref{it:TSMNM}, we have
  $$
  \myvec{f}'(\myvec{z}_{-1}) (\myvec{y}_0 - \myvec{x}_0) = -\myvec{f}(\myvec{x}_0) > \myvec{0}.
  $$
  This leads to 
  $\myvec{y}_0 - \myvec{x}_0 
    = \myvec{f}'(\myvec{z}_{-1})^{-1}[\myvec{f}'(\myvec{z}_{-1})(\myvec{y}_0 - \myvec{x}_0)] > \myvec{0}$,
  which gives $\myvec{0} \leq \myvec{x}_0 < \myvec{y}_0$ 
  and so $\myvec{x}_0 < \myvec{z}_0 = (\myvec{x}_0 + \myvec{y}_0)/2 < \myvec{y}_0$.
  Recalling \eqref{it:TSMNM}, we obtain
  \begin{align}
    \myvec{f}'(\myvec{z}_{-1})(\myvec{y}_0 - \myvec{x}^*)
      &= \myvec{f}'(\myvec{z}_{-1})(\myvec{y}_0 - \myvec{x}_0 + \myvec{x}_0 - \myvec{x}^*) \nonumber \\
      &= \myvec{f}'(\myvec{z}_{-1})(\myvec{y}_0 - \myvec{x}_0) + \myvec{f}'(\myvec{z}_{-1})(\myvec{x}_0 - \myvec{x}^*) \nonumber \\
      &= -\myvec{f}(\myvec{x}_0) + \myvec{f}'(\myvec{z}_{-1})(\myvec{x}_0 - \myvec{x}^*). \label{eq:f'(z-1)(y0-x*)}
  \end{align}
  By Taylor's expansion \eqref{eq:TaylorExpansion}, it holds that 
  $$
  \myvec{0} = \myvec{f}(\myvec{x}^*) 
    = \myvec{f}(\myvec{z}_{-1}) + \myvec{f}'(\myvec{z}_{-1})(\myvec{x}^* - \myvec{z}_{-1})
      + \frac{1}{2}\myvec{f}''(\myvec{z}_{-1})(\myvec{x}^* - \myvec{z}_{-1})(\myvec{x}^* - \myvec{z}_{-1}).
  $$
  By substituting this expansion into \eqref{eq:f'(z-1)(y0-x*)}, we conclude from Lemma \ref{lem:f''(x)hh} that
  $$  
    \myvec{f}'(\myvec{z}_{-1})(\myvec{y}_0 - \myvec{x}^*)
      = \frac{1}{2}\myvec{f}''(\myvec{z}_{-1})(\myvec{x}^* - \myvec{z}_{-1})(\myvec{x}^* - \myvec{z}_{-1})
      < \myvec{0}.
  $$
  This yields that 
  $\myvec{y}_0 - \myvec{x}^* = \myvec{f}'(\myvec{z}_{-1})^{-1}[\myvec{f}'(\myvec{z}_{-1})(\myvec{y}_0 - \myvec{x}^*)] < \myvec{0}$,
  which gives $\myvec{y}_0 < \myvec{x}^*$ and so $\myvec{z}_0 = (\myvec{x}_0 + \myvec{y}_0)/2 < \myvec{x}^*$.
  Thus, $\myvec{f}'(\myvec{z}_0)$ is a nonsingular $M$-matrix and $\myvec{f}'(\myvec{z}_0)^{-1} \geq 0$
  by Lemma \ref{lem:NonsingularMmatrix}. 
  Since $\myvec{z}_{-1} = \myvec{x}_0 < \myvec{z}_0 < \myvec{x}^*$ and
  $$
    \myvec{x}_1 - \myvec{y}_0
      = [\myvec{f}'(\myvec{z}_{-1})^{-1} - \myvec{f}'(\myvec{z}_0)^{-1}]\myvec{f}(\myvec{x}_0),
  $$
  it follows from Lemma \ref{lem:f'(x)-1<f'(y)-1} that $\myvec{x}_1 - \myvec{y}_0 > \myvec{0}$, 
  i.e., $\myvec{x}_1 > \myvec{y}_0$.
  From \eqref{it:TSMNM} we infer that 
  \begin{align}
    \myvec{f}'(\myvec{z}_0)(\myvec{x}_1 - \myvec{x}^*)
    &= \myvec{f}'(\myvec{z}_0)(\myvec{x}_1 - \myvec{x}_0) 
      + \myvec{f}'(\myvec{z}_0)(\myvec{x}_0 - \myvec{z}_0) 
      + \myvec{f}'(\myvec{z}_0)(\myvec{z}_0 - \myvec{x}^*) \nonumber \\
    &= -\myvec{f}(\myvec{x}_0) + \myvec{f}'(\myvec{z}_0)(\myvec{x}_0 - \myvec{z}_0) 
    + \myvec{f}'(\myvec{z}_0)(\myvec{z}_0 - \myvec{x}^*) \label{eq:f'(z0)(x1-x*)}.
  \end{align}
  By Taylor's expansion \eqref{eq:TaylorExpansion} again, we have
  $$
  \myvec{f}(\myvec{x}_0) = \myvec{f}(\myvec{z}_0)
    + \myvec{f}'(\myvec{z}_0)(\myvec{x}_0 - \myvec{z}_0)
    + \frac{1}{2}\myvec{f}''(\myvec{z}_0)(\myvec{x}_0 - \myvec{z}_0)(\myvec{x}_0 - \myvec{z}_0)
  $$
  and
  $$
  \myvec{0} = \myvec{f}(\myvec{x}^*) = \myvec{f}(\myvec{z}_0) + \myvec{f}'(\myvec{z}_0)(\myvec{x}^* - \myvec{z}_0)
    + \frac{1}{2}\myvec{f}''(\myvec{z}_0)(\myvec{x}^* - \myvec{z}_0)(\myvec{x}^* - \myvec{z}_0).
  $$
  Substituting these expansions into \eqref{eq:f'(z0)(x1-x*)} gives
  $$
  \myvec{f}'(\myvec{z}_0)(\myvec{x}_1 - \myvec{x}^*)
    = \frac{1}{2}[\myvec{f}''(\myvec{z}_0)(\myvec{x}^* - \myvec{z}_0)(\myvec{x}^* - \myvec{z}_0) - \myvec{f}''(\myvec{z}_0)(\myvec{x}_0 - \myvec{z}_0)(\myvec{x}_0 - \myvec{z}_0)].
  $$
  Note that 
  $\myvec{z}_0 - \myvec{x}_0 
    = (\myvec{y}_0 - \myvec{x}_0)/2 < (\myvec{x}^* - \myvec{x}_0)/2 < \myvec{x}^* - \myvec{z}_0$.
  It follows from the first inequality in \eqref{ineq:f''(x)h1h2h3} that 
  $$
  \myvec{f}'(\myvec{z}_0)(\myvec{x}_1 - \myvec{x}^*) < \myvec{0}.
  $$
  Then we have 
  $\myvec{x}_1 - \myvec{x}^* 
    = \myvec{f}'(\myvec{z}_0)^{-1}[\myvec{f}'(\myvec{z}_0)(\myvec{x}_1 - \myvec{x}^*)] < \myvec{0}$,
  which is equivalent to $\myvec{x}_1 < \myvec{x}^*$. 
  Hence we have that $\myvec{0} \leq \myvec{x}_0 < \myvec{z}_0 < \myvec{y}_0 < \myvec{x}_1 < \myvec{x}^*$.
  This completes the proof of the base case.
  
  Now, suppose that the statements (i) and (ii) are true for $k$th iteration. That is, we assume that 
  $$
  \myvec{0} \leq \myvec{x}_k < \myvec{z}_k < \myvec{y}_k < \myvec{x}_{k+1} < \myvec{x}^*
  $$
  holds for some $k \geq 0$. Let us consider iteration $k+1$. 
  By induction hypothesis, 
  it follows from Lemmas \ref{lem:f'invertibility} and \ref{lem:NonsingularMmatrix} that 
  $\myvec{f}'(\myvec{z}_k)$ is a nonsingular $M$-matrix and $\myvec{f}'(\myvec{z}_k)^{-1} \geq 0$.
  Then, we can apply \eqref{it:TSMNM} and Taylor's expansion \eqref{eq:TaylorExpansion} to get
  \begin{align*}
    \myvec{f}(\myvec{x}_{k+1})
    &= \myvec{f}(\myvec{z}_k) + \myvec{f}'(\myvec{z}_k)(\myvec{x}_{k+1} - \myvec{z}_k) 
      + \frac{1}{2}\myvec{f}''(\myvec{z}_k)(\myvec{x}_{k+1} - \myvec{z}_k)(\myvec{x}_{k+1} - \myvec{z}_k) \\
    &= \myvec{f}(\myvec{z}_k) + \myvec{f}'(\myvec{z}_k)(\myvec{x}_{k+1} - \myvec{x}_k)
      + \myvec{f}'(\myvec{z}_k)(\myvec{x}_{k} - \myvec{z}_k) \\
    &\quad + \frac{1}{2}\myvec{f}''(\myvec{z}_k)(\myvec{x}_{k+1} - \myvec{z}_k)(\myvec{x}_{k+1} - \myvec{z}_k) \\
    &= \myvec{f}(\myvec{z}_k) - \myvec{f}(\myvec{x}_k) + \myvec{f}'(\myvec{z}_k)(\myvec{x}_{k} - \myvec{z}_k)
      + \frac{1}{2}\myvec{f}''(\myvec{z}_k)(\myvec{x}_{k+1} - \myvec{z}_k)(\myvec{x}_{k+1} - \myvec{z}_k) \\
    &= \frac{1}{2}\left[\myvec{f}''(\myvec{z}_k)(\myvec{x}_{k+1} - \myvec{z}_k)(\myvec{x}_{k+1} - \myvec{z}_k)
      - \myvec{f}''(\myvec{z}_k)(\myvec{x}_{k} - \myvec{z}_k)(\myvec{x}_{k} - \myvec{z}_k)\right].
  \end{align*}
  Note that $\myvec{z}_k = (\myvec{x}_k + \myvec{y}_k)/2$. We have
  $$
  \myvec{z}_k - \myvec{x}_k < (\myvec{x}_{k+1} - \myvec{x}_k)/2 
    < [(\myvec{x}_{k+1}-\myvec{x}_k) + (\myvec{x}_{k+1} - \myvec{y}_k)]/2 
    = \myvec{x}_{k+1} - \myvec{z}_k.
  $$
  Combining this with the first inequality in \eqref{ineq:f''(x)h1h2h3}, we obtain
  $$
  \myvec{f}(\myvec{x}_{k+1})
    = \frac{1}{2}\left[\myvec{f}''(\myvec{z}_k)(\myvec{x}_{k+1} - \myvec{z}_k)(\myvec{x}_{k+1} - \myvec{z}_k)
    - \myvec{f}''(\myvec{z}_k)(\myvec{x}_{k} - \myvec{z}_k)(\myvec{x}_{k} - \myvec{z}_k)\right] < \myvec{0}.
  $$
  Hence $\myvec{f}'(\myvec{z}_k)(\myvec{y}_{k+1} - \myvec{x}_{k+1}) = -\myvec{f}(\myvec{x}_{k+1}) > \myvec{0}$.
  In view of $\myvec{f}'(\myvec{z}_k)^{-1} \geq 0$, it holds
  $$
  \myvec{y}_{k+1} - \myvec{x}_{k+1}
    = \myvec{f}'(\myvec{z}_k)^{-1}[\myvec{f}'(\myvec{z}_k)(\myvec{y}_{k+1} - \myvec{x}_{k+1})]
    > \myvec{0},
  $$
  This means that 
  $\myvec{x}_{k+1} < \myvec{z}_{k+1} = (\myvec{x}_{k+1} + \myvec{y}_{k+1})/2 < \myvec{y}_{k+1}$.
  Applying the inductive hypothesis, we have
  \begin{align}
    \myvec{f}'(\myvec{z}_k)(\myvec{y}_{k+1} - \myvec{x}^*)
      &= \myvec{f}'(\myvec{z}_k)(\myvec{y}_{k+1} - \myvec{x}_{k+1})
       + \myvec{f}'(\myvec{z}_k)(\myvec{x}_{k+1} - \myvec{z}_k)
       + \myvec{f}'(\myvec{z}_k)(\myvec{z}_{k} - \myvec{x}^*) \nonumber \\
      &= -\myvec{f}(\myvec{x}_{k+1}) + \myvec{f}'(\myvec{z}_k)(\myvec{x}_{k+1} - \myvec{z}_k)
       + \myvec{f}'(\myvec{z}_k)(\myvec{z}_{k} - \myvec{x}^*). \label{eq:f'(z_k)(y_k+1-x*)}
  \end{align}
  We use Taylor's expansion \eqref{eq:TaylorExpansion} to get
  $$
  \myvec{0} = \myvec{f}(\myvec{x}^*)
    = \myvec{f}(\myvec{z}_k) + \myvec{f}'(\myvec{z}_k)(\myvec{x}^* - \myvec{z}_k)
      + \frac{1}{2}\myvec{f}''(\myvec{z}_k)(\myvec{x}^* - \myvec{z}_k)(\myvec{x}^* - \myvec{z}_k)
  $$
  and
  $$
  \myvec{f}(\myvec{x}_{k+1})
    = \myvec{f}(\myvec{z}_k) + \myvec{f}'(\myvec{z}_k)(\myvec{x}_{k+1} - \myvec{z}_k)
      + \frac{1}{2}\myvec{f}''(\myvec{z}_k)(\myvec{x}_{k+1} - \myvec{z}_k)(\myvec{x}_{k+1} - \myvec{z}_k).
  $$
  Note that $\myvec{x}_{k+1} - \myvec{z}_k < \myvec{x}^* - \myvec{z}_k$.
  By substituting these expansions into \eqref{eq:f'(z_k)(y_k+1-x*)}, 
  we conclude from the first inequality in \eqref{ineq:f''(x)h1h2h3} that
  \begin{align*}
    \myvec{f}'(\myvec{z}_k)(\myvec{y}_{k+1} - \myvec{x}^*)
    &= \frac{1}{2}\left[\myvec{f}''(\myvec{z}_k)(\myvec{x}^* - \myvec{z}_k)(\myvec{x}^* - \myvec{z}_k)
      - \myvec{f}''(\myvec{z}_k)(\myvec{x}_{k+1} - \myvec{z}_k)(\myvec{x}_{k+1} - \myvec{z}_k)\right] \\
    & < \myvec{0}.
  \end{align*}
  This together with $\myvec{f}'(\myvec{z}_k)^{-1} \geq 0$ gives 
  $$
  \myvec{y}_{k+1} - \myvec{x}^* 
    = \myvec{f}'(\myvec{z}_k)^{-1}[\myvec{f}'(\myvec{z}_k)(\myvec{y}_{k+1} - \myvec{x}^*)]
    < \myvec{0},
  $$
  and so $\myvec{z}_{k+1} < \myvec{y}_{k+1} < \myvec{x}^*$.
  Thus, $\myvec{f}'(\myvec{z}_{k+1})$ is a nonsingular $M$-matrix and $\myvec{f}'(\myvec{z}_{k+1})^{-1} \geq 0$
  due to Lemmas \ref{lem:f'invertibility} and \ref{lem:NonsingularMmatrix}.
  Further, by Lemma \ref{lem:f'(x)-1<f'(y)-1}, we have
  $$
  \myvec{x}_{k+2} - \myvec{y}_{k+1}
    = [\myvec{f}'(\myvec{z}_{k})^{-1} - \myvec{f}'(\myvec{z}_{k+1})^{-1}]\myvec{f}(\myvec{x}_{k+1}) 
    > \myvec{0},
  $$
  which gives $\myvec{y}_{k+1} < \myvec{x}_{k+2}$.
  To complete the induction, it suffices to show that $\myvec{x}_{k+2} < \myvec{x}^*$.
  To this end, we first observe from Taylor's expansion \eqref{eq:TaylorExpansion} that
  $$
  \myvec{f}'(\myvec{z}_{k+1})(\myvec{x}_{k+1} - \myvec{z}_{k+1})
    = \myvec{f}(\myvec{x}_{k+1}) - \myvec{f}(\myvec{z}_{k+1})
    - \frac{1}{2}\myvec{f}''(\myvec{z}_{k+1})(\myvec{x}_{k+1} - \myvec{z}_{k+1})(\myvec{x}_{k+1} - \myvec{z}_{k+1})
  $$
  and
  $$
  \myvec{f}'(\myvec{z}_{k+1})(\myvec{z}_{k+1} - \myvec{x}^*)
    = \myvec{f}(\myvec{z}_{k+1}) 
    + \frac{1}{2}\myvec{f}''(\myvec{z}_{k+1})(\myvec{x}^* - \myvec{z}_{k+1})(\myvec{x}^* - \myvec{z}_{k+1}).
  $$
  By noting that
  $\myvec{x}_{k+2} - \myvec{x}^* 
    = (\myvec{x}_{k+2} - \myvec{x}_{k+1}) + (\myvec{x}_{k+1} - \myvec{z}_{k+1})
      + (\myvec{z}_{k+1} - \myvec{x}^*)$, it follows from \eqref{it:TSMNM} that
  \begin{align*}
    \lefteqn{\myvec{f}'(\myvec{z}_{k+1})(\myvec{x}_{k+2} - \myvec{x}^*)
    = -\myvec{f}(\myvec{x}_{k+1}) + \myvec{f}'(\myvec{z}_{k+1})(\myvec{x}_{k+1} - \myvec{z}_{k+1})
      + \myvec{f}'(\myvec{z}_{k+1})(\myvec{z}_{k+1} - \myvec{x}^*) } \\
    &= \frac{1}{2} \left[\myvec{f}''(\myvec{z}_{k+1})(\myvec{x}^* - \myvec{z}_{k+1})(\myvec{x}^* - \myvec{z}_{k+1})
      - \myvec{f}''(\myvec{z}_{k+1})(\myvec{x}_{k+1} - \myvec{z}_{k+1})(\myvec{x}_{k+1} - \myvec{z}_{k+1})\right].
  \end{align*}
  Recall that $\myvec{z}_{k+1} = (\myvec{x}_{k+1} + \myvec{y}_{k+1})/2$. We have
  $$
  \myvec{z}_{k+1} - \myvec{x}_{k+1}
    < (\myvec{x}^* - \myvec{x}_{k+1})/2
    < (\myvec{x}^* - \myvec{x}_{k+1} + \myvec{x}^* - \myvec{y}_{k+1})/2
    = \myvec{x}^* - \myvec{z}_{k+1}.
  $$
  Then the first inequality in \eqref{ineq:f''(x)h1h2h3} is applicable to obtain
  $$
  \myvec{f}'(\myvec{z}_{k+1})(\myvec{x}_{k+2} - \myvec{x}^*) < \myvec{0}.
  $$
  This implies that 
  $\myvec{x}_{k+2} - \myvec{x}^* 
    = \myvec{f}'(\myvec{z}_{k+1})^{-1}[\myvec{f}'(\myvec{z}_{k+1})(\myvec{x}_{k+2} - \myvec{x}^*)]
    < \myvec{0}$.
  Hence we arrive at 
  $\myvec{0} < \myvec{x}_{k+1} < \myvec{z}_{k+1} < \myvec{y}_{k+1} < \myvec{x}_{k+2} < \myvec{x}^*$.
  That is, the statements (i) and (ii) are true for the case $k+1$.
  Therefore, the statements (i) and (ii) hold for all $k \geq 0$ by induction.
  
  To prove the statement (iii), we first observe from statement (ii) that
  the positive sequences $\{\myvec{x}_k\}$ increases monotonically
  and is bounded above by $\myvec{x}^*$. 
  Then there exists a nonnegative vector $\myvec{x}^{**} \in \RS^{2n}$ such that 
  $\mylim_{k \to \infty} \myvec{x}_k = \myvec{x}^{**}$ and $\myvec{x}^{**} \leq \myvec{x}^*$.
  Letting $k \to \infty$ in \eqref{it:TSMNM}, 
  we know that $\myvec{x}^{**}$ is also a positive solution of the equation \eqref{eq:f(x)=0}
  and $\myvec{x}^* \leq \myvec{x}^{**}$.
  Consequently, we have $\myvec{x}^{**} = \myvec{x}^*$.
  It follows from the statement (ii) again that
  $\mylim_{k \to \infty} \myvec{y}_k = \mylim_{k \to \infty} \myvec{z}_k = \myvec{x}^{*}$.
  This completes the proof of the theorem.
\end{proof}

\begin{myremark}
  \label{remark:MonotoneConvergence}
  The condition $\myvec{f}(\myvec{x}_0) < \myvec{0}$ in Theorem \ref{th:MonotoneConvergence} 
  can be easily verified. For example, we can choose $\myvec{x}_0 = \myvec{0}$ or $\myvec{e} \in \RS^{2n}$.
\end{myremark}
  
Next, we consider the convergence rate of the two-step modified Newton method \eqref{it:TSMNM}.
It is well-known that the Newton-Kantorovich theorem \cite{KantorvichA1982} 
guarantees local quadratic convergence of Newton's method in Banach spaces, 
provided that the Jacobian $\myvec{f}'$ is Lipschitz continuous with constant $L_J$, 
and the quantity 
\begin{equation}
  \label{cons:beta}
  \beta := \|\myvec{f}'(\myvec{x}_0)^{-1}\myvec{f}(\myvec{x}_0)\|
\end{equation}
is small enough in the sense that $L_J\beta \leq 1/2$.
This result was extended by Potra \cite{Potra2017} to the two-step modified Newton method \eqref{it:TSMNM}, 
where convergence is guaranteed under the more restrictive condition $L_J\beta \leq 1/3$.
Inequalities of this form are generally referred to as the Kantorovich-type convergence criteria, 
as they determine whether the iterative methods will converge from a given initial guess.

Clearly, it holds from (ii) in Theorem \ref{th:MonotoneConvergence} that
$$
\|\myvec{x}^* - \myvec{y}_k\| \leq \|\myvec{x}^* - \myvec{x}_k\| \quad \text{for all $k \geq 0$}.
$$
Moreover, for the case when the Jacobian matrix $\myvec{f}'(\myvec{x}^*)$ is nonsingular, 
i.e., $\alpha \neq 0$ or $c \neq 1$,
we conclude from \eqref{eq:f''(x)h1h2} that the Jacobian of $\myvec{f}$ is Lipschitz continuous.
Specifically, we choose initial points $\myvec{x}_0 = \myvec{0} \in \RS^{2n}$. 
Then we have $\myvec{f}(\myvec{x}_0) = - \myvec{e}$ and $\myvec{f}'(\myvec{x}_0) = I_{2n}$.
Hence, the quantity $\beta$ defined in \eqref{cons:beta} becomes
$$
\beta = \|\myvec{f}'(\myvec{x}_0)^{-1}\myvec{f}(\myvec{x}_0)\|_{\infty}
  = \|\myvec{e}\|_{\infty} = 1.
$$
Besides, for any $\myvec{x}, \myvec{y} \in \RS^{2n}$, it follows from \eqref{eq:Jacobian_f} that
\begin{align*}
  \|\myvec{f}'(\myvec{x}_0)^{-1}[\myvec{f}'(\myvec{x}) - \myvec{f}'(y)]\|_{\infty}
  &= \|G(\myvec{x}) - G(\myvec{y})\|_{\infty} \\
  &\leq 2 \max_{1 \leq i \leq n}\left\{\sum_{j=1}^n p_{ij}, \sum_{j=1}^n \tilde{p}_{ij}\right\} 
    \cdot \|\myvec{x} - \myvec{y}\|_{\infty}.
\end{align*}
By \cite[Lemma 3]{Lu2005a}, Lu deduced that $\mysum_{j=1}^n p_{ij} < c(1-\alpha)/2$
and $\mysum_{j=1}^n \tilde{p}_{ij} < c(1+\alpha)/2$.
By making use of the above inequalities, we have
$$
\|\myvec{f}'(\myvec{x}_0)^{-1}[\myvec{f}'(\myvec{x}) - \myvec{f}'(y)]\|_{\infty}
  < c(1 + \alpha) \cdot \|\myvec{x} - \myvec{y}\|_{\infty}.
$$
This means that the Jacobian of $\myvec{f}$ is Lipschitz continuous 
with Lipschitz constant $L_J = c(1 + \alpha)$.
Consequently, the convergence criterion $L_J\beta \leq 1/3$ for the two-step modified Newton \eqref{it:TSMNM},
when applied to the nonlinear equation \eqref{eq:f(x)=0}, reduces to $c(1+\alpha) \leq 1/3$.
Therefore, Theorem \ref{th:MonotoneConvergence} and the convergence results in \cite[Theorem 2.7]{Potra2017}
are applicable to conclude the following corollary.
  
\begin{mycorollary}
  \label{cor:NonsingularQuadraticConvergence}
  Let $\myvec{x}^* \in \RS^{2n}$ be the minimal positive solution of the nonlinear system \eqref{eq:f(x)=0}
  such that the Jacobian matrix $\myvec{f}'(\myvec{x}^*)$ is nonsingular, i.e., $\alpha \neq 0$ or $c \neq 1$.
  If $c(1+\alpha) \leq 1/3$,
  then the iterative sequence $\{\myvec{x}_k\}$
  generated by the two-step modified Newton method \eqref{it:TSMNM}
  starting from the zero vector $\myvec{0} \in \RS^{2n}$ converges
  Q-quadratically to $\myvec{x}^*$, and the following error bound holds:
  $$
  \|\myvec{x}^* - \myvec{x}_k\|_{\infty} < 1.32c(1+\alpha) \|\myvec{x}_k - \myvec{x}_{k-1}\|^2_{\infty}, \quad k \geq 1.
  $$
  Moreover, the minimal positive solution $\myvec{x}^*$  belongs to the open ball $\ball(\myvec{0},r)$,
  where
  $$
  \frac{1 - \sqrt{1 - 2c(1+\alpha)}}{c(1+\alpha)} \leq r < \frac{1 + \sqrt{1 - 2c(1+\alpha)}}{c(1+\alpha)}.
  $$
\end{mycorollary}
  
\begin{myremark}
  Corollary \ref{cor:NonsingularQuadraticConvergence} implies that $0 < \|\myvec{x}^*\| \leq r$,
  which coincides with the one given in \cite[Theorem 4.1]{BaiGL2008}.
\end{myremark}

\subsection{The convergence rates at singular roots}

For the case when the Jacobian matrix $\myvec{f}'(\myvec{x}^*)$ is singular, 
i.e., $\alpha = 0$ and $c = 1$,
we will encounter new difficulties in investigating the convergence rates 
for the two-step modified Newton method \eqref{it:TSMNM}.
These difficulties primarily arise from the existence of a family of codimension-one manifolds 
through $\myvec{x}^*$ where $\myvec{f}'(\myvec{x})$ is singular. 
See \cite{Keller1981,DeckerK1985} for more details.
As a result, selecting initial guesses from a region surrounding $\myvec{x}^*$, 
where the invertibility of $\myvec{f}'(\myvec{x})$ is guaranteed, becomes essential. 
Moreover, we must demonstrate that subsequent iterates are well-defined, 
ensuring they remain within a region of invertibility.

Following the techniques used in much of the literature on singular problems 
(see, e.g., \cite{Reddien1978,DeckerK1980a,DeckerK1980b,DeckerKK1983,Kelley1986,KelleyX1993,Argyros1999,GuoL2000b,HuangKH2010,DallasP2023,Mannel2023}),
we let 
$$
\mathcal{N} = \ker(\myvec{f}'(\myvec{x}^*)) \quad \text{and} \quad \mathcal{R} = \range(\myvec{f}'(\myvec{x}^*)).
$$
Then $P_{\mathcal{N}}$ and $P_{\mathcal{R}}$ are the orthogonal projections onto 
$\mathcal{N}$ and $\mathcal{R}$, respectively.
It follows from \cite[Lemma 3.4]{HuangKH2010} that 
$\dim(\mathcal{N}) = 1, \RS^{2n} = \mathcal{N} \oplus \mathcal{R}$, 
$I = P_{\mathcal{N}} + P_{\mathcal{R}}$ and 
the restriction operator $\myvec{f}'(\myvec{x}^*)\big|_{\mathcal{R}}$ is invertiable on $\mathcal{R}$.
In addition, we define
\begin{equation}
  \label{set:K}
  \mathcal{K}(\omega) = \left\{k \in \NS \mid \|P_{\mathcal{N}}(\myvec{x}_k - \myvec{x}^*)\| 
    < \omega \|P_{\mathcal{R}}(\myvec{x}_k - \myvec{x}^*)\|\right\}
\end{equation}
and
\begin{equation}
  \label{set:W}
  \mathcal{W}(r,\theta) = \left\{\myvec{x} \in \RS^{2n} \mid \|\myvec{x} - \myvec{x}^*\| < r,
  \|P_{\mathcal{R}}(\myvec{x} - \myvec{x}^*)\| 
    \leq \theta \|P_{\mathcal{N}}(\myvec{x} - \myvec{x}^*)\|\right\}
\end{equation}
for $\omega, r, \theta > 0$ sufficiently small.

\begin{mytheorem}
  \label{th:SingularConvergenceSetK}
  Assume that $\myvec{f}'(\myvec{x}^*)$ is singular, i.e., $\alpha = 0$ and $c = 1$.
  Let $\{\myvec{x}_k\}, \{\myvec{y}_k\}$ and $\{\myvec{z}_k\}$ 
  be the sequences generated by the two-step modified Newton method \eqref{it:TSMNM}
  with an appropriate initial guess $\myvec{x}_0 \in \RS^{2n}$.
  If the index set $\mathcal{K}(\omega)$ defined by \eqref{set:K} is an infinite set for some $\omega > 0$,
  then the following error bound
  \begin{equation}
    \label{ineq:ConvergenceRateSetK}
    \|\myvec{x}^* - \myvec{x}_{k+1}\| \leq \eta \|\myvec{x}^* - \myvec{x}_k\|^2
  \end{equation}
  holds for all $k + 1 \in \mathcal{K}(\omega)$ large enough,
  where 
  $\eta = (1+\omega)\|(\myvec{f}'(\myvec{x}^*)\big|_{\mathcal{R}})^{-1}\|\|\myvec{f}''(\myvec{x}^*)\|$.
\end{mytheorem}

\begin{proof}
  For any $k \geq 0$, it follows from the Taylor expansion \eqref{eq:TaylorExpansion} that
  \begin{align*}
    \myvec{f}(\myvec{x}_{k+1})
    &= \myvec{f}(\myvec{x}^*) + \myvec{f}'(\myvec{x}^*)(\myvec{x}_{k+1} - \myvec{x}^*)
      + \frac{1}{2}\myvec{f}''(\myvec{x}^*)(\myvec{x}_{k+1} - \myvec{x}^*)(\myvec{x}_{k+1} - \myvec{x}^*) \\
    &= \myvec{f}'(\myvec{x}^*)(\myvec{x}_{k+1} - \myvec{x}^*) + 
    \frac{1}{2}\myvec{f}''(\myvec{x}^*)(\myvec{x}_{k+1} - \myvec{x}^*)(\myvec{x}_{k+1} - \myvec{x}^*).
  \end{align*}
  Recall that $\RS^{2n} = \mathcal{N} \oplus \mathcal{R}$ and $I = P_{\mathcal{N}} + P_{\mathcal{R}}$.
  We have $\myvec{f}'(\myvec{x}^*)[P_{\mathcal{N}}\myvec{x}] = \myvec{0}$ for any $\myvec{x} \in \RS^{2n}$, and so
  $$
  \myvec{f}'(\myvec{x}^*)(\myvec{x}_{k+1} - \myvec{x}^*)
    =\myvec{f}'(\myvec{x}^*)[P_{\mathcal{N}}(\myvec{x}_{k+1} - \myvec{x}^*) + P_{\mathcal{R}}(\myvec{x}_{k+1} - \myvec{x}^*)]
    = \myvec{f}'(\myvec{x}^*)\big|_{\mathcal{R}}[P_{\mathcal{R}}(\myvec{x}_{k+1} - \myvec{x}^*)].
  $$
  This leads to
  $$
  \myvec{f}(\myvec{x}_{k+1}) 
  = \myvec{f}'(\myvec{x}^*)\big|_{\mathcal{R}}[P_{\mathcal{R}}(\myvec{x}_{k+1} - \myvec{x}^*)]
    + \frac{1}{2}\myvec{f}''(\myvec{x}^*)(\myvec{x}_{k+1} - \myvec{x}^*)(\myvec{x}_{k+1} - \myvec{x}^*).
  $$
  Applying the reverse triangle inequality gives
  \begin{align}
    \|\myvec{f}(\myvec{x}_{k+1})\|
    &\geq \left\|\myvec{f}'(\myvec{x}^*)\big|_{\mathcal{R}}[P_{\mathcal{R}}(\myvec{x}_{k+1} - \myvec{x}^*)]\right\|
    - \frac{1}{2}\left\|\myvec{f}''(\myvec{x}^*)(\myvec{x}_{k+1} - \myvec{x}^*)(\myvec{x}_{k+1} - \myvec{x}^*)\right\| \nonumber \\
    &\geq \left\|\myvec{f}'(\myvec{x}^*)\big|_{\mathcal{R}}[P_{\mathcal{R}}(\myvec{x}_{k+1} - \myvec{x}^*)]\right\|
    - \frac{1}{2}\|\myvec{f}''(\myvec{x}^*)\|\|\myvec{x}_{k+1} - \myvec{x}^*\|^2. \label{ineq:norm_f(xk)}
  \end{align}
  Since $\myvec{f}'(\myvec{x}^*)\big|_{\mathcal{R}}$ is nonsingular on $\mathcal{R}$, 
  it follows from Lemma \ref{lem:MatrixLowerBound} that
  \begin{equation}
  \label{ineq:MatrixLowerBound}
  \left\|\left(\myvec{f}'(\myvec{x}^*)\big|_{\mathcal{R}}\right)^{-1}\right\|^{-1}
    \|P_{\mathcal{R}}(\myvec{x}_{k+1} - \myvec{x}^*)\|
    \leq \|\myvec{f}'(\myvec{x}^*)\big|_{\mathcal{R}}[P_{\mathcal{R}}(\myvec{x}_{k+1} - \myvec{x}^*)]\|.
  \end{equation}
  In addition, if $k+1 \in \mathcal{K}$, then
  \begin{align}
    \|\myvec{x}_{k+1} - \myvec{x}^*\|
    &\leq \|P_{\mathcal{R}}(\myvec{x}_{k+1} - \myvec{x}^*)\| 
      + \|P_{\mathcal{N}}(\myvec{x}_{k+1} - \myvec{x}^*)\| \nonumber \\
    &\leq  (1 + \omega)\|P_{\mathcal{R}}(\myvec{x}_{k+1} - \myvec{x}^*)\|. \label{ineq:norm_xk+1-x*}
  \end{align}
  By applying \eqref{ineq:MatrixLowerBound} and \eqref{ineq:norm_xk+1-x*} to \eqref{ineq:norm_f(xk)}, 
  we further derive that
  \begin{align}
    \|\myvec{f}(\myvec{x}_{k+1})\|
    &\geq \left\|\left(\myvec{f}'(\myvec{x}^*)\big|_{\mathcal{R}}\right)^{-1}\right\|^{-1}
        \|P_{\mathcal{R}}(\myvec{x}_{k+1} - \myvec{x}^*)\|
      - \frac{1}{2}\|\myvec{f}''(\myvec{x}^*)\|\|\myvec{x}_{k+1} - \myvec{x}^*\|^2 \nonumber \\
    &\geq (1+\omega)^{-1}\|\myvec{x}_{k+1} - \myvec{x}^*\| \left\|\left(\myvec{f}'(\myvec{x}^*)\big|_{\mathcal{R}}\right)^{-1}\right\|^{-1}
      - \frac{1}{2}\|\myvec{f}''(\myvec{x}^*)\|\|\myvec{x}_{k+1} - \myvec{x}^*\|^2 \nonumber \\
    &= \left(\left[(1+\omega)\left\|\left(\myvec{f}'(\myvec{x}^*)\big|_{\mathcal{R}}\right)^{-1}\right\|\right]^{-1}
      -\frac{1}{2}\|\myvec{f}''(\myvec{x}^*)\|\|\myvec{x}_{k+1} - \myvec{x}^*\|\right)\|\myvec{x}_{k+1} - \myvec{x}^*\|.
      \label{ineq:norm_f(xk+1)_lowerbound}
  \end{align}
  On the other hand, thanks to \eqref{it:TSMNM}, one has
  $$
    \myvec{f}(\myvec{x}_{k+1})  = \myvec{f}(\myvec{x}_{k+1}) 
      - \myvec{f}(\myvec{x}_k) - \myvec{f}'(\myvec{z}_k)(\myvec{x}_{k+1} - \myvec{x}_k).
  $$
  Then the Taylor expansion \eqref{eq:TaylorExpansion} is applicable again to get
  \begin{align*}
    \myvec{f}(\myvec{x}_{k+1}) 
    &= \myvec{f}'(\myvec{x}_k)(\myvec{x}_{k+1} - \myvec{x}_k)
      + \frac{1}{2}\myvec{f}''(\myvec{x}_k)(\myvec{x}_{k+1} - \myvec{x}_k)(\myvec{x}_{k+1} - \myvec{x}_k)
      - \myvec{f}'(\myvec{z}_k)(\myvec{x}_{k+1} - \myvec{x}_k) \\
    &= \myvec{f}''(\myvec{z}_k)(\myvec{x}_k - \myvec{z}_k)(\myvec{x}_{k+1} - \myvec{x}_k)
      + \frac{1}{2}\myvec{f}''(\myvec{x}_k)(\myvec{x}_{k+1} - \myvec{x}_k)(\myvec{x}_{k+1} - \myvec{x}_k) \\
    &= \frac{1}{2}\myvec{f}''(\myvec{z}_k)(\myvec{x}_k - \myvec{y}_k)(\myvec{x}_{k+1} - \myvec{x}_k)
      + \frac{1}{2}\myvec{f}''(\myvec{x}_k)(\myvec{x}_{k+1} - \myvec{x}_k)(\myvec{x}_{k+1} - \myvec{x}_k).
  \end{align*}
  Observe from Lemma \ref{lem:f''(x)hh} that $\myvec{f}''(\myvec{x})\myvec{h}\myvec{h}$ is independent of $\myvec{x}$ for any $\myvec{h} \in \RS^{2n}$.
  Then, by the third inequality in \eqref{ineq:f''(x)h1h2h3}, we further obtain that
  \begin{align*}
    \myvec{f}(\myvec{x}_{k+1}) 
    &= \frac{1}{2}\myvec{f}''(\myvec{x}^*)(\myvec{x}_{k+1} - \myvec{y}_k)(\myvec{x}_{k+1} - \myvec{x}_k) 
    > \frac{1}{2}\myvec{f}''(\myvec{x}^*)(\myvec{x}_{k+1} - \myvec{x}_k)(\myvec{x}_{k+1} - \myvec{x}_k).
  \end{align*}
  This allows us to deduce that
  \begin{align*}
    \myvec{0} < -\myvec{f}(\myvec{x}_{k+1})
    &< -\frac{1}{2}\myvec{f}''(\myvec{x}^*)(\myvec{x}_{k+1} - \myvec{x}_k)(\myvec{x}_{k+1} - \myvec{x}_k) \\
    &= -\frac{1}{2}\myvec{f}''(\myvec{x}^*)(\myvec{x}_{k+1} - \myvec{x}^*)(\myvec{x}_{k+1} - \myvec{x}^*) \\
      &\quad - \myvec{f}''(\myvec{x}^*)(\myvec{x}_{k+1} - \myvec{x}^*)(\myvec{x}^* - \myvec{x}_k) 
        - \frac{1}{2}\myvec{f}''(\myvec{x}^*)(\myvec{x}^* - \myvec{x}_k)(\myvec{x}^* - \myvec{x}_k).
  \end{align*}
  Then the second inequality in \eqref{ineq:f''(x)h1h2h3} implies that
  $$
  -\myvec{f}(\myvec{x}_{k+1})
    < -\frac{1}{2}\left[\myvec{f}''(\myvec{x}^*)(\myvec{x}_{k+1} - \myvec{x}^*)(\myvec{x}_{k+1} - \myvec{x}^*)
      + \myvec{f}''(\myvec{x}^*)(\myvec{x}^* - \myvec{x}_k)(\myvec{x}^* - \myvec{x}_k)\right],
  $$
  which yields
  \begin{equation}
    \label{ineq:norm_f(xk+1)_upperbound}
    \|\myvec{f}(\myvec{x}_{k+1})\|
    \leq \frac{1}{2}\|\myvec{f}''(\myvec{x}^*)\|
      \left(\|\myvec{x}_{k+1} - \myvec{x}^*\|^2 + \|\myvec{x}^* - \myvec{x}_k\|^2\right).
  \end{equation}
  From \eqref{ineq:norm_f(xk+1)_lowerbound} and \eqref{ineq:norm_f(xk+1)_upperbound}, 
  it follows that
  \begin{align*}
    &\left(\left[(1+\omega)\left\|\left(\myvec{f}'(\myvec{x}^*)\big|_{\mathcal{R}}\right)^{-1}\right\|\right]^{-1}
      -\frac{1}{2}\|\myvec{f}''(\myvec{x}^*)\|\|\myvec{x}_{k+1} - \myvec{x}^*\|\right)\|\myvec{x}_{k+1} - \myvec{x}^*\| \\
    &\qquad \leq \frac{1}{2}\|\myvec{f}''(\myvec{x}^*)\|
    \left(\|\myvec{x}_{k+1} - \myvec{x}^*\|^2 + \|\myvec{x}^* - \myvec{x}_k\|^2\right).
  \end{align*}
  Subtracting the first term on the right-hand side of the above inequality from both sides gives
  \begin{align}
    &\left(\left[(1+\omega)\left\|\left(\myvec{f}'(\myvec{x}^*)\big|_{\mathcal{R}}\right)^{-1}\right\|\right]^{-1}
    - \|\myvec{f}''(\myvec{x}^*)\|\|\myvec{x}_{k+1} - \myvec{x}^*\|\right)\|\myvec{x}_{k+1} - \myvec{x}^*\| \nonumber \\
    &\qquad\leq \frac{1}{2}\|\myvec{f}''(\myvec{x}^*)\|\|\myvec{x}^* - \myvec{x}_k\|^2.
    \label{ineq:norm_f''_x^-xk}
  \end{align}
  Recall that $\mylim_{k \to \infty} \myvec{x}_k = \myvec{x}^*$.
  We have for $k+1 \in \mathcal{K}$ large enough that
  $$
  \left[(1+\omega)\left\|\left(\myvec{f}'(\myvec{x}^*)\big|_{\mathcal{R}}\right)^{-1}\right\|\right]^{-1}
    - \|\myvec{f}''(\myvec{x}^*)\|\|\myvec{x}_{k+1} - \myvec{x}^*\|
    > \frac{1}{2}\left[(1+\omega)\left\|\left(\myvec{f}'(\myvec{x}^*)\big|_{\mathcal{R}}\right)^{-1}\right\|\right]^{-1}.
  $$
  This together with \eqref{ineq:norm_f''_x^-xk} permits us to arrive at the desired error bound \eqref{ineq:ConvergenceRateSetK}.
\end{proof}

\begin{myremark}
  Theorem \ref{th:SingularConvergenceSetK} 
  says that the two-step modified Newton method \eqref{it:TSMNM} is expected to exhibit
  the fast convergence behavior perpendicular to the null space directions.
\end{myremark}

Recall that $\{\myvec{x}_k\}$, $\{\myvec{y}_k\}$ and $\{\myvec{z}_k\}$ are the sequences generated 
by the two-step modified Newton method \eqref{it:TSMNM} with an appropriate initial guess $\myvec{x}_0 \in \RS^{2n}$.
To show the next theorem, we first need the following lemma, which is taken from \cite[Lemma 3.4]{HuangKH2010}.

\begin{mylemma}
  \label{lem:PNy}
  If $\myvec{f}'(\myvec{x}^*)$ is singular, 
  then there exists a nonsingular matrix $U \in \RS^{2n\times 2n}$ such that
  $$
  U^{-1} \myvec{f}'(\myvec{x}^*) U = 
  \begin{bmatrix}
    0 & \myvec{0} \\
    \myvec{0} & M_{22}
  \end{bmatrix},
  $$
  where $M_{22} \in \RS^{(2n-1)\times(2n-1)}$ is nonsingular.
  Moreover, 
  if denote by $\myvec{u}_1$ the first column of $U$ and by $\myvec{v}_1^{\top}$ the first row of $U^{-1}$, 
  then we have $\myvec{v}_1 > \myvec{0}$ or $\myvec{v}_1 < \myvec{0}$,
  and 
  \begin{equation}
    \label{eq:PNy}
    P_{\mathcal{N}}\myvec{y} = (\myvec{v}_1^{\top}\myvec{y})\myvec{u}_1 
    \quad \text{for any $\myvec{y} \in \RS^{2n}$}.
  \end{equation}
\end{mylemma}

\begin{mylemma}
  \label{lem:zk_in_W(r,theta)}
  Assume that $\myvec{f}'(\myvec{x}^*)$ is singular.
  Let $\mathcal{W}(r,\theta)$ be defined by \eqref{set:W} with any $r, \theta > 0$ sufficiently small.
  If $\myvec{x}_k, \myvec{y}_k \in \mathcal{W}(r,\theta)$ for some $k \geq 0$,
  then $\myvec{z}_k \in \mathcal{W}(r,\theta)$.
\end{mylemma}

\begin{proof}
  It is clear from \eqref{set:W} that $\myvec{x}_k, \myvec{y}_k \in \mathcal{W}(r,\theta)$ implies
  \begin{equation}
  \label{ineq:norm_zk-x*}
  \|\myvec{z}_k - \myvec{x}^*\| = \|(\myvec{x}_k - \myvec{x}^*) + (\myvec{y}_k - \myvec{x}^*)\|/2
  \leq \max\{\|\myvec{x}_k - \myvec{x}^*\|, \|\myvec{y}_k - \myvec{x}^*\|\} < r.
  \end{equation}
  On the other hand, we let $\myvec{u}_1$ and $\myvec{v}_1^{\top}$ be the first column of $U$ and the first row of $U^{-1}$, respectively,
  where $U \in \RS^{2n\times 2n}$ is the nonsingular matrix in Lemma \ref{lem:PNy}.
  It follows from \eqref{eq:PNy} that
  $$
  P_{\mathcal{N}}(\myvec{x}_k - \myvec{x}^*) = [\myvec{v}_1^{\top}(\myvec{x}_k - \myvec{x}^*)]\myvec{u}_1
  \quad \text{and} \quad
  P_{\mathcal{N}}(\myvec{y}_k - \myvec{x}^*) = [\myvec{v}_1^{\top}(\myvec{y}_k - \myvec{x}^*)]\myvec{u}_1.
  $$
  Then we have
  \begin{align*}
    \|P_{\mathcal{R}}(\myvec{x}_k - \myvec{x}^*)\| 
      &\leq \theta \|P_{\mathcal{N}}(\myvec{x}_k - \myvec{x}^*)\|
        = \theta \left|\myvec{v}_1^{\top}(\myvec{x}_k - \myvec{x}^*)\right|\|\myvec{u}_1\|, \\
    \|P_{\mathcal{R}}(\myvec{y}_k - \myvec{x}^*)\|
      &\leq \theta \|P_{\mathcal{N}}(\myvec{y}_k - \myvec{x}^*)\|
        = \theta \left|\myvec{v}_1^{\top}(\myvec{y}_k - \myvec{x}^*)\right|\|\myvec{u}_1\|.
  \end{align*}
  Since $\myvec{v}_1 > \myvec{0}$ or $\myvec{v}_1 < \myvec{0}$, one has that
  $\myvec{v}_1^{\top}(\myvec{x}_k - \myvec{x}^*)$ and
  $\myvec{v}_1^{\top}(\myvec{y}_k - \myvec{x}^*)$
  are both nonnegative or nonpositive. This implies that
  \begin{align*}
    \|P_{\mathcal{R}}(\myvec{x}_k - \myvec{x}^*)\| + \|P_{\mathcal{R}}(\myvec{y}_k - \myvec{x}^*)\|
    &\leq \theta \left|\myvec{v}_1^{\top}(\myvec{x}_k - \myvec{x}^*)\right|\|\myvec{u}_1\|
      + \theta \left|\myvec{v}_1^{\top}(\myvec{y}_k - \myvec{x}^*)\right|\|\myvec{u}_1\| \\
    &= \theta \left|\myvec{v}_1^{\top}(\myvec{x}_k - \myvec{x}^*) 
      + \myvec{v}_1^{\top}(\myvec{y}_k - \myvec{x}^*)\right|\|\myvec{u}_1\| \\
    &= \theta \|[\myvec{v}_1^{\top}(\myvec{x}_k - \myvec{x}^*)]\myvec{u}_1 
      + [\myvec{v}_1^{\top}(\myvec{y}_k - \myvec{x}^*)]\myvec{u}_1\| \\
    &= \theta \|P_{\mathcal{N}}(\myvec{x}_k - \myvec{x}^*) + P_{\mathcal{N}}(\myvec{y}_k - \myvec{x}^*)\|.
  \end{align*}
  Thus we conclude that
  \begin{align*}
    \|P_{\mathcal{R}}(\myvec{z}_k - \myvec{x}^*)\|
    &= \frac{1}{2}\|P_{\mathcal{R}}(\myvec{x}_k - \myvec{x}^*) + P_{\mathcal{R}}(\myvec{y}_k - \myvec{x}^*)\| \\
    &\leq \frac{1}{2}\left(\|P_{\mathcal{R}}(\myvec{x}_k - \myvec{x}^*)\| 
      + \|P_{\mathcal{R}}(\myvec{y}_k - \myvec{x}^*)\|\right) \\
    &\leq \frac{\theta}{2}\|P_{\mathcal{N}}(\myvec{x}_k - \myvec{x}^*) 
      + P_{\mathcal{N}}(\myvec{y}_k - \myvec{x}^*)\|
      = \theta \|P_{\mathcal{N}}(\myvec{z}_k - \myvec{x}^*)\|,
  \end{align*}
  which together with \eqref{ineq:norm_zk-x*} means that $\myvec{z}_k \in \mathcal{W}(r,\theta)$. 
  This completes the proof of the lemma.
\end{proof}

The below lemma taken from \cite[Lemma 3.6]{HuangKH2010} is also needed.

\begin{mylemma}
  Assume that $\myvec{f}'(\myvec{x}^*)$ is singular.
  For any $\myvec{x} \in \RS^{2n}$ satisfying $\myvec{0} < \myvec{x} < \myvec{x}^*$
  and $\|\myvec{x} - \myvec{x}^*\| \ll 1$, we have
  $\|P_{\mathcal{N}}\myvec{f}'(\myvec{x})^{-1}\| = \mathcal{O}(\|\myvec{x} - \myvec{x}^*\|^{-1})$
  and $\|P_{\mathcal{R}}\myvec{f}'(\myvec{x})^{-1}\| = \mathcal{O}(1)$.
\end{mylemma}

Due to the above lemmas, there exist constants $r_0, \mu_0, \nu_0 > 0$ such that
\begin{equation}
  \label{ineq:PNf'(x)-1}
  \|P_{\mathcal{N}}\myvec{f}'(\myvec{x})^{-1}\| \leq \mu_0\|\myvec{x} - \myvec{x}^*\|^{-1}
  \quad \text{and} \quad
  \|P_{\mathcal{R}}\myvec{f}'(\myvec{x})^{-1}\| \leq \nu_0
\end{equation}
for any $\myvec{x} \in \RS^{2n}$ 
satisfying $\myvec{0} < \myvec{x} < \myvec{x}^*$ and $\|\myvec{x} - \myvec{x}^*\| \leq r_0$.
Then we have the following lemma.

\begin{mylemma}
  \label{lem:yk+1_in_W(r,theta)}
  Assume that $\myvec{f}'(\myvec{x}^*)$ is singular.
  Let $\mathcal{W}(r,\theta)$ be defined by \eqref{set:W} 
  with $0 < \theta < 1/(\mu_0\|\myvec{f}''(\myvec{x}^*)\|)$ and 
  $$
  r = \min\{r_0, \theta(1 - \mu_0\theta\|\myvec{f}''(\myvec{x}^*)\|)/[2\nu_0(1+\theta)\|\myvec{f}''(\myvec{x}^*)\|]\},
  $$
  where the constants $r_0, \mu_0, \nu_0 > 0$ are defined in \eqref{ineq:PNf'(x)-1}.
  If $\myvec{x}_k \in \mathcal{W}(r,\theta)$ for some $k \geq 0$,
  then $\myvec{y}_k \in \mathcal{W}(r,\theta)$.
\end{mylemma}

\begin{proof}
  From (ii) in Theorem \ref{th:MonotoneConvergence}, 
  it is clear that $\|\myvec{y}_k - \myvec{x}^*\| < \|\myvec{x}_k - \myvec{x}^*\| < r$.
  To show $\myvec{y}_k \in \mathcal{W}(r,\theta)$, it suffices to examine that
  $\|P_{\mathcal{R}}(\myvec{y}_k - \myvec{x}^*)\| \leq \theta\|P_{\mathcal{N}}(\myvec{y}_k - \myvec{x}^*)\|$
  holds for any $k \in \NS$.
  By \eqref{it:TSMNM}, we have
  $$
  \myvec{y}_{k} - \myvec{x}^* = \myvec{x}_k - \myvec{x}^* - \myvec{f}'(\myvec{z}_{k-1})^{-1}\myvec{f}(\myvec{x}_k).
  $$
  For the case $k = 0$, recall that $\myvec{z}_{-1} = \myvec{x}_0$. 
  We get from the Taylor expansion \eqref{eq:TaylorExpansion} that
  \begin{align*}
    \myvec{y}_0 - \myvec{x}^*
      &= \myvec{f}'(\myvec{x}_0)^{-1}[\myvec{f}'(\myvec{x}_0)(\myvec{x}_0 - \myvec{x}^*) - \myvec{f}(\myvec{x}_0)] \\
      &= \frac{1}{2}\myvec{f}'(\myvec{x}_0)^{-1}\myvec{f}''(\myvec{x}_0)(\myvec{x}^* - \myvec{x}_0)(\myvec{x}^* - \myvec{x}_0).
  \end{align*}
  Notice from Lemma \ref{lem:f''(x)hh} that $\myvec{f}''(\myvec{x})\myvec{h}\myvec{h}$ is independent of $\myvec{x}$ for any $\myvec{h} \in \RS^{2n}$.
  It follows from \eqref{ineq:PNf'(x)-1} that
  \begin{equation}
    \label{ineq:PNy0}
    \|P_{\mathcal{R}}(\myvec{y}_0 - \myvec{x}^*)\|
      \leq \frac{1}{2}\nu_0\|\myvec{f}''(\myvec{x}^*)\|\|\myvec{x}^* - \myvec{x}_0\|^2.
  \end{equation}
  Since $\myvec{f}'(\myvec{x}_0) = \myvec{f}'(\myvec{x}^*) + \myvec{f}''(\myvec{x}^*)(\myvec{x}_0 - \myvec{x}^*)$
  and $\myvec{f}'(\myvec{x}^*)[P_{\mathcal{N}}(\myvec{x}_0 - \myvec{x}^*)] = \myvec{0}$,
  one has 
  \begin{align*}
    \myvec{y}_0 - \myvec{x}^*
    &= \frac{1}{2}\myvec{f}'(\myvec{x}_0)^{-1}\myvec{f}''(\myvec{x}^*)(\myvec{x}_0 - \myvec{x}^*)(\myvec{x}_0 - \myvec{x}^*) \\
    &= \frac{1}{2}\myvec{f}'(\myvec{x}_0)^{-1}[\myvec{f}'(\myvec{x}_0) - \myvec{f}'(\myvec{x}^*)]
      [P_{\mathcal{N}}(\myvec{x}_0 - \myvec{x}^*) + P_{\mathcal{R}}(\myvec{x}_0 - \myvec{x}^*)] \\
    &= \frac{1}{2}P_{\mathcal{N}}(\myvec{x}_0 - \myvec{x}^*) 
      + \frac{1}{2}\myvec{f}'(\myvec{x}_0)^{-1}\myvec{f}''(\myvec{x}^*)(\myvec{x}_0 - \myvec{x}^*)[P_{\mathcal{R}}(\myvec{x}_0 - \myvec{x}^*)].
  \end{align*}
  Then the reverse triangle inequality and \eqref{ineq:PNf'(x)-1} yield
  \begin{align*}
    \|P_{\mathcal{N}}(\myvec{y}_0 - \myvec{x}^*)\|
      &\geq \frac{1}{2}\|P_{\mathcal{N}}(\myvec{x}_0 - \myvec{x}^*)\|
        - \frac{1}{2}\mu_0\|\myvec{f}''(\myvec{x}^*)\|\|P_{\mathcal{R}}(\myvec{x}_0 - \myvec{x}^*)\| \\
      &\geq \frac{1}{2}\big(1 - \mu_0\theta\|\myvec{f}''(\myvec{x}^*)\|\big)\|P_{\mathcal{N}}(\myvec{x}_0 - \myvec{x}^*)\|.
  \end{align*}
  This together with \eqref{ineq:PNy0} gives
  \begin{align*}
    \frac{\|P_{\mathcal{R}}(\myvec{y}_0 - \myvec{x}^*)\|}{\|P_{\mathcal{N}}(\myvec{y}_0 - \myvec{x}^*)\|}
      &\leq \frac{\nu_0\|\myvec{f}''(\myvec{x}^*)\|\|\myvec{x}^* - \myvec{x}_0\|^2}
      {(1 - \mu_0\theta\|\myvec{f}''(\myvec{x}^*)\|)\|P_{\mathcal{N}}(\myvec{x}_0 - \myvec{x}^*)\|} \\
      &\leq \frac{r\nu_0(1+\theta)\|\myvec{f}''(\myvec{x}^*)\|}{1 - \mu_0\theta\|\myvec{f}''(\myvec{x}^*)\|}
      \leq \frac{2r\nu_0(1+\theta)\|\myvec{f}''(\myvec{x}^*)\|}{1 - \mu_0\theta\|\myvec{f}''(\myvec{x}^*)\|}
      \leq \theta,
  \end{align*}
  which means that $\myvec{y}_0 \in \mathcal{W}(r,\theta)$.
  For the case $k \geq 1$,
  we deduce again from the Taylor expansion \eqref{eq:TaylorExpansion} that
  \begin{align*}
    \lefteqn{\myvec{y}_{k} - \myvec{x}^*
      = \myvec{f}'(\myvec{z}_{k-1})^{-1}[\myvec{f}'(\myvec{z}_{k-1})(\myvec{x}_{k} - \myvec{x}^*) - \myvec{f}(\myvec{x}_{k})]} \\
      &= \myvec{f}'(\myvec{z}_{k-1})^{-1}\left[\big(\myvec{f}'(\myvec{z}_{k-1}) - \myvec{f}'(\myvec{x}_{k})\big)(\myvec{x}_{k} - \myvec{x}^*)
        + \frac{1}{2}\myvec{f}''(\myvec{x}_{k})(\myvec{x}^* - \myvec{x}_{k})(\myvec{x}^* - \myvec{x}_{k})\right] \\
      &= \frac{1}{2}\myvec{f}'(\myvec{z}_{k-1})^{-1}
      \big[2\myvec{f}''(\myvec{z}_{k-1})(\myvec{z}_{k-1} - \myvec{x}_{k}) 
        + \myvec{f}''(\myvec{x}_{k})(\myvec{x}_{k} - \myvec{x}^*)\big](\myvec{x}_{k} - \myvec{x}^*) \\
      &= \frac{1}{2}\myvec{f}'(\myvec{z}_{k-1})^{-1}\myvec{f}''(\myvec{x}^*)[(\myvec{z}_{k-1} - \myvec{x}_{k}) + (\myvec{z}_{k-1} - \myvec{x}^*)](\myvec{x}_{k} - \myvec{x}^*).
  \end{align*}
  It follows from (ii) in Theorem \ref{th:MonotoneConvergence} that
  $\myvec{0} < \myvec{x}_k - \myvec{z}_{k-1} < \myvec{x}^* - \myvec{z}_{k-1}$, 
  which leads to $\|\myvec{x}_k - \myvec{z}_{k-1}\| \leq \|\myvec{x}^* - \myvec{z}_{k-1}\|$.
  Thus, from \eqref{ineq:PNf'(x)-1}, we conclude
  \begin{align}
    \|P_{\mathcal{R}}(\myvec{y}_{k} - \myvec{x}^*)\|
    &\leq \frac{1}{2}\nu_0\|\myvec{f}''(\myvec{x}^*)\|
      \|(\myvec{z}_{k-1} - \myvec{x}_{k}) + (\myvec{z}_{k-1} - \myvec{x}^*)\|
      \|\myvec{x}_{k} - \myvec{x}^*\| \nonumber \\
    &\leq \nu_0\|\myvec{f}''(\myvec{x}^*)\|\|\myvec{z}_{k-1} - \myvec{x}^*\|\|\myvec{x}_{k} - \myvec{x}^*\|.
    \label{ineq:norm_PRyk+1-x*}
  \end{align}
  On the other hand, since
  $\myvec{f}'(\myvec{z}_{k-1}) = \myvec{f}'(\myvec{x}^*) + \myvec{f}''(\myvec{x}^*)(\myvec{z}_{k-1} - \myvec{x}^*)$,
  we have
  \begin{align*}
    \lefteqn{\myvec{f}''(\myvec{x}^*)[(\myvec{z}_{k-1} - \myvec{x}_{k}) + (\myvec{z}_{k-1} - \myvec{x}^*)](\myvec{x}_{k} - \myvec{x}^*)} \\
      &= \myvec{f}''(\myvec{x}^*)(\myvec{z}_{k-1} - \myvec{x}_{k})(\myvec{x}_{k} - \myvec{x}^*)
        + [\myvec{f}'(\myvec{z}_{k-1}) - \myvec{f}'(\myvec{x}^*)](\myvec{x}_{k} - \myvec{x}^*).
  \end{align*}
  In view of $\myvec{f}'(\myvec{x}^*)[P_{\mathcal{N}}(\myvec{x}_{k} - \myvec{x}^*)] = \myvec{0}$,
  we can further obtain
  \begin{align*}
    \myvec{y}_{k} - \myvec{x}^*
      &= \frac{1}{2}\myvec{f}'(\myvec{z}_{k-1})^{-1}\myvec{f}''(\myvec{x}^*)(\myvec{z}_{k-1} - \myvec{x}_{k})(\myvec{x}_{k} - \myvec{x}^*) \\
      &\quad + \frac{1}{2}\myvec{f}'(\myvec{z}_{k-1})^{-1}[\myvec{f}'(\myvec{z}_{k-1}) - \myvec{f}'(\myvec{x}^*)]
      [P_{\mathcal{N}}(\myvec{x}_{k} - \myvec{x}^*) + P_{\mathcal{R}}(\myvec{x}_{k} - \myvec{x}^*)] \\
      &= \frac{1}{2}\myvec{f}'(\myvec{z}_{k-1})^{-1}\myvec{f}''(\myvec{x}^*)(\myvec{z}_{k-1} - \myvec{x}_{k})(\myvec{x}_{k} - \myvec{x}^*) \\
      &\quad + \frac{1}{2}P_{\mathcal{N}}(\myvec{x}_{k} - \myvec{x}^*)
        + \frac{1}{2}\myvec{f}'(\myvec{z}_{k-1})^{-1}\myvec{f}''(\myvec{x}^*)(\myvec{z}_{k-1} - \myvec{x}^*)[P_{\mathcal{R}}(\myvec{x}_{k} - \myvec{x}^*)].
  \end{align*}
  Thanks to (i) in Theorem \ref{th:MonotoneConvergence} and the second inequality in \eqref{ineq:PNf'(x)-1}, we have
  $$
  \myvec{f}'(\myvec{z}_{k-1})^{-1}\myvec{f}''(\myvec{x}^*)(\myvec{z}_{k-1} - \myvec{x}_{k})(\myvec{x}^* - \myvec{x}_{k}) > \myvec{0}.
  $$
  This implies that
  $$
  \myvec{x}^* - \myvec{y}_{k}
  > \frac{1}{2}P_{\mathcal{N}}(\myvec{x}^* - \myvec{x}_{k})
  + \frac{1}{2}\myvec{f}'(\myvec{z}_{k-1})^{-1}\myvec{f}''(\myvec{x}^*)(\myvec{z}_{k-1} - \myvec{x}^*)[P_{\mathcal{R}}(\myvec{x}^* - \myvec{x}_{k})].
  $$
  By noting that $P_{\mathcal{N}}$ is idempotent,
  it follows from \eqref{ineq:PNf'(x)-1} that
  \begin{align*}
    \|P_{\mathcal{N}}(\myvec{x}^* - \myvec{y}_{k})\|
    &\geq \frac{1}{2}\|P_{\mathcal{N}}(\myvec{x}^* - \myvec{x}_{k})\| \\
    &\quad  - \frac{1}{2}\|P_{\mathcal{N}}\myvec{f}'(\myvec{z}_{k-1})^{-1}
      \myvec{f}''(\myvec{x}^*)(\myvec{z}_{k-1} - \myvec{x}^*)[P_{\mathcal{R}}(\myvec{x}^* - \myvec{x}_{k})]\| \\
    &\geq \frac{1}{2}\big(1 - \mu_0\theta\|\myvec{f}''(\myvec{x}^*)\|\big)\|P_{\mathcal{N}}(\myvec{x}^* - \myvec{x}_{k})\|.
  \end{align*}
  This together with \eqref{ineq:norm_PRyk+1-x*} and Lemma \ref{lem:zk_in_W(r,theta)} gives
  \begin{align*}
    \frac{\|P_{\mathcal{R}}(\myvec{y}_{k} - \myvec{x}^*)\|}{\|P_{\mathcal{N}}(\myvec{y}_{k} - \myvec{x}^*)\|}
    &\leq \frac{\nu_0\|\myvec{f}''(\myvec{x}^*)\|\|\myvec{z}_{k-1} - \myvec{x}^*\|\|\myvec{x}_{k} - \myvec{x}^*\|}
      {\frac{1}{2}\big(1 - \mu_0\theta\|\myvec{f}''(\myvec{x}^*)\|\big)\|P_{\mathcal{N}}(\myvec{x}^* - \myvec{x}_{k})\|} \\
    &\leq \frac{2\nu_0(1+\theta)\|\myvec{f}''(\myvec{x}^*)\|\|\myvec{z}_{k-1} - \myvec{x}^*\|}
      {1 - \mu_0\theta\|\myvec{f}''(\myvec{x}^*)\|} 
      \leq \frac{2r\nu_0(1+\theta)\|\myvec{f}''(\myvec{x}^*)\|}{1 - \mu_0\theta\|\myvec{f}''(\myvec{x}^*)\|} 
      \leq \theta,
  \end{align*}
  which means that $\myvec{y}_{k} \in \mathcal{W}(r,\theta)$.
\end{proof}

\begin{mylemma}
  \label{lem:xk+1_in_W(r,theta)}
  Assume that $\myvec{f}'(\myvec{x}^*)$ is singular.
  Let $\mathcal{W}(r,\theta)$ be defined by \eqref{set:W} 
  with $0 < \theta < 1/(\mu_0\|\myvec{f}''(\myvec{x}^*)\|)$ and 
  $$
  r = \min\{r_0, \theta(1 - \mu_0\theta\|\myvec{f}''(\myvec{x}^*)\|)/[\nu_0(1+\theta)\|\myvec{f}''(\myvec{x}^*)\|]\},
  $$
  where $r_0, \mu_0, \nu_0 > 0$ are defined in \eqref{ineq:PNf'(x)-1}.
  If $\myvec{x}_k, \myvec{y}_k \in \mathcal{W}(r,\theta)$ for some $k \geq 0$,
  then $\myvec{x}_{k+1} \in \mathcal{W}(r,\theta)$.
\end{mylemma}

\begin{proof}
  It is clear from (ii) in Theorem \ref{th:MonotoneConvergence} that 
  $\|\myvec{x}_{k+1} - \myvec{x}^*\| < \|\myvec{x}_k - \myvec{x}^*\| \leq r$.
  We observe from \eqref{it:TSMNM} that
  \begin{equation}
    \label{eq:xk+1-x*}
    \myvec{x}_{k+1} - \myvec{x}^* 
    = \myvec{x}_k - \myvec{x}^* - \myvec{f}'(\myvec{z}_k)^{-1}\myvec{f}(\myvec{x}_k)
    = \myvec{f}'(\myvec{z}_k)^{-1}[\myvec{f}'(\myvec{z}_k)(\myvec{x}_k - \myvec{x}^*) - \myvec{f}(\myvec{x}_k)].
  \end{equation}
  By the Taylor expansion \eqref{eq:TaylorExpansion}, we have
  $$
  \myvec{0} = \myvec{f}(\myvec{x}^*) = \myvec{f}(\myvec{x}_k)
    + \myvec{f}'(\myvec{x}_k)(\myvec{x}^* - \myvec{x}_k)
    + \frac{1}{2}\myvec{f}''(\myvec{x}_k)(\myvec{x}^* - \myvec{x}_k)(\myvec{x}^* - \myvec{x}_k),
  $$
  which gives
  $$
  \myvec{f}(\myvec{x}_k) = - \myvec{f}'(\myvec{x}_k)(\myvec{x}^* - \myvec{x}_k)
    - \frac{1}{2}\myvec{f}''(\myvec{x}_k)(\myvec{x}^* - \myvec{x}_k)(\myvec{x}^* - \myvec{x}_k).
  $$
  Then we get that
  $$
  \myvec{f}'(\myvec{z}_k) - \myvec{f}'(\myvec{x}_k)
    = [\myvec{f}'(\myvec{x}_k) + \myvec{f}''(\myvec{x}_k)(\myvec{z}_k - \myvec{x}_k)]
      - \myvec{f}'(\myvec{x}_k)
    = \myvec{f}''(\myvec{x}_k)(\myvec{z}_k - \myvec{x}_k).
  $$
  Combining the above equation with \eqref{eq:xk+1-x*} yields
  \begin{align}
    \myvec{x}_{k+1} - \myvec{x}^* &= \myvec{f}'(\myvec{z}_k)^{-1}
    \left[\myvec{f}''(\myvec{x}_k)(\myvec{z}_k - \myvec{x}_k)(\myvec{x}_k - \myvec{x}^*) 
      + \frac{1}{2}\myvec{f}''(\myvec{x}_k)(\myvec{x}^* - \myvec{x}_k)(\myvec{x}^* - \myvec{x}_k)\right] \nonumber \\
    &= \frac{1}{2}\myvec{f}'(\myvec{z}_k)^{-1}\myvec{f}''(\myvec{x}_k)[2(\myvec{z}_k - \myvec{x}_k) + (\myvec{x}_k - \myvec{x}^*)](\myvec{x}_k - \myvec{x}^*) \nonumber \\
    &= \frac{1}{2}\myvec{f}'(\myvec{z}_k)^{-1}\myvec{f}''(\myvec{x}_k)(\myvec{y}_k - \myvec{x}^*)(\myvec{x}_k - \myvec{x}^*). \label{eq:xk+1-x*2}
  \end{align}
  Notice from Lemma \ref{lem:f''(x)hh} that $\myvec{f}''(\myvec{x})\myvec{h}\myvec{h}$ is independent of $\myvec{x}$ for any $\myvec{h} \in \RS^{2n}$.
  This implies
  $$
  P_{\mathcal{R}}(\myvec{x}_{k+1} - \myvec{x}^*)
    = \frac{1}{2}P_{\mathcal{R}}[\myvec{f}'(\myvec{z}_k)^{-1}\myvec{f}''(\myvec{x}^*)(\myvec{y}_k - \myvec{x}^*)(\myvec{x}_k - \myvec{x}^*)].
  $$
  Then \eqref{ineq:PNf'(x)-1} is applicable to obtain
  \begin{equation}
    \label{ineq:norm_PRxk+1-x*}
    \|P_{\mathcal{R}}(\myvec{x}_{k+1} - \myvec{x}^*)\| 
    \leq \frac{1}{2}\nu_0\|\myvec{f}''(\myvec{x}^*)\|\|\myvec{y}_k - \myvec{x}^*\|\|\myvec{x}_k - \myvec{x}^*\|.
  \end{equation}
  On the other hand, since $\myvec{0} < \myvec{x}_k < \myvec{z}_k < \myvec{x}^*$, 
  it follows from Lemma \ref{lem:f'(x)-1<f'(y)-1} that 
  $$
  0 < \myvec{f}'(\myvec{x}_k)^{-1} < \myvec{f}'(\myvec{z}_k)^{-1}.
  $$
  Then we can further deduce from \eqref{eq:xk+1-x*2} that
  $$
  \myvec{x}_{k+1} - \myvec{x}^* 
    < \frac{1}{2} \myvec{f}'(\myvec{x}_k)^{-1}\myvec{f}''(\myvec{x}^*)(\myvec{x}^* - \myvec{y}_k)(\myvec{x}^* - \myvec{x}_k).
  $$
  Since $\myvec{f}'(\myvec{x}_k) = \myvec{f}'(\myvec{x}^*) + \myvec{f}''(\myvec{x}^*)(\myvec{x}_k - \myvec{x}^*)$ 
  and $\myvec{f}'(\myvec{x}^*)[P_{\mathcal{N}}(\myvec{x}^* - \myvec{y}_k)] = \myvec{0}$,
  it follows that
  \begin{align*}
    \myvec{x}^* - \myvec{x}_{k+1}
      &> \frac{1}{2} \myvec{f}'(\myvec{x}_k)^{-1}[\myvec{f}'(\myvec{x}_k) - \myvec{f}'(\myvec{x}^*)](\myvec{x}^* - \myvec{y}_k) \\
      &= \frac{1}{2} \myvec{f}'(\myvec{x}_k)^{-1}[\myvec{f}'(\myvec{x}_k) - \myvec{f}'(\myvec{x}^*)]
        [P_{\mathcal{N}}(\myvec{x}^* - \myvec{y}_k) + P_{\mathcal{R}}(\myvec{x}^* - \myvec{y}_k)] \\
      &= \frac{1}{2}P_{\mathcal{N}}(\myvec{x}^* - \myvec{y}_k)
        + \frac{1}{2} \myvec{f}'(\myvec{x}_k)^{-1}\myvec{f}''(\myvec{x}^*)(\myvec{x}_k - \myvec{x}^*)[P_{\mathcal{R}}(\myvec{x}^* - \myvec{y}_k)].
  \end{align*}
  By noting that $P_{\mathcal{N}}$ is idempotent, 
  we use the reverse triangle inequality to get
  \begin{align*}
    \|P_{\mathcal{N}}(\myvec{x}^* - \myvec{x}_{k+1})\| 
    &\geq \frac{1}{2}\|P_{\mathcal{N}}(\myvec{x}^* - \myvec{y}_k)\|
      - \frac{1}{2}\|P_{\mathcal{N}}\myvec{f}'(\myvec{x}_k)^{-1}\myvec{f}''(\myvec{x}^*)(\myvec{x}_k - \myvec{x}^*)[P_{\mathcal{R}}(\myvec{x}^* - \myvec{y}_k)]\| \\
    &\geq \frac{1}{2}\|P_{\mathcal{N}}(\myvec{x}^* - \myvec{y}_k)\|
      - \frac{1}{2}\mu_0\|\myvec{f}''(\myvec{x}^*)\|\|P_{\mathcal{R}}(\myvec{x}^* - \myvec{y}_k)\| \\
    &\geq \frac{1}{2}\big(1 - \mu_0\theta\|\myvec{f}''(\myvec{x}^*)\|\big)\|P_{\mathcal{N}}(\myvec{x}^* - \myvec{y}_k)\|.
  \end{align*}
  This together with \eqref{ineq:norm_PRxk+1-x*} implies that
  \begin{align*}
    \frac{\|P_{\mathcal{R}}(\myvec{x}_{k+1} - \myvec{x}^*)\|}{\|P_{\mathcal{N}}(\myvec{x}_{k+1} - \myvec{x}^*)\|}
    &\leq \frac{\nu_0\|\myvec{f}''(\myvec{x}^*)\|\|\myvec{y}_k - \myvec{x}^*\|\|\myvec{x}_k - \myvec{x}^*\|}
      {\big(1 - \mu_0\theta\|\myvec{f}''(\myvec{x}^*)\|\big)\|P_{\mathcal{N}}(\myvec{x}^* - \myvec{y}_k)\|} \\
    &\leq \frac{\nu_0(1+\theta)\|\myvec{f}''(\myvec{x}^*)\|\|\myvec{x}_k - \myvec{x}^*\|}
      {1 - \mu_0\theta\|\myvec{f}''(\myvec{x}^*)\|}
    \leq \frac{\nu_0r(1+\theta)\|\myvec{f}''(\myvec{x}^*)\|}
    {1 - \mu_0\theta\|\myvec{f}''(\myvec{x}^*)\|} \leq \theta,
  \end{align*}
  which means that $\myvec{x}_{k+1} \in \mathcal{W}(r,\theta)$.
\end{proof}

\begin{mytheorem}
  \label{th:SingularConvergenceSetW}
  Assume that $\myvec{f}'(\myvec{x}^*)$ is singular, i.e., $\alpha = 0$ and $c = 1$.
  Let $\{\myvec{x}_k\}, \{\myvec{y}_k\}$ and $\{\myvec{z}_k\}$ 
  be the sequences generated by the two-step modified Newton method \eqref{it:TSMNM}
  starting from an initial guess $\myvec{x}_0 \in \mathcal{W}(r,\theta)$ 
  with $0 < \theta < \min\{1/(\mu_0\|\myvec{f}''(\myvec{x}^*)\|),1\}$ and 
  $$
  r = \min\{r_0, \theta(1 - \mu_0\theta\|\myvec{f}''(\myvec{x}^*)\|)/[2\nu_0(1+\theta)\|\myvec{f}''(\myvec{x}^*)\|]\},
  $$
  where constants $r_0, \mu_0, \nu_0 > 0$ are defined in \eqref{ineq:PNf'(x)-1} below.
  Then the sequences $\{\myvec{x}_k\}, \{\myvec{y}_k\}$ and $\{\myvec{z}_k\}$ remain in $\mathcal{W}(r,\theta)$,
  and the following error bound
  \begin{equation}
    \label{ineq:ConvergenceRateSetW}
    \|\myvec{x}^* - \myvec{x}_{k+1}\| 
      \leq \frac{(1+\theta)(1+\mu_0\theta\|\myvec{f}''(\myvec{x}^*)\|)}{2(1-\theta)} \|\myvec{x}^* - \myvec{x}_k\|
  \end{equation}
  holds for all $k \geq 0$.
\end{mytheorem}

\begin{proof}
  It follows from Lemmas \ref{lem:zk_in_W(r,theta)}, \ref{lem:yk+1_in_W(r,theta)}
  and \ref{lem:xk+1_in_W(r,theta)} that the sequences $\{\myvec{x}_k\}, \{\myvec{y}_k\}$
  and $\{\myvec{z}_k\}$ are all contained in $\mathcal{W}(r,\theta)$.
  It remains to show the error bound \eqref{ineq:ConvergenceRateSetW}.
  By the third inequality in \eqref{ineq:f''(x)h1h2h3}, we obtain from \eqref{eq:xk+1-x*2} that
  \begin{align*}
    \myvec{x}^* - \myvec{x}_{k+1}
      &= \frac{1}{2}\myvec{f}'(\myvec{z}_k)^{-1}\myvec{f}''(\myvec{x}^*)(\myvec{y}_k - \myvec{x}^*)(\myvec{x}^* - \myvec{x}_k) \\
      &< \frac{1}{2}\myvec{f}'(\myvec{z}_k)^{-1}\myvec{f}''(\myvec{x}^*)(\myvec{z}_k - \myvec{x}^*)(\myvec{x}^* - \myvec{x}_k) \\
      &= \frac{1}{2}\myvec{f}'(\myvec{z}_k)^{-1}[\myvec{f}'(\myvec{z}_k) - \myvec{f}'(\myvec{x}^*)]
      [P_{\mathcal{N}}(\myvec{x}^* - \myvec{x}_k) + P_{\mathcal{R}}(\myvec{x}^* - \myvec{x}_k)]  \\
      &= \frac{1}{2}P_{\mathcal{N}}(\myvec{x}^* - \myvec{x}_k)
        + \frac{1}{2}\myvec{f}'(\myvec{z}_k)^{-1}\myvec{f}''(\myvec{x}^*)(\myvec{z}_k - \myvec{x}^*)[P_{\mathcal{R}}(\myvec{x}^* - \myvec{x}_k)].
  \end{align*}
  By noting that $P_{\mathcal{N}}$ is idempotent, it follows from \eqref{ineq:PNf'(x)-1} that
  $$
  \|P_{\mathcal{N}}(\myvec{x}^* - \myvec{x}_{k+1})\|
    \leq \frac{1}{2}\big(1 + \mu_0\theta\|\myvec{f}''(\myvec{x}^*)\|\big)\|P_{\mathcal{N}}(\myvec{x}^* - \myvec{x}_k)\|.
  $$
  Since
  $$
  \|\myvec{x}^* - \myvec{x}_k\|
    \geq \|P_{\mathcal{N}}(\myvec{x}^* - \myvec{x}_k)\| - \|P_{\mathcal{R}}(\myvec{x}^* - \myvec{x}_k)\|
    \geq (1 - \theta)\|P_{\mathcal{N}}(\myvec{x}^* - \myvec{x}_k)\|,
  $$
  we infer that 
  \begin{align*}
    \frac{\|\myvec{x}^* - \myvec{x}_{k+1}\|}{\|\myvec{x}^* - \myvec{x}_k\|}
    &\leq \frac{(1 + \theta)\|P_{\mathcal{N}}(\myvec{x}^* - \myvec{x}_{k+1})\|}
      {(1 - \theta)\|P_{\mathcal{N}}(\myvec{x}^* - \myvec{x}_k)\|}
    \leq \frac{(1+\theta)(1+\mu_0\theta\|\myvec{f}''(\myvec{x}^*)\|)}{2(1-\theta)},
  \end{align*}
  which yields the desired error bound \eqref{ineq:ConvergenceRateSetW}.
\end{proof}

\begin{myremark}
  Theorem \ref{th:SingularConvergenceSetW} 
  says that the two-step modified Newton method \eqref{it:TSMNM} is expected to exhibit
  the slow convergence behavior parallel to the null space directions.
\end{myremark}

\begin{myremark}
  We point out that much of the recent work on the theory of singular equations 
  has focused on weakening the assumptions imposed on the singularities, 
  often by employing 2-regularity \cite{OberlinW2009,IzmailovKS2018a,IzmailovKS2018b,FischerIS2021}. 
  The convergence of the two-step modified Newton method under weaker assumptions, 
  such as 2-regularity, remains a topic for future research. 
\end{myremark}

\section{Numerical experiments}
\label{sec:NumericalExperiments}

In this section, we provide some numerical examples to illustrate the effectiveness of the proposed algorithm.
Below are the algorithms being tested, with abbreviations corresponding to the table columns and figure captions.

\begin{itemize}
  \item TSMNM (for the two-step modified Newton method) is our implementation of Algorithm \ref{alg:TSMNM}.
  \item NM (for the Newton method) is the algorithm from \cite{Lu2005b} 
  by using the standard Newton method.
  \item TSNM1 is the algorithm from \cite{LinBW2008} by using a two-step Newton method.
  \item TSNM2 is the algorithm from \cite{LingLL2022} by using another two-step Newton method.
  \item NSM$(m)$ is the algorithm from \cite{LinB2008} by using the 
  Newton-Shamanskii method, which is a family of iterative methods derived from Newton's method
  that converge with order $m + 1$. For $m=1$, it reduces to the Newton method;
  for $m=2$, it reduces to the two-step Newton method studied in \cite{LingLL2022}.
  \item FPI is the algorithm from \cite{Lu2005a} by using the fixed-point iteration.
  \item NBJ is the nonlinear block Jacobi iteration algorithm proposed in \cite{BaiGL2008}.
  \item NBGS is the nonlinear block Gauss-Seidel iteration algorithm presented in \cite{BaiGL2008}.
\end{itemize}
As noted in Remark \ref{remark:MonotoneConvergence}, 
the condition $\myvec{f}(\myvec{x}_0) < \myvec{0}$, 
which guarantees the monotone convergence of the two-step modified Newton method 
as stated in Theorem \ref{th:MonotoneConvergence}, 
is satisfied when the initial point is chosen as $\myvec{x}_0 = \myvec{0}$.
This initial point is used for all the algorithms considered above.
As Example 5.2 in \cite{GuoL2000b}, the constants $c_i$ and $\omega_i$ are determined 
using a numerical quadrature formula on the interval $[0,1]$.
This involves dividing the interval into $n/4$ subintervals of equal length
and employing Gauss-Legendre quadrature with four nodes on each subinterval.
All algorithms were implemented and executed in 64-bit version of MATLAB R2019b 
on a laptop equipped with Intel(R) Core(TM) i7-8550U 1.80GHz CPU and 16 GB memory.
In light of the convergence results in Corollary \ref{cor:NonsingularQuadraticConvergence},
we use the stopping criterion in our implementations:
$$
  \text{RES} := \max\left\{\frac{\|\myvec{u}_{k+1} - \myvec{u}_k\|_{\infty}}{\|\myvec{u}_{k+1}\|_{\infty}},
  \frac{\|\myvec{v}_{k+1} - \myvec{v}_k\|_{\infty}}{\|\myvec{v}_{k+1}\|_{\infty}}\right\}
  \leq n\cdot\texttt{eps},
$$
where $n$ is the order of matrix $A$ given in \eqref{mat:ABCD} 
and $\texttt{eps} = 2^{-52} \approx 2.2204 \times 10^{-16}$ is the machine epsilon.
In our implementation of Newton-type methods,
we utilize MATLAB's \texttt{lu} function to factorize the coefficient matrices,
and solve the resulting triangular systems using \texttt{linsolve} with options specified by \texttt{opts}.
The CPU time (in seconds) is computed by using MATLAB's \texttt{tic/toc} commands.
Each numerical experiment is repeated 10 times and the results are averaged to produce the time displayed in the tables and figures.
Moreover, we use ``IT" to denote the number of iterations.

To begin, we perform an experiment aimed at verifying the results of Corollary \ref{cor:NonsingularQuadraticConvergence}.
Seven distinct pairs of $(\alpha,c)$ are chosen to examine their convergence behavior. 
Figure \ref{fig:IterHistory_TSMNM} presents the iteration histories of TSMNM
for problem sizes $n = 1024, 2048$ and $4096$. 
It can be observed that faster convergence is achieved as the quantity $c(1+\alpha)$ becomes smaller.
It is worth noting that the convergence criterion $c(1+\alpha) \leq 1/3$ 
in Corollary \ref{cor:NonsingularQuadraticConvergence} is a sufficient but not necessary condition 
to guarantee at least quadratic convergence of TSMNM. 
As shown in Figure \ref{fig:IterHistory_TSMNM}, 
when $(\alpha,c) = (1/2,1/3)$, which yields $c(1+\alpha) = 1/2$, 
TSMNM exhibits a convergence rate comparable to the case $(1/4,4/15)$, 
where the criterion is satisfied exactly. 
This observation suggests that the convergence criterion may be further weakened, 
though refining it remains analytically challenging due to the complexity of the convergence analysis.

\begin{figure}
  \centering
  \subfigure{\includegraphics[width=0.42\textwidth]{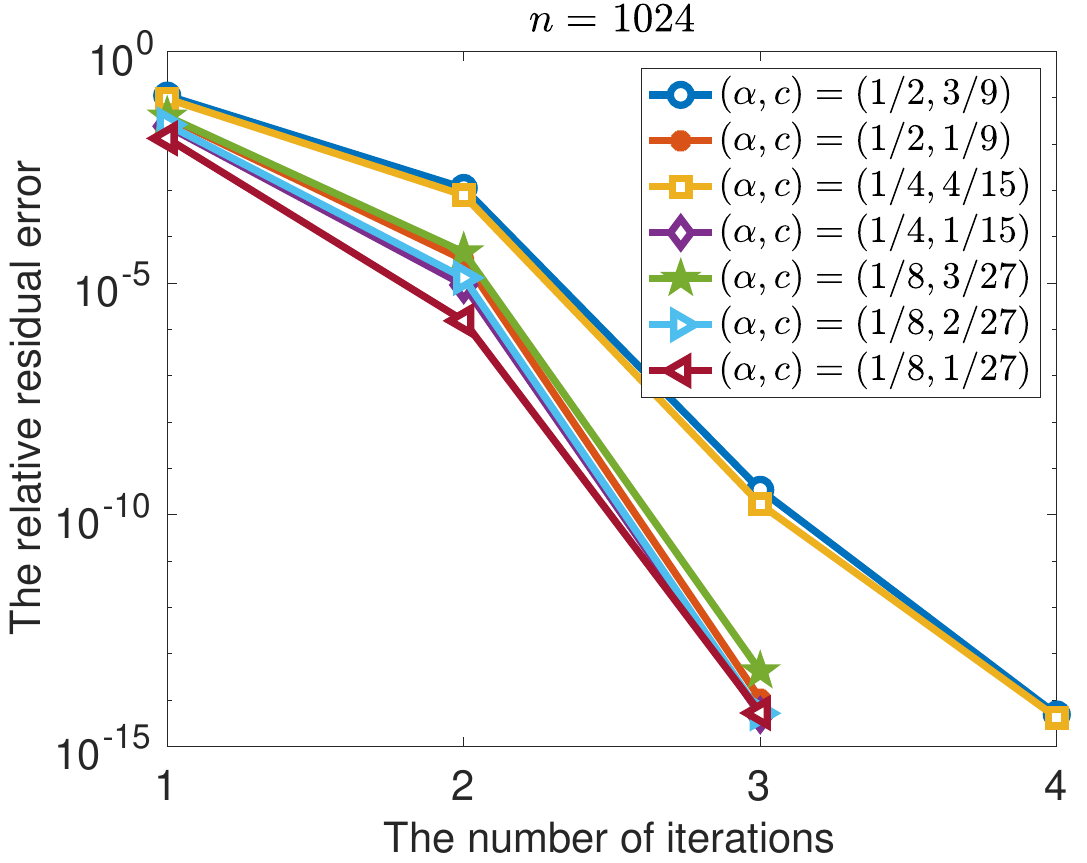}}\quad
  \subfigure{\includegraphics[width=0.42\textwidth]{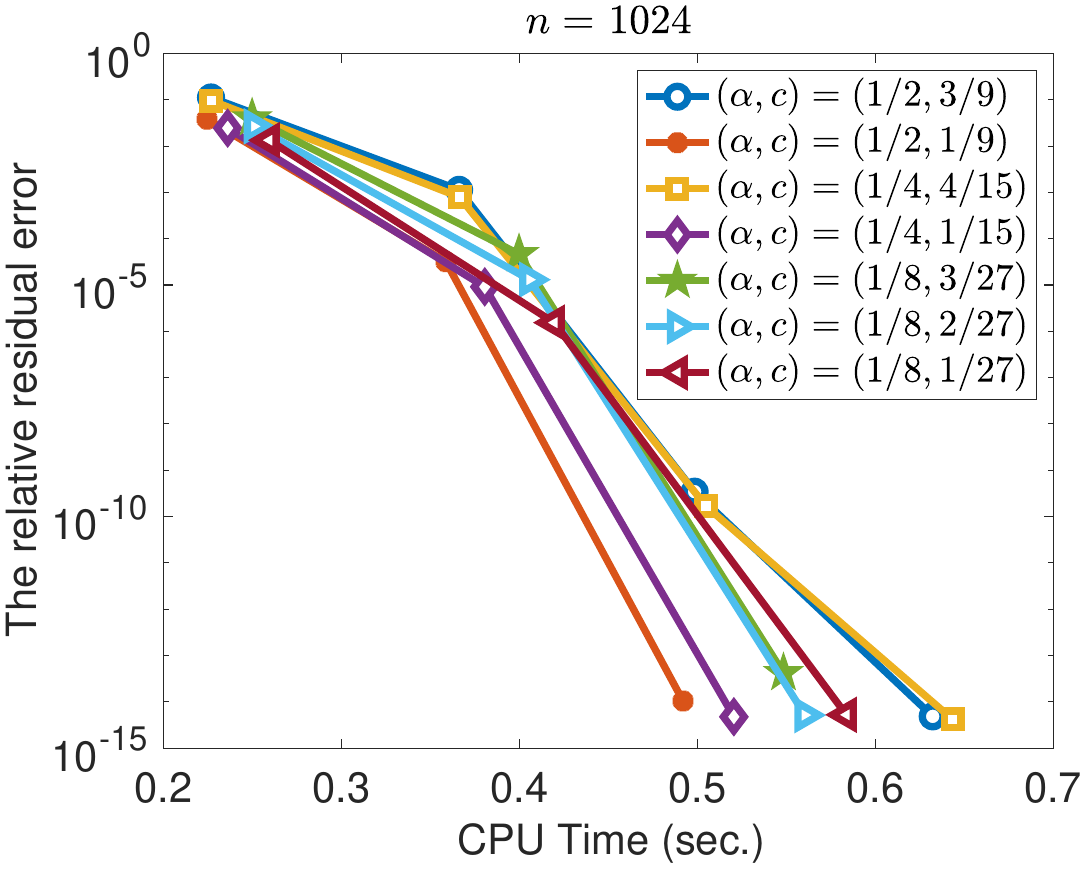}} \\
  \subfigure{\includegraphics[width=0.42\textwidth]{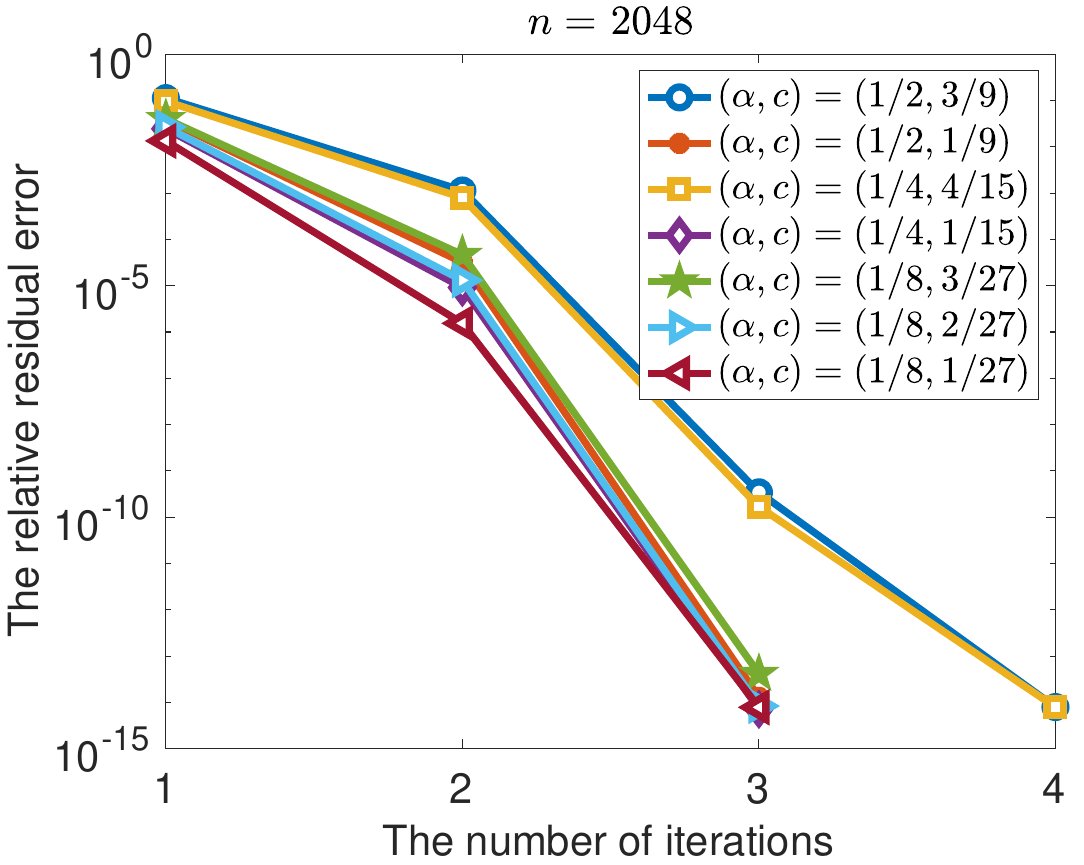}}\quad
  \subfigure{\includegraphics[width=0.42\textwidth]{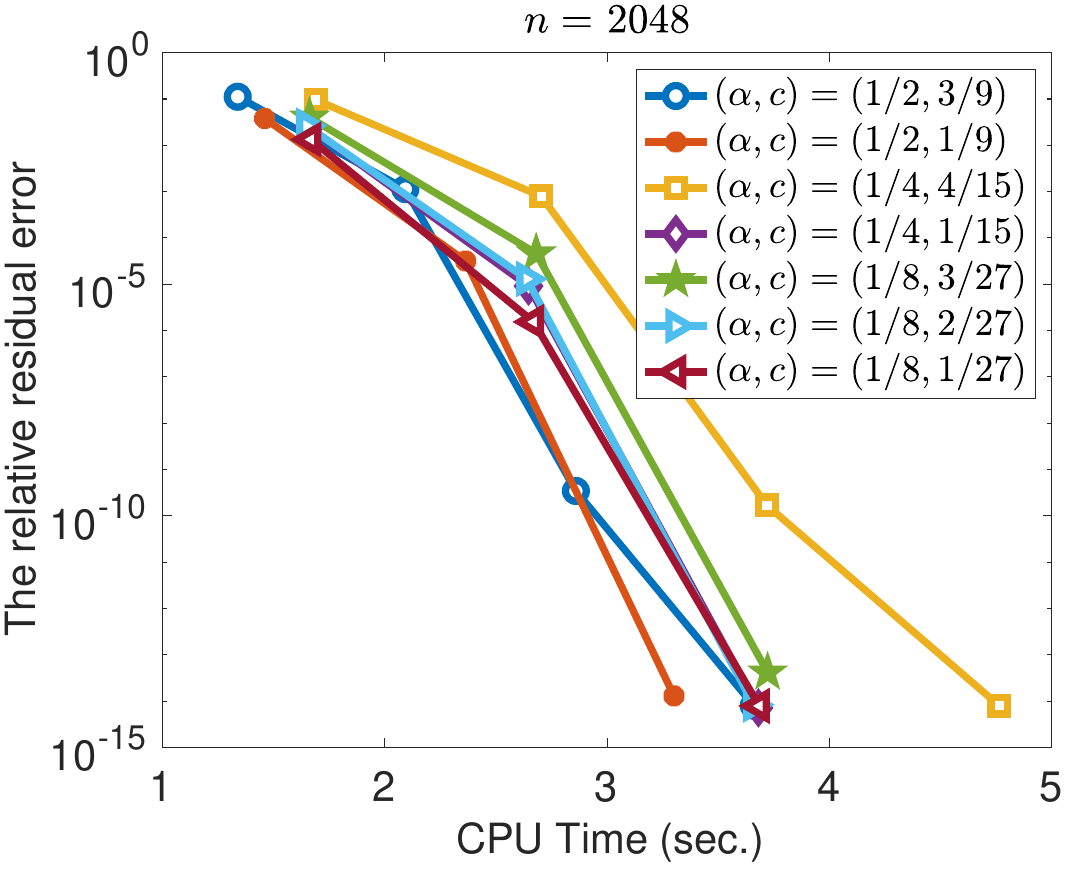}} \\
  \subfigure{\includegraphics[width=0.42\textwidth]{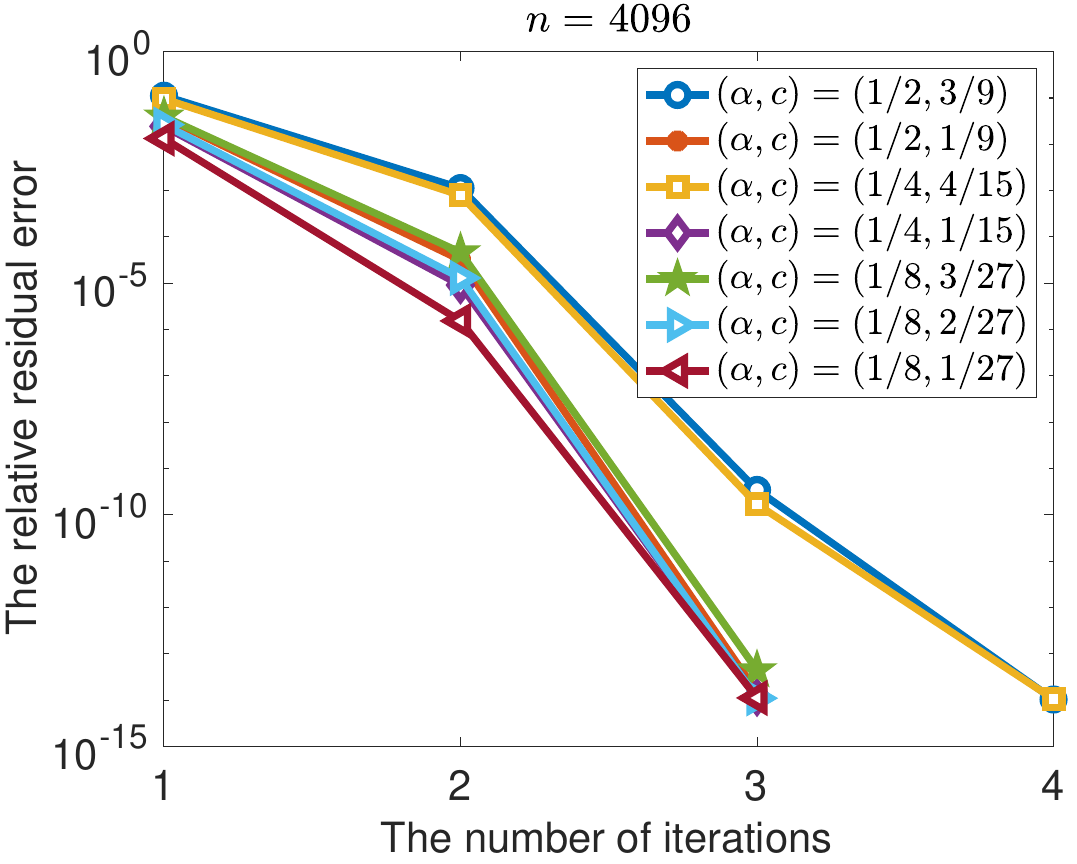}}\quad
  \subfigure{\includegraphics[width=0.42\textwidth]{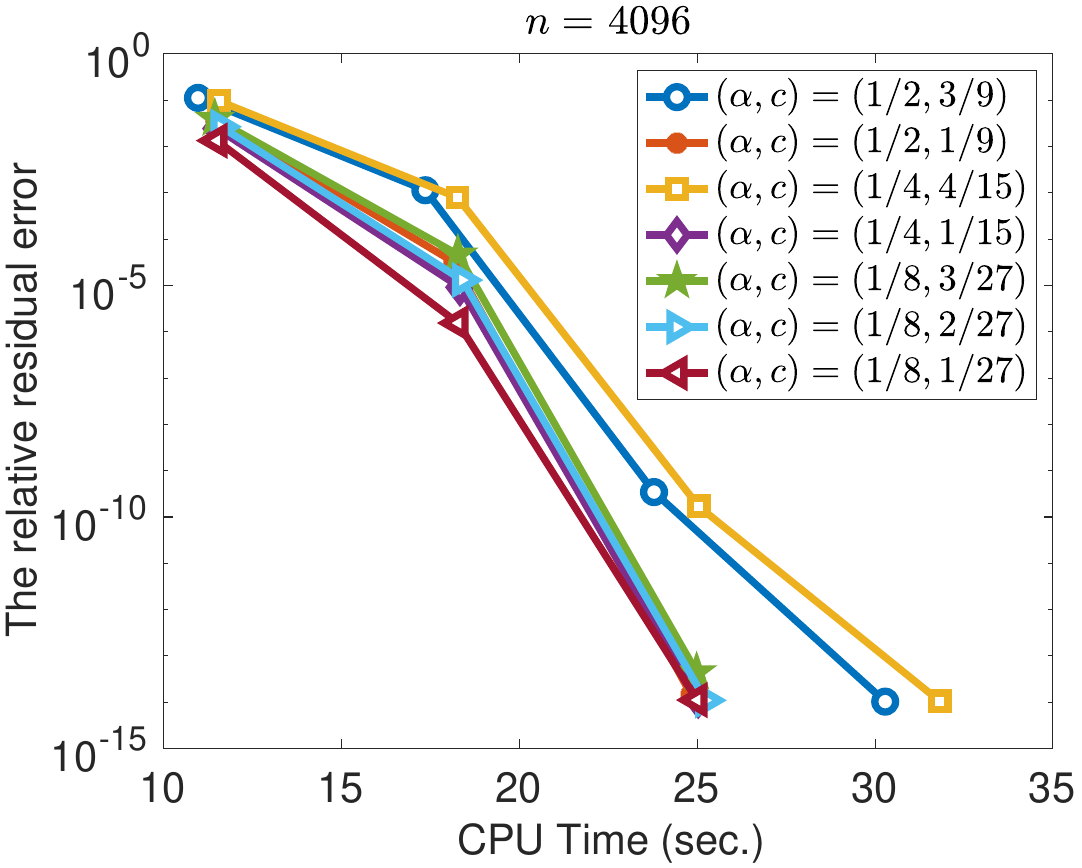}} \\
  \caption{The iteration histories of TSMNM for various $(\alpha, c)$ when the problem size $n = 1024, 2048, 4096$, respectively.}
  \label{fig:IterHistory_TSMNM}
\end{figure}

Since Theorems \ref{th:SingularConvergenceSetK} and \ref{th:SingularConvergenceSetW} 
concern local convergence of TSMNM 
and the minimal positive solution to the nonlinear equation \eqref{eq:f(x)=0} 
is usually unavailable, direct verification is impractical. 
Instead, we provide indirect verification of these results through comparative numerical 
experiments with other effective methods.
Recall that the convergence of Newton's method for systems with singular Jacobians 
at the solution has been shown to be locally linear \cite{Reddien1978}, 
whereas it achieves quadratic convergence for nonsingular problems \cite{Wang2000a,Lu2005b}.
In comparison, the following numerical experiments demonstrate that TSMNM achieves superquadratic
convergence for nonsingular problems and superlinear convergence for singular problems.

Let us first consider the normal case $(\alpha,c) = (0.5,0.5)$.
Figure \ref{fig:IterHistory_alphaC0.5} presents the iteration histories with the problem size $n = 1024, 2048, 4096, 8192$.
It shows that the TSMNM performed comparably or better than the existing methods.
One might notice that the number of iterations for all algorithms seems to be independent of the problem size.

\begin{figure}
  \centering
  \subfigure{\includegraphics[width=0.42\textwidth]{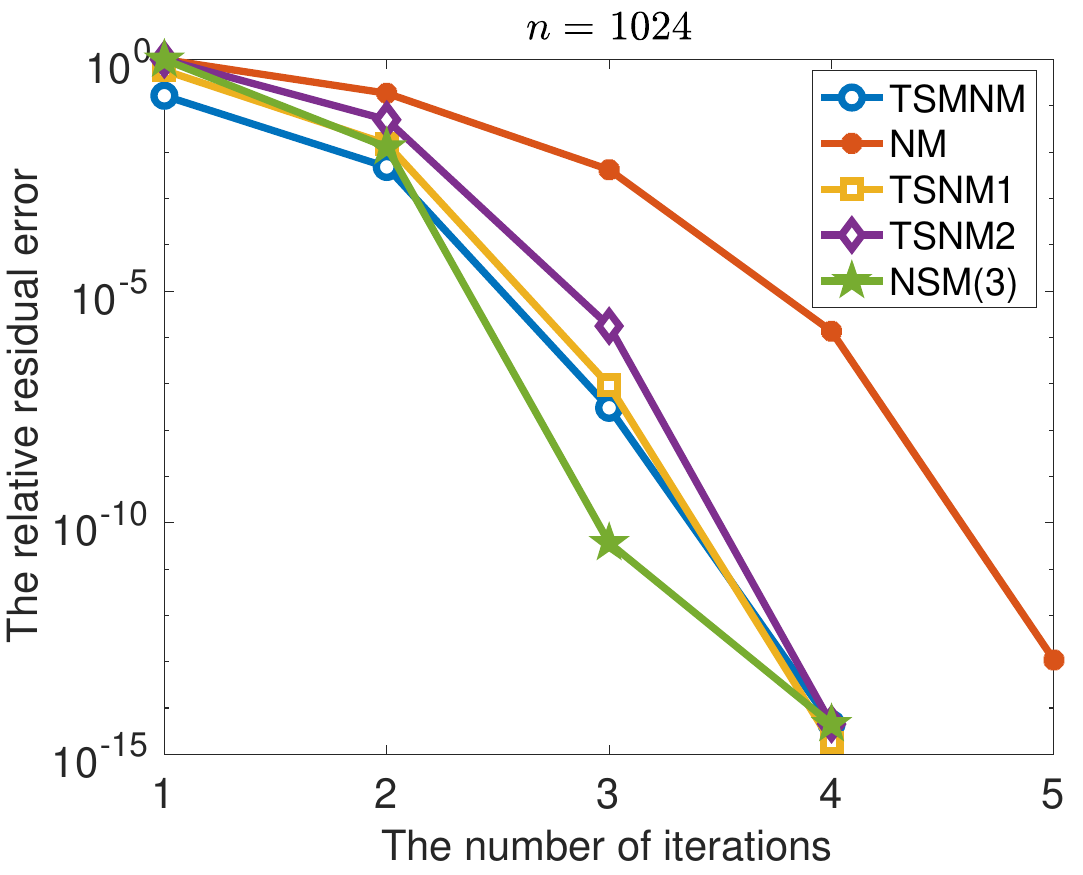}}\quad
  \subfigure{\includegraphics[width=0.42\textwidth]{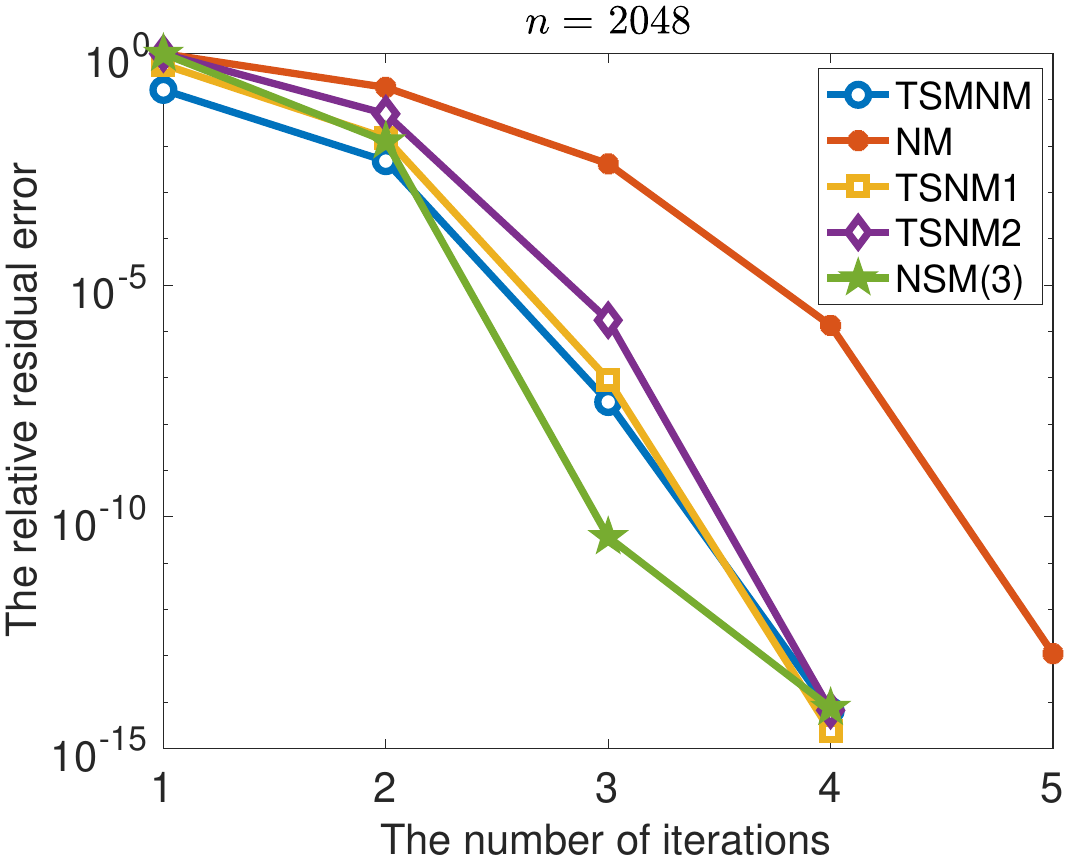}} \\
  \subfigure{\includegraphics[width=0.42\textwidth]{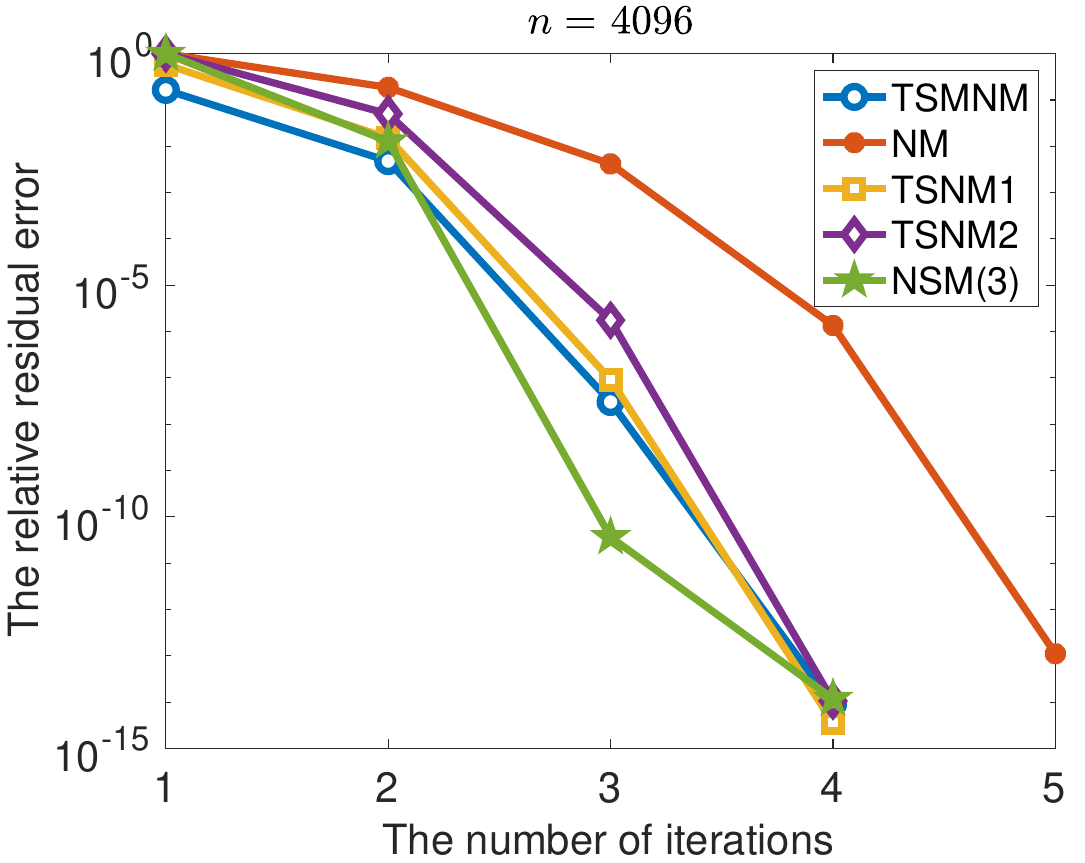}}\quad
  \subfigure{\includegraphics[width=0.42\textwidth]{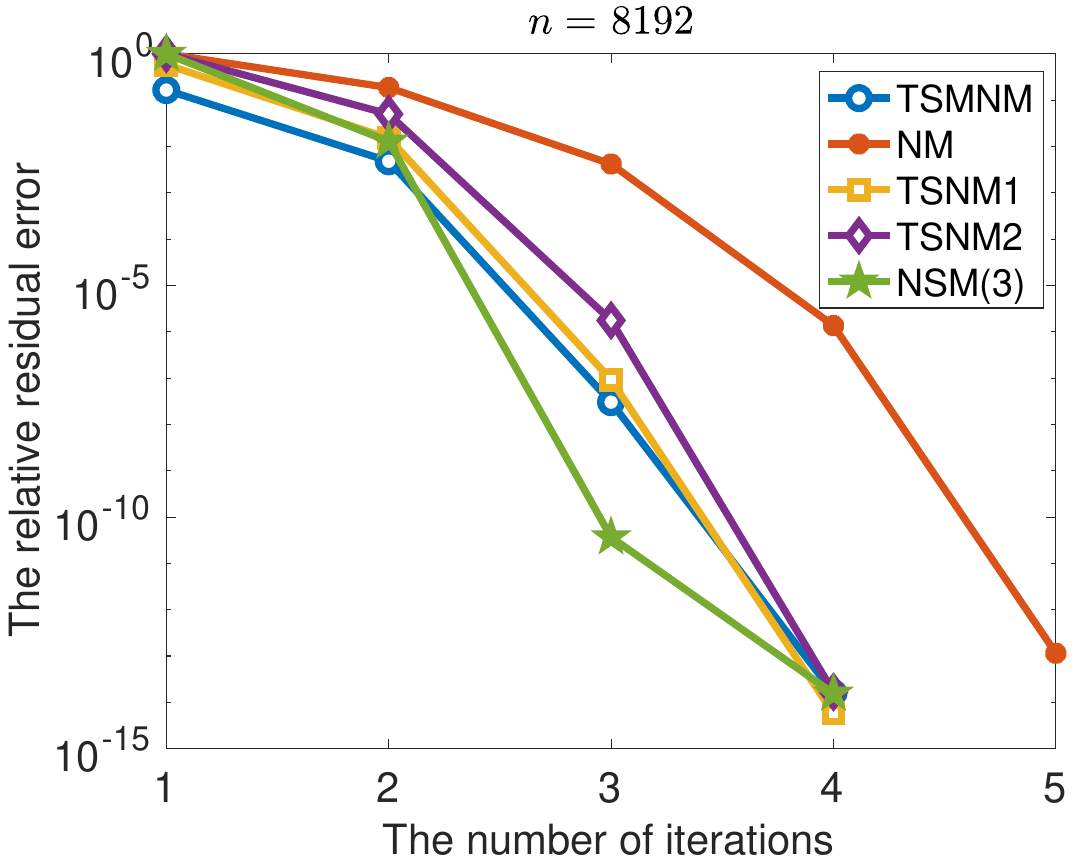}} \\
  \caption{The iteration histories for $(\alpha, c) = (0.5,0.5)$ when the problem size $n = 1024, 2048, 4096, 8192$, respectively.}
  \label{fig:IterHistory_alphaC0.5}
\end{figure}

Figure \ref{fig:IterHistory_alphaC10-4} shows the iteration histories for the problem size $n = 1024, 2048, 4096, 8192$
in a nearly singular case $(\alpha, c) = (10^{-4},1-10^{-4})$.
We see that the TSMNM achieves fewer iterations than NM,
although it requires an equal or greater number of iterations than 
TSNM1, TSNM2 and NSM(3).
It is not surprising that the computationally more expensive 
TSNM1, TSNM2 and NSM(3)
often require fewer iterations than the TSMNM.
Indeed, TSNM1 and TSNM2 are two-step Newton-type iterative methods with cubic convergence 
under regular differentiability conditions (See \cite{LingLL2022,LinBJ2010,ChunSN2011} for more details),
while NSM(3) achieves fourth-order convergence (See \cite{Kelley1986} for more details).
In contrast, TSMNM demonstrates superquadratic convergence, as established in \cite{Potra2017,CardenasCS2020,CardenasCS2022}.

\begin{figure}
  \centering
  \subfigure{\includegraphics[width=0.42\textwidth]{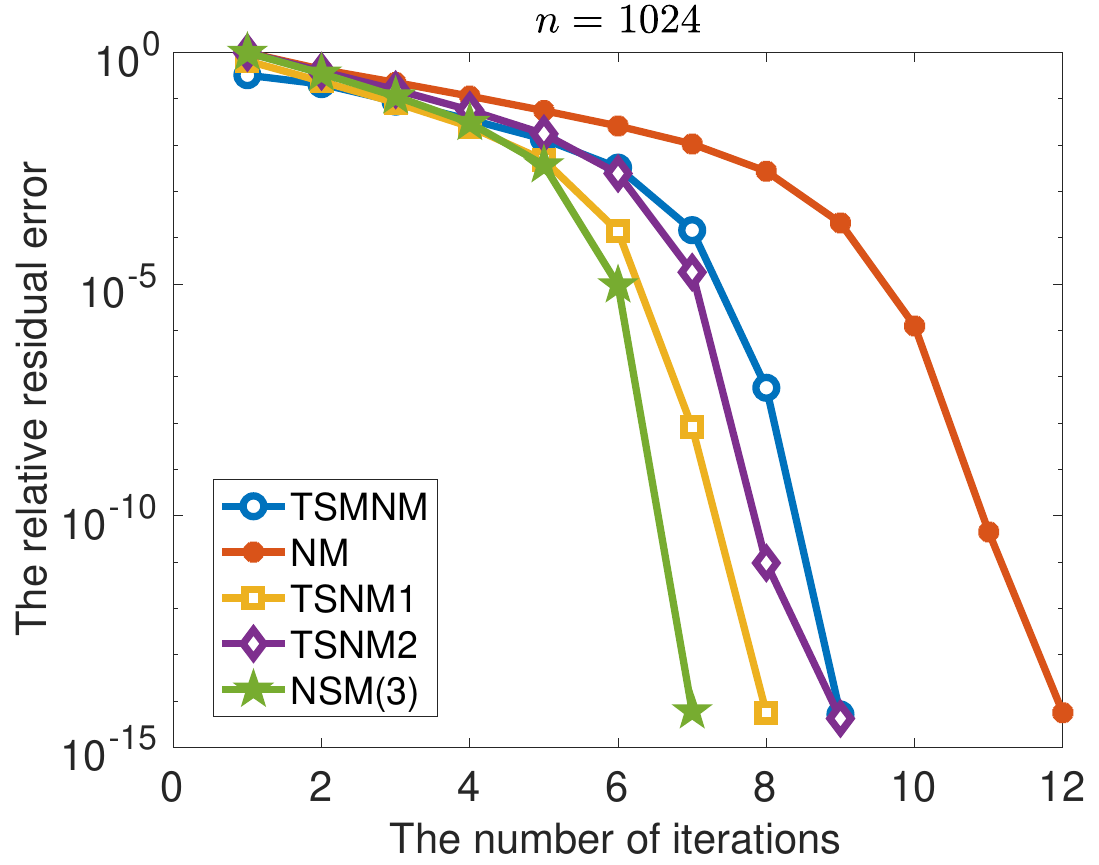}}\quad
  \subfigure{\includegraphics[width=0.42\textwidth]{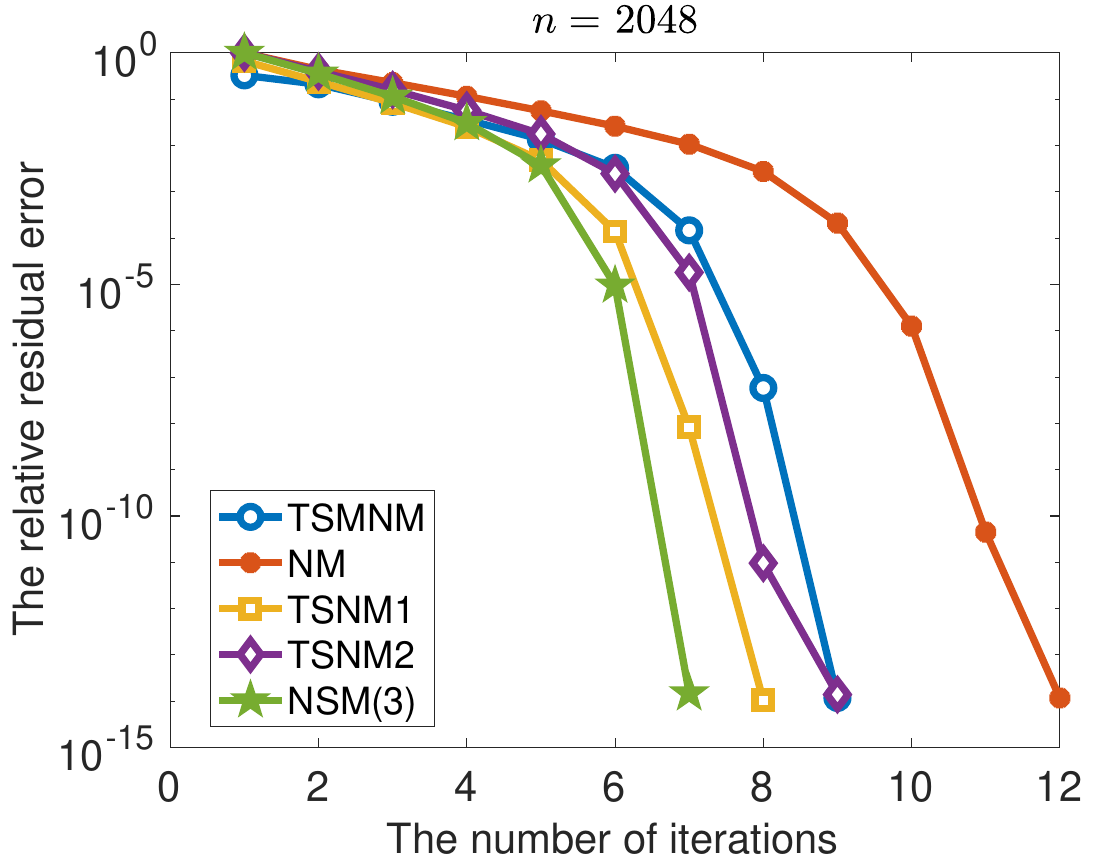}} \\
  \subfigure{\includegraphics[width=0.42\textwidth]{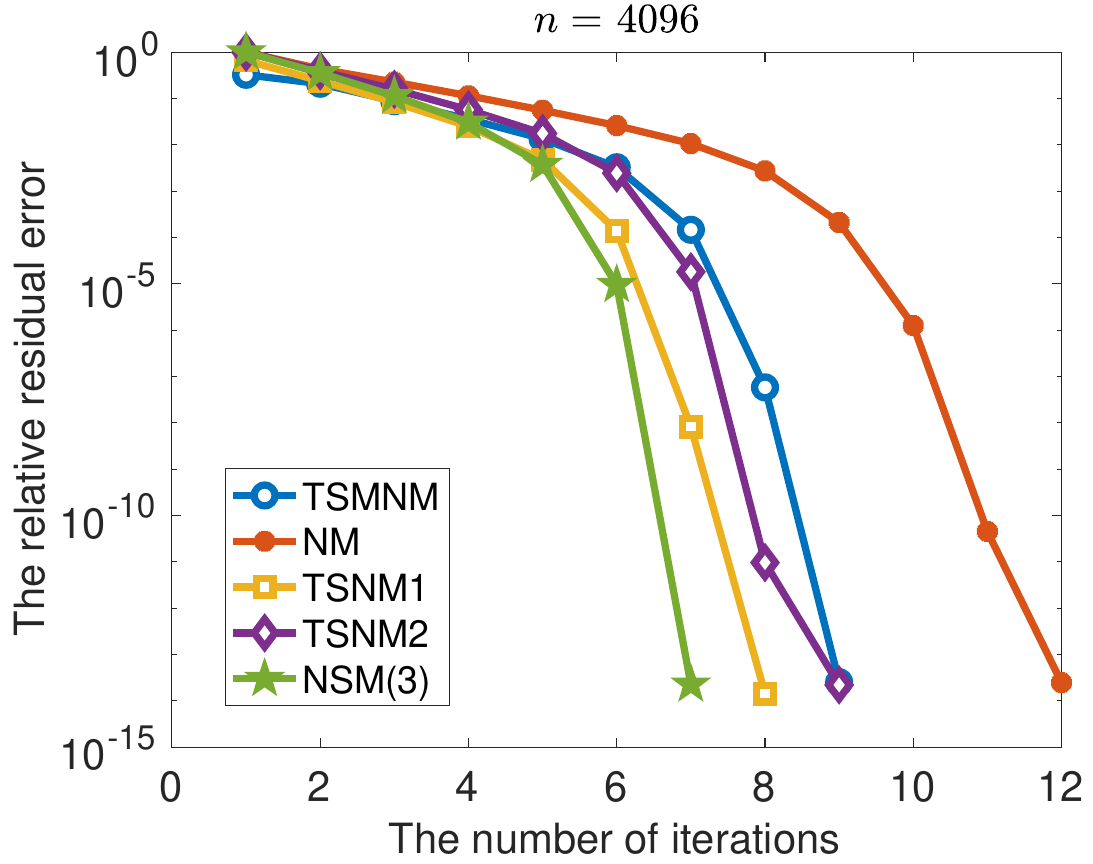}}\quad
  \subfigure{\includegraphics[width=0.42\textwidth]{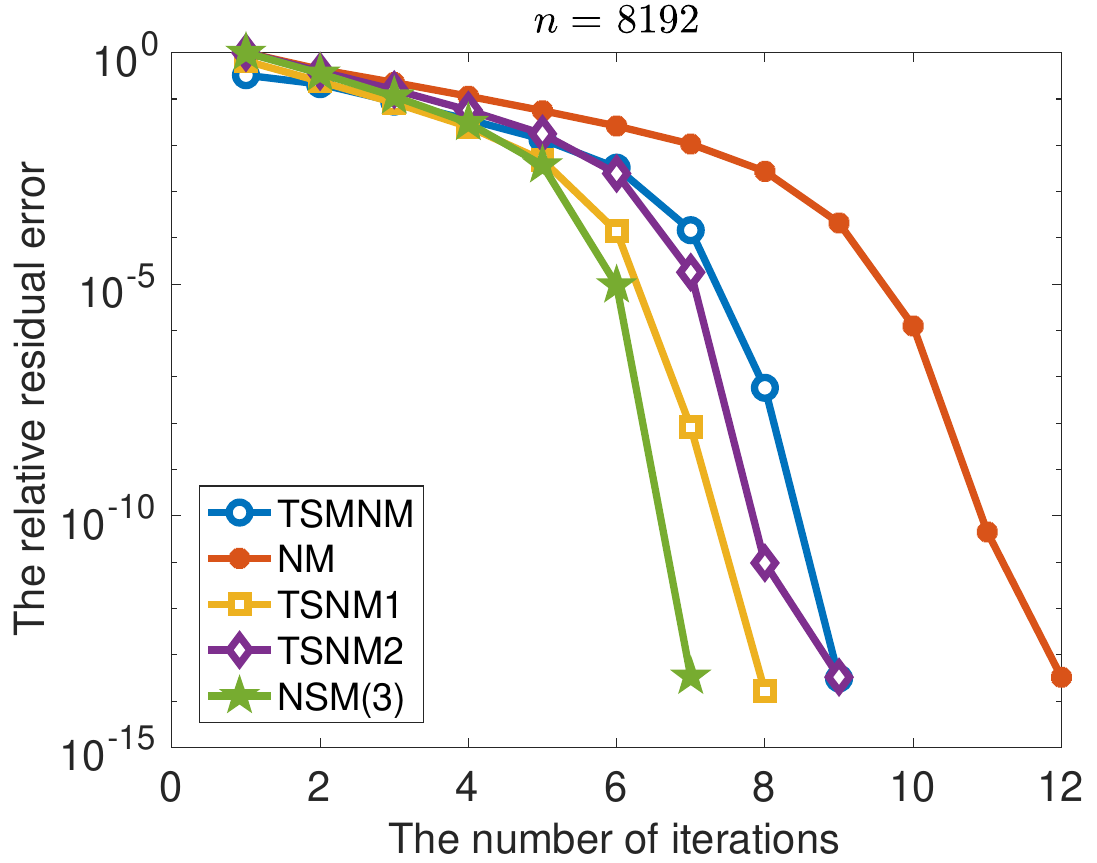}} \\
  \caption{The iteration histories for $(\alpha, c) = (10^{-4},1-10^{-4})$ when the problem size $n = 1024, 2048, 4096, 8192$, respectively.}
  \label{fig:IterHistory_alphaC10-4}
\end{figure}

We should note that the results in 
Figures \ref{fig:IterHistory_alphaC0.5} and \ref{fig:IterHistory_alphaC10-4}
do not imply that TSNM1, TSNM2 and NSM(3) are superior to TSMNM.
Tables \ref{tab:n1024}, \ref{tab:n2048}, \ref{tab:n4096} and \ref{tab:n8192}
provide the overall numerical results on eight cases for the problem sizes $n = 1024, 2048, 4096, 8192$, respectively.
These tables show that TSMNM outperforms NM in terms of the number of iterations
and performs comparable performance to TSNM1 and TSNM2 in terms of the number of iterations and the desired accuracy.
In particular, TSMNM has a significant advantage over other Newton-type methods
in terms of CPU time for $n = 2048, 4096, 8192$.
This advantage is further illustrated in Figures 
\ref{fig:IterHistory_alphaC10-3}, \ref{fig:IterHistory_alphaC10-5} and \ref{fig:IterHistory_alphaC10-7},
which present the iteration histories for the nearly singular cases 
$(\alpha, c) = (10^{-3},1-10^{-3})$, $(10^{-5},1-10^{-5})$ and $(10^{-7},1-10^{-7})$, respectively.

\begin{sidewaystable}
  \centering
  \caption{Numerical results for $n = 1024$}
  \label{tab:n1024}
  \begin{tabular}{cccccccccccc}
    \toprule
    $(\alpha,c)$ & Item & TSMNM & NM & TSNM1 & TSNM2 & NSM(3) & NSM(5) & NSM(10) & FPI & NBJ & NBGS \\
    \midrule
    (0.9,0.1) & IT  & 3          & 4          & 3          & 3          & 3          & 3          & 2          & 9          & 7 & 5     \\
              & CPU & 0.5902     & 0.9155     & 0.9362     & 0.8552     & 0.9880     & 1.3351     & 1.6072     & 0.0429     & 0.0367 & 0.0442 \\
              & RES & 5.0302e-15 & 5.6864e-15 & 1.9684e-15 & 4.8115e-15 & 6.3425e-15 & 5.2489e-15 & 4.1554e-15 & 6.3425e-15 & 6.2987e-14 & 4.4191e-16 \\
    (0.7,0.3) & IT  & 4          & 5          & 4          & 4          & 3          & 3          & 3          & 14         & 11 & 7  \\
              & CPU & 0.6626     & 1.0576     & 1.3547     & 1.2645     & 1.1020     & 1.4683     & 2.2202     & 0.0365     & 0.0340 & 0.0369  \\
              & RES & 4.4536e-15 & 4.6561e-15 & 1.6195e-15 & 5.0609e-15 & 5.6683e-15 & 4.2512e-15 & 5.6683e-15 & 1.3887e-13 & 3.1175e-14 & 1.0616e-15  \\
    (0.3,0.7) & IT  & 4          & 6          & 4          & 5          & 4          & 3          & 3          & 34         & 21 & 12  \\
              & CPU & 0.6611     & 1.2143     & 1.2925     & 1.3894     & 1.2966     & 1.3338     & 2.3527     & 0.0497     & 0.0412 & 0.0397   \\
              & RES & 1.1668e-13 & 5.3453e-15 & 1.9854e-15 & 4.1235e-15 & 4.8872e-15 & 6.5671e-15 & 3.8181e-15 & 1.1729e-13 & 2.0832e-13 & 4.8971e-14 \\
    (0.1,0.9) & IT  & 5          & 7          & 5          & 5          & 4          & 4          & 3          & 71         & 39 & 21   \\
              & CPU & 0.9125     & 1.5599     & 1.5355     & 1.5290     & 1.3759     & 1.9195     & 2.3689     & 0.0693     & 0.0506 & 0.0417  \\
              & RES & 4.7281e-15 & 4.0189e-15 & 1.7730e-15 & 4.6099e-15 & 2.2057e-13 & 4.1371e-15 & 4.3735e-15 & 1.9562e-13 & 1.9586e-13 & 1.2689e-13 \\
    ($10^{-3}$,$1-10^{-3}$) & IT & 8 & 10 & 7 & 8 & 6 & 5 & 4 & 727  & 335 & 173   \\
              & CPU & 1.2362     & 2.0296     & 2.1694     & 2.2245     & 2.0141     & 2.2014     & 2.9741     & 0.4248     & 0.2122 & 0.1258  \\
              & RES & 3.3827e-15 & 3.7049e-15 & 3.0635e-15 & 3.3827e-15 & 1.0631e-14 & 4.0271e-15 & 7.4098e-15 & 2.2326e-13 & 2.1654e-13 & 2.1654e-13 \\
    ($10^{-5}$,$1-10^{-5}$) & IT & 11 & 13 & 9 & 10 & 8 & 7 & 6 & 5944  & 2697 & 1397   \\
              & CPU & 1.9283     & 2.7065     & 2.8683     & 2.8955     & 2.6864     & 3.1833    & 4.5515     & 3.2460     & 1.4608 & 0.7996  \\
              & RES & 5.3754e-15 & 6.8038e-14 & 1.1519e-14 & 8.1400e-15 & 5.6826e-15 & 5.9898e-15 & 6.9113e-15 & 2.2700e-13 & 2.2715e-13 & 2.2639e-13 \\
    ($10^{-7}$,$1-10^{-7}$) & IT & 13 & 17 & 11 & 12 & 10 & 8 & 7 & 45005  & 20646 & 10796   \\
              & CPU & 2.2908     & 3.5179     & 3.4528     & 3.4504     & 3.3342     & 3.5883     & 5.4163     & 24.4269    & 11.0671 & 5.8570  \\
              & RES & 5.0434e-15 & 1.0133e-13 & 1.6506e-13 & 4.1570e-14 & 6.4189e-14 & 3.3776e-14 & 4.4168e-14 & 2.2772e-13 & 2.2680e-13 & 2.2711e-13 \\
    ($10^{-8}$,$1-10^{-8}$) & IT & 16 & 18 & 12 & 13 & 11 & 9 & 7 & 119320  & 55314 & 29155   \\
              & CPU & 2.9754     & 3.6582     & 3.6639     & 3.5681     & 3.6928     & 4.0579     & 5.4164     & 64.6936    & 29.9190 & 16.0651  \\
              & RES & 5.6527e-15 & 2.0625e-14 & 1.8425e-13 & 1.6866e-13 & 8.9373e-14 & 1.5415e-13 & 7.6693e-14 & 2.2733e-13 & 2.2733e-13 & 2.2718e-13 \\
    \bottomrule
  \end{tabular}
\end{sidewaystable}

\begin{sidewaystable}
  \caption{Numerical results for $n = 2048$}
  \label{tab:n2048}
  \begin{tabular}{cccccccccccc}
    \toprule
    $(\alpha,c)$ & Item & TSMNM & NM & TSNM1 & TSNM2 & NSM(3) & NSM(5) & NSM(10) & FPI &  NBJ & NBGS  \\
    \midrule
    (0.9,0.1) & IT  & 3           & 4          & 3          & 3           & 3          & 3            & 2          & 8          & 7          & 5 \\
              & CPU & 3.0558      & 4.8500     & 5.6766     & 4.8513      & 5.4003     & 7.3064       & 7.8967     & 0.1560     & 0.1512     & 0.1449  \\
              & RES & 9.1856e-15  & 6.9986e-15 & 2.6245e-15 & 6.5612e-15  & 6.7799e-15 & 8.3108e-15   & 5.6863e-15 & 3.8689e-13 & 6.2987e-14 & 4.4191e-16  \\
    (0.7,0.3) & IT  & 4           & 5          & 4          & 4           & 3          & 3            & 3          & 14         & 11         & 7 \\
              & CPU & 4.1650      & 6.0788     & 7.2487     & 6.5112     & 5.3977     & 7.2580       & 11.7011     & 0.1665     & 0.1549     & 0.1515  \\
              & RES & 6.4780e-15  & 7.6926e-15 & 2.4292e-15 & 7.2877e-15  & 6.6804e-15 & 7.6926e-15   & 8.9072e-15 & 1.3907e-13 & 3.1175e-14 & 1.0616e-15 \\
    (0.3,0.7) & IT  & 4           & 6          & 4          & 5           & 4          & 3            & 3          & 33         & 21         & 12 \\
              & CPU & 4.7678      & 7.1545     & 7.4123     & 8.0721      & 7.1510     & 7.2443       & 11.7533     & 0.2217     & 0.1874     & 0.1613 \\
              & RES & 1.2187e-13  & 7.3306e-15 & 3.0544e-15 & 8.7051e-15  & 6.5670e-15 & 1.0538e-14   & 8.2469e-15 & 2.8421e-13 & 2.0938e-13 & 4.9472e-14 \\
    (0.1,0.9) & IT  & 5           & 7          & 5          & 5           & 4          & 4            & 3          & 69         & 38         & 21  \\
              & CPU & 5.9817      & 8.3425     & 9.3229     & 8.1214      & 7.1530     & 9.4888       & 11.8931     & 0.3267     & 0.2383     & 0.1873 \\
              & RES & 7.9192e-15  & 8.6284e-15 & 3.5459e-15 & 8.8648e-15  & 2.2375e-13 & 1.0047e-14   & 9.5740e-15 & 4.3260e-13 & 4.2847e-13 & 1.2863e-13  \\
    ($10^{-3}$,$1-10^{-3}$) & IT & 8 & 10      & 7          & 7           & 6          & 5            & 4          & 705        & 325        & 168 \\
              & CPU & 9.4888      & 11.8904     & 12.9433     & 11.2598   & 10.6996   & 11.8986       & 15.6142     & 2.1866     & 1.0848     & 0.6233  \\
              & RES & 1.2563e-14  & 1.7073e-14 & 8.5367e-15 & 3.3067e-13  & 2.4966e-14 & 1.5785e-14   & 1.8684e-14 & 4.5325e-13 & 4.5062e-13 & 4.4804e-13  \\
    ($10^{-5}$,$1-10^{-5}$) & IT & 11 & 13     & 9          & 10          & 8          & 7            & 6          & 5725       & 2604       & 1350 \\
              & CPU & 12.6250     & 15.6192     & 16.6738     & 15.8981   & 14.2586    & 16.6978       & 23.4670      & 16.8142    & 7.7762     & 4.0582  \\
              & RES & 1.6586e-14  & 8.6769e-14 & 1.1364e-14 & 1.6125e-14  & 1.7661e-14 & 1.7046e-14   & 3.3018e-14  & 4.5426e-13 & 4.5411e-13 & 4.5258e-13  \\
    ($10^{-7}$,$1-10^{-7}$) & IT & 13 & 17     & 11         & 12          & 10         & 8            & 7          & 42817       & 19705      & 10327 \\
              & CPU & 14.8090     & 20.1359      & 20.1903     & 21.3254  & 17.7938     & 19.1700       & 27.4882     & 125.8284   & 57.6355    & 30.4087   \\
              & RES & 7.6410e-14 & 2.3580e-13  & 1.7345e-13 & 5.0430e-14  & 1.0667e-13 & 8.3439e-14   & 1.7360e-13 & 4.5448e-13 & 4.5433e-13 & 4.5403e-13  \\
    ($10^{-8}$,$1-10^{-8}$) & IT & 14 & 18     & 12         & 13          & 11         & 9            & 7          & 112384     & 52336      & 27666 \\
              & CPU & 14.9636    & 21.8056     & 21.8460    & 20.5552     & 20.8220    & 28.4572       & 28.4572   & 356.6351   & 155.5863    & 82.7813   \\
              & RES & 1.1610e-14 & 1.7644e-13  & 2.4243e-13 & 1.9248e-13  & 2.8933e-13 & 3.2737e-13   & 4.0772e-13 & 4.5462e-13 & 4.5386e-13 & 4.5401e-13  \\
    \bottomrule
  \end{tabular}
\end{sidewaystable}

\begin{sidewaystable}
  \caption{Numerical results for $n = 4096$}
  \label{tab:n4096}
  \begin{tabular}{cccccccccccc}
    \toprule
    $(\alpha,c)$ & Item & TSMNM & NM & TSNM1 & TSNM2 & NSM(3) & NSM(5) & NSM(10) &  FPI & NBJ & NBGS  \\
    \midrule
    (0.9,0.1)    & IT   & 3           & 4          & 3           & 3          & 3          & 3          & 2          &  8         & 7          & 5 \\
                 & CPU  & 25.1027     & 31.5777    & 37.6058     & 31.4310    & 34.6456    & 45.7067    & 48.3898    &  0.7430    & 0.7282     & 0.7068 \\
                 & RES  & 1.0279e-14  & 1.1373e-14 & 3.9367e-15  & 1.1154e-14 & 1.0279e-14 & 7.8734e-15 & 6.9986e-15 & 3.8711e-13 & 6.2768e-14 & 4.4191e-16  \\
    (0.7,0.3)    & IT   & 4           & 5          & 4           &  4         & 3          & 3          & 3          & 14         & 11         & 6  \\
                 & CPU  & 34.4989     & 39.4826    & 49.4020     & 41.3927    & 34.6917    & 46.5338    & 72.6999    & 0.8155     & 0.7762     & 0.7190  \\
                 & RES  & 1.1336e-14  & 1.0527e-14 & 4.2511e-15  & 1.1944e-14 & 1.0122e-14 & 1.1134e-14 & 1.4778e-14 &  1.3887e-13 & 3.1378e-14 & 4.7410e-13  \\
      (0.3,0.7)  & IT   & 4           & 6          & 4           &  4         & 4          & 3          & 3          & 32         & 21         & 12  \\
                 & CPU  & 34.6231     & 46.7633    & 49.6253     & 41.6914    & 46.3811    & 45.4788    & 72.5778    & 1.0199     & 0.9175     & 0.7905 \\
                 & RES  & 1.2279e-13  & 1.2065e-14 & 4.2761e-15  & 5.7850e-13 & 9.6213e-15 & 1.4356e-14 & 1.0385e-14 &  6.9212e-13 & 2.1060e-13 & 5.0475e-14 \\
      (0.1,0.9)  & IT   & 5           & 7          & 5           &  5         & 4          & 4          & 3          & 68         & 37         & 20 \\
                 & CPU  & 41.6221     & 55.1117    & 62.1316     & 51.8485    & 46.0889    & 61.0542    & 73.0014    & 1.4464     & 1.0898     & 0.8889  \\
                 & RES  & 1.2174e-14  & 1.5956e-14 & 6.2643e-15  & 1.4065e-14 & 2.2906e-13 & 1.3120e-14 & 1.2529e-14 &  6.4263e-13 & 8.8138e-13 & 5.7073e-13 \\
    ($10^{-3}$,$1-10^{-3}$) & IT & 8  & 10         & 7           &  7         & 6          & 5          & 4          & 684        & 316        & 164  \\
                 & CPU  & 63.1484     & 78.1408    & 86.5390     & 72.5856    & 68.9883    & 75.3689    & 96.7285    & 8.7052     & 4.4138     & 2.6127  \\
                 & RES  & 1.5462e-14  & 2.2066e-14 & 1.2402e-14  & 3.3437e-13 & 2.3999e-14 & 1.8200e-14 & 2.2871e-14 & 8.8972e-13 & 8.6250e-13 & 8.0367e-13 \\
    ($10^{-5}$,$1-10^{-5}$) & IT & 10 & 13         & 9           &  10        & 8          & 7          & 6          & 5507       & 2508       & 1303  \\
                 & CPU  & 77.2035     & 101.8549   & 111.6842    & 103.4050   & 92.0644    & 105.7777   & 146.5145   &  65.7542   & 30.6062    & 16.1430 \\
                 & RES  & 7.6661e-13  & 1.0243e-13 & 1.5510e-14  & 2.2113e-14 & 1.9196e-14 & 2.3188e-14 & 2.7488e-14 & 9.0743e-13  & 9.0660e-13 & 8.9867e-13 \\
    ($10^{-7}$,$1-10^{-7}$) & IT & 13 & 17         & 11          &  12        & 10         & 8          & 7          & 40627      & 18753      & 9852 \\
                 & CPU  & 97.4787     & 132.4134   & 135.7965    & 123.6346   & 114.6638   & 120.5447   & 169.1866   & 482.0787   & 225.6549   & 117.8627 \\
                 & RES  & 1.8750e-13  & 5.5471e-14 & 4.7372e-14  & 3.3772e-14 & 1.2775e-13 & 5.2873e-14 & 2.1806e-13 &  9.0924e-13 & 9.0786e-13 & 9.0649e-13 \\
    ($10^{-8}$,$1-10^{-8}$) & IT & 14 & 18         & 12          &  13        & 11         & 9          & 7          & 105468     & 49331      & 26165 \\
                 & CPU  & 106.4377    & 144.2261   & 152.6316    & 140.1546   & 131.1754   & 141.9870   & 178.5167   & 1276.4307   & 601.3953   & 317.9439  \\
                 & RES  & 9.8069e-14  & 2.9497e-13 & 3.0093e-14  & 3.8051e-13 & 7.1184e-14 & 3.0597e-13 & 2.7252e-13 &  9.0905e-13 & 9.0920e-13 & 9.0737e-13 \\
    \bottomrule
  \end{tabular}
\end{sidewaystable}

\begin{sidewaystable}
  \caption{Numerical results for $n = 8192$}
  \label{tab:n8192}
  \begin{tabular}{cccccccccccc}
    \toprule
    $(\alpha,c)$ & Item & TSMNM & NM & TSNM1 & TSNM2 & NSM(3) & NSM(5) & NSM(10) & FPI & NBJ & NBGS \\
    \midrule
    (0.9,0.1)    & IT   & 3          & 4          & 3          & 3          & 3          & 3          & 2          & 8           & 7          & 5  \\
                 & CPU  & 180.6637   & 212.5825   & 263.0222   & 206.9046   & 221.1345   & 282.6704   & 293.8701   & 2.9567      & 2.9104     & 2.8348  \\
                 & RES  & 1.6622e-14 & 1.4435e-14 & 5.0302e-15 & 1.7715e-14 & 1.4653e-14 & 8.9669e-15 & 1.1591e-14 & 3.8711e-13  & 6.2768e-14 & 4.4191e-16 \\
    (0.7,0.3)    & IT   & 4          & 5          & 4          & 4          & 3          & 3          & 3          & 13          & 10         & 6 \\
                 & CPU  & 232.0622   & 263.4752   & 350.2669   & 275.3082   & 220.5804   & 283.3959   & 441.7767   & 3.1890      & 3.0415     & 2.8649  \\
                 & RES  & 1.5385e-14 & 1.5385e-14 & 5.0609e-15 & 1.6195e-14 & 1.1695e-14 & 1.6195e-14 & 1.6397e-14 & 1.3871e-12  & 9.4132e-13 & 4.7432e-13  \\
      (0.3,0.7)  & IT   & 4          & 6          & 4          & 4          & 4          & 3          & 3          & 31          & 20         & 11  \\
                 & CPU  & 232.5082   & 316.1660   & 349.4747   & 280.4237   & 294.6426   & 283.1361   & 434.7072   & 4.0063      & 3.5101     & 3.1087  \\
                 & RES  & 1.2401e-13 & 1.6035e-14 & 6.2614e-15 & 5.7895e-13 & 1.6341e-14 & 1.9242e-14 & 1.6952e-14 & 1.6774e-12  & 9.8701e-13 & 9.1456e-13 \\
      (0.1,0.9)  & IT   & 5          & 7          & 5          & 5          & 4          & 4          & 3          & 66          & 37         & 20  \\
                 & CPU  & 281.4508   & 382.3518   & 436.1224   & 343.9620   & 294.7776   & 376.6070   & 440.2856   & 5.6905      & 4.3283     & 3.5192 \\
                 & RES  & 2.0093e-14 & 1.7847e-14 & 8.2736e-15 & 1.9502e-14 & 2.3296e-13 & 1.9502e-14 & 2.0448e-14 & 1.4162e-12  & 8.8421e-13 & 5.7260e-13  \\
    ($10^{-3}$,$1-10^{-3}$) & IT & 8 & 10         & 7          & 7          & 6          & 5          & 4          & 662         & 306        & 159 \\
                 & CPU  & 435.2637   & 540.0512   & 609.5896   & 514.7841   & 440.8923   & 470.5313   & 578.9107   & 33.8761     & 16.9705    & 10.0798 \\
                 & RES  & 2.1582e-14 & 3.0118e-14 & 1.6106e-14 & 3.4950e-13 & 4.0910e-14 & 3.0953e-14 & 3.7044e-14 & 1.8058e-12  & 1.7838e-12 & 1.6694e-12 \\
    ($10^{-5}$,$1-10^{-5}$) & IT & 10  & 13       &  9         & 10         & 8          & 7          & 6          & 5288        & 2413       & 1255  \\
                & CPU  & 530.8189    & 699.4390   & 784.3909   & 705.7387   & 587.2113   & 660.7858   & 862.007    & 252.4722    & 116.3761   & 62.8775 \\
                & RES  & 7.6966e-13  & 6.5878e-14 & 1.9656e-14 & 3.9773e-14 & 3.7008e-14 & 4.1769e-14 & 3.6394e-14 & 1.8160e-12  & 1.8162e-12 & 1.8016e-12 \\
    ($10^{-7}$,$1-10^{-7}$) & IT & 13 & 16        &  11        & 12         & 10         & 8          & 7          & 38433       & 17804      & 9376 \\
                & CPU  & 669.3444    & 858.4441   & 957.3584   & 813.9931   & 733.4424   & 752.0414   & 1014.3375  & 1821.7865   & 842.279    & 445.4703 \\
                & RES  & 2.7200e-14  & 1.2129e-12 & 8.7713e-14 & 1.1980e-13 & 1.8551e-13 & 1.0513e-13 & 1.5021e-13 & 1.8187e-12  & 1.8149e-12 & 1.8189e-12 \\
    ($10^{-8}$,$1-10^{-8}$) & IT & 14 & 18        &  12        & 13         & 11         & 9          & 7          & 98540       & 46324      & 24660 \\
                & CPU  & 756.4959    & 1003.1779  & 1085.9651  & 942.6702   & 836.6878   & 888.0454   & 1136.2941  & 4687.9974   & 2191.4931  & 1172.7650 \\
                & RES  & 6.5256e-13  & 7.1946e-14 & 6.0765e-13 & 5.7282e-14 & 2.8992e-13 & 8.8444e-14 & 1.0311e-13 & 1.8188e-12  & 1.8174e-12 & 1.8178e-12 \\
    \bottomrule
  \end{tabular}
\end{sidewaystable}

\begin{figure}
  \centering
  \subfigure{\includegraphics[width=0.42\textwidth]{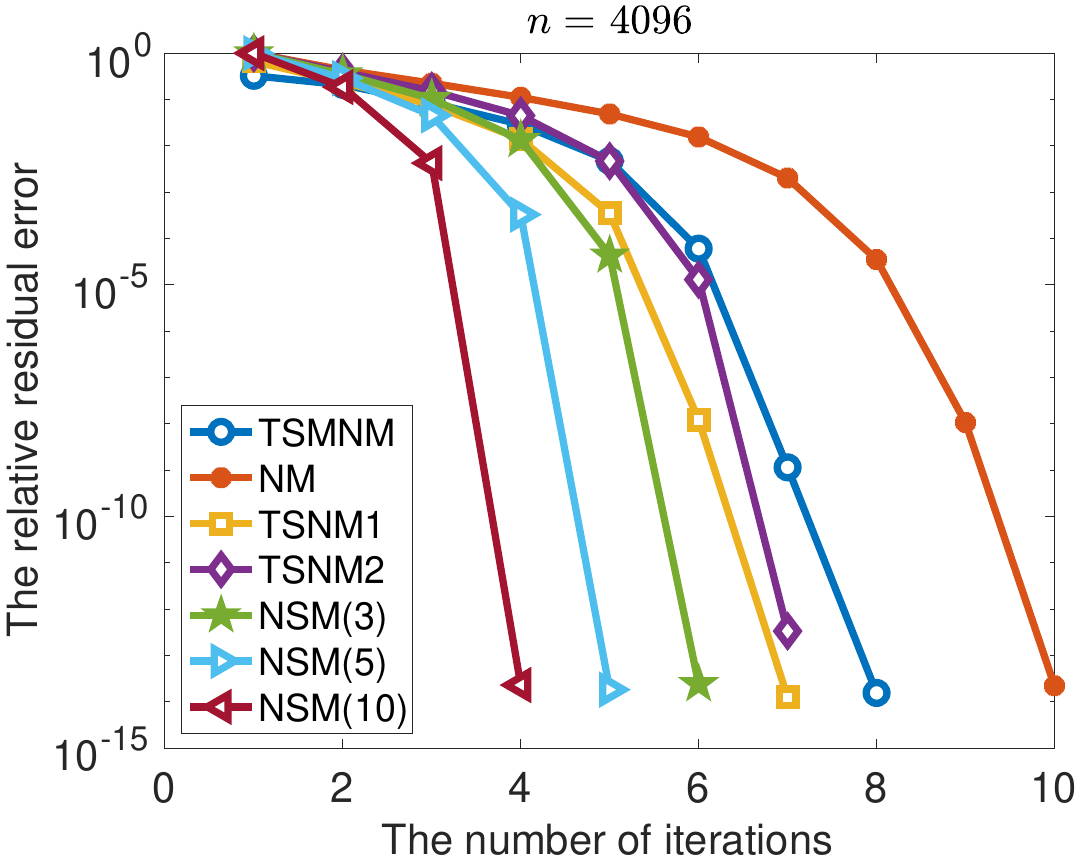}}\quad
  \subfigure{\includegraphics[width=0.42\textwidth]{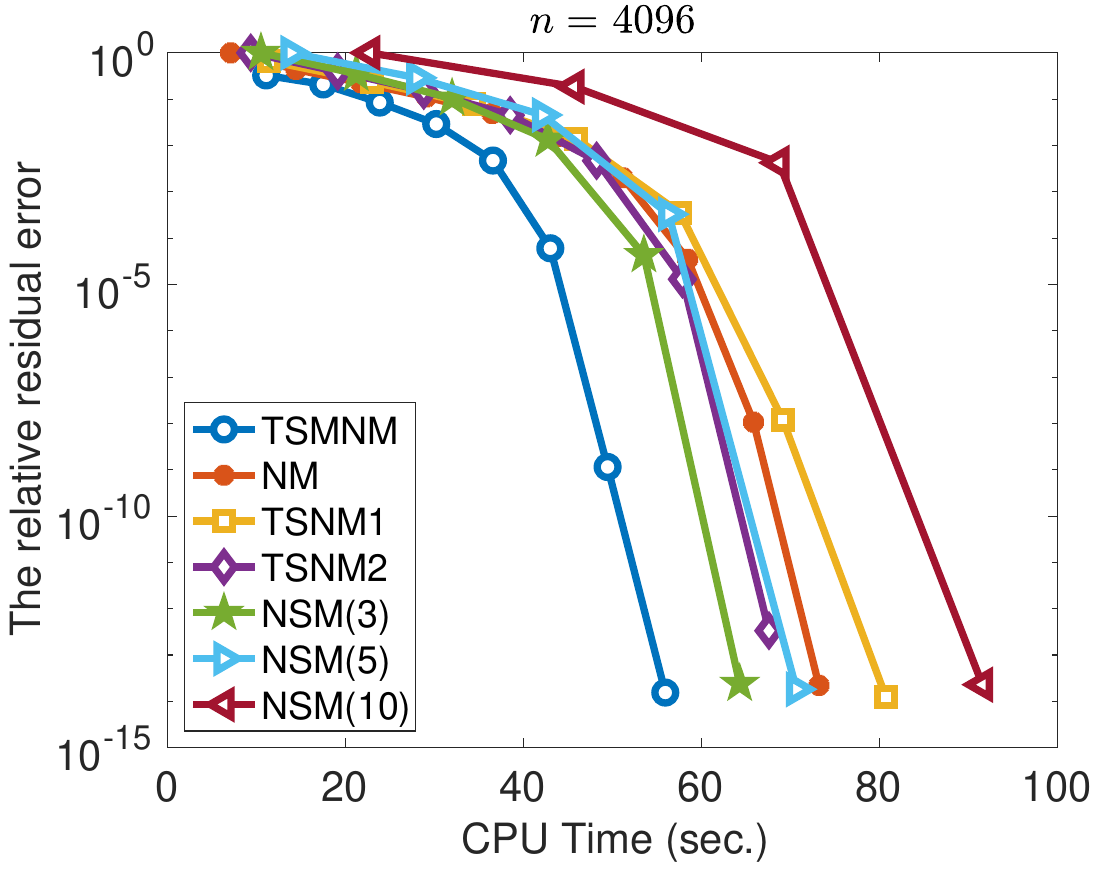}} \\
  \subfigure{\includegraphics[width=0.42\textwidth]{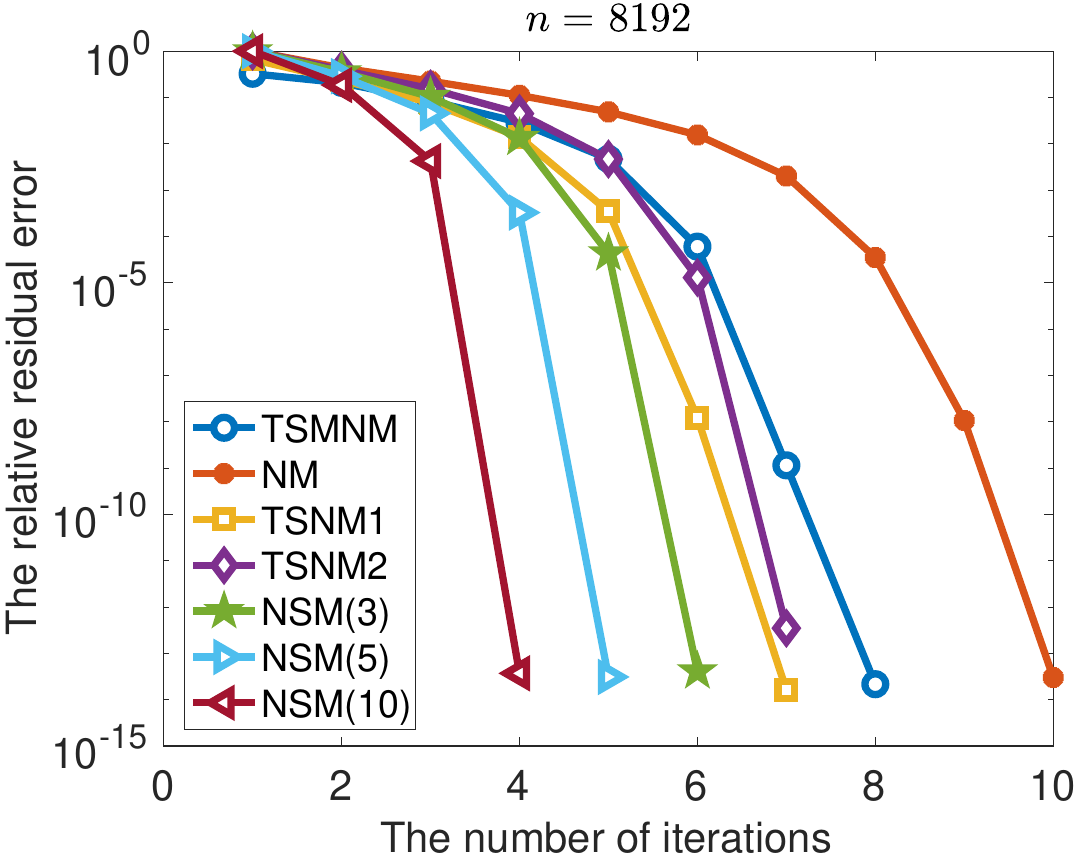}}\quad
  \subfigure{\includegraphics[width=0.42\textwidth]{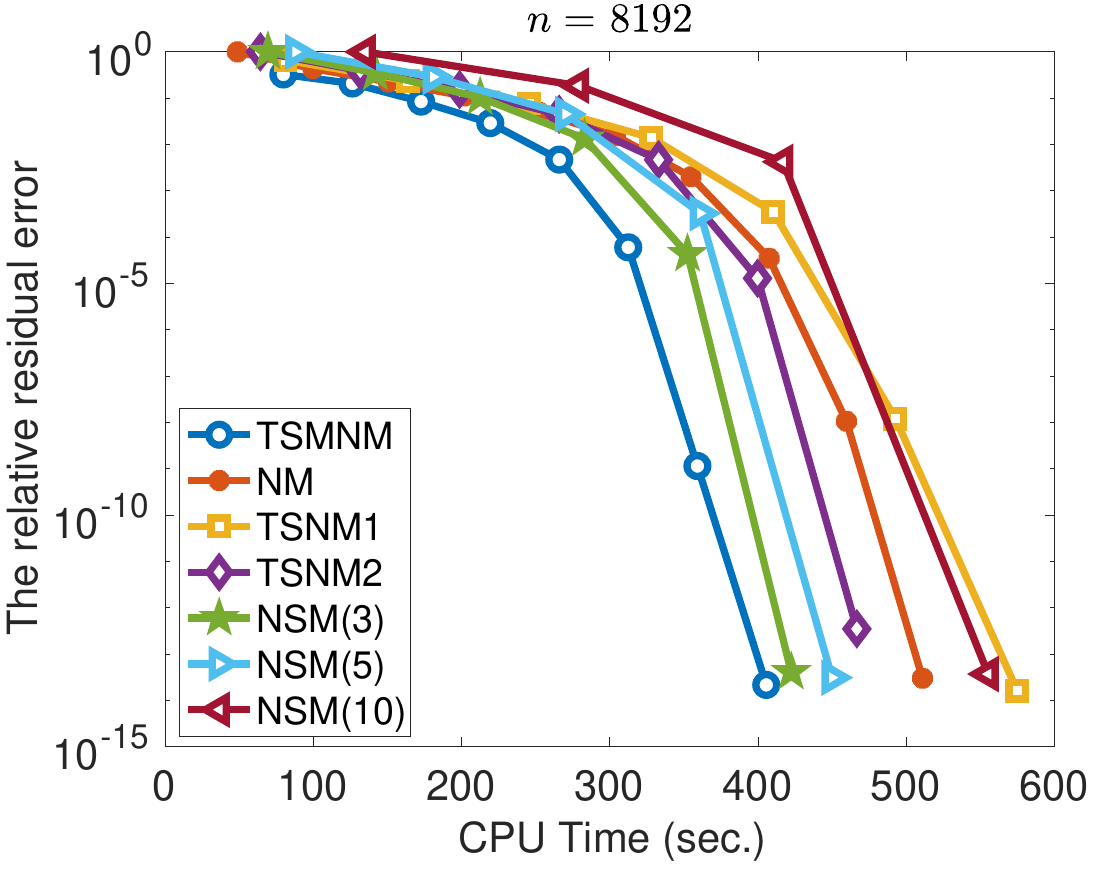}} \\
  \caption{Left: Iterations. Right: Time. Comparison of TSMNM with other Newton-type methods for $(\alpha, c) = (10^{-3},1-10^{-3})$ when the problem size $n = 4096, 8192$, respectively.}
  \label{fig:IterHistory_alphaC10-3}
\end{figure}

\begin{figure}
  \centering
  \subfigure{\includegraphics[width=0.42\textwidth]{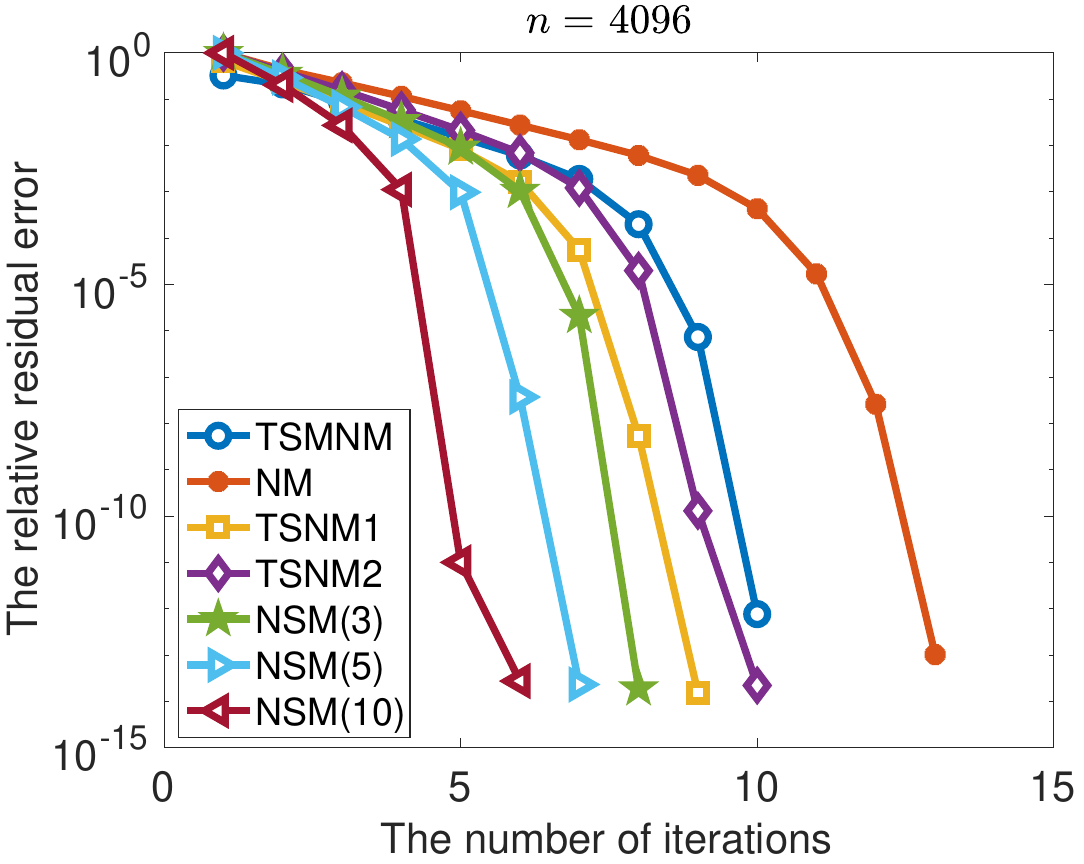}}\quad
  \subfigure{\includegraphics[width=0.42\textwidth]{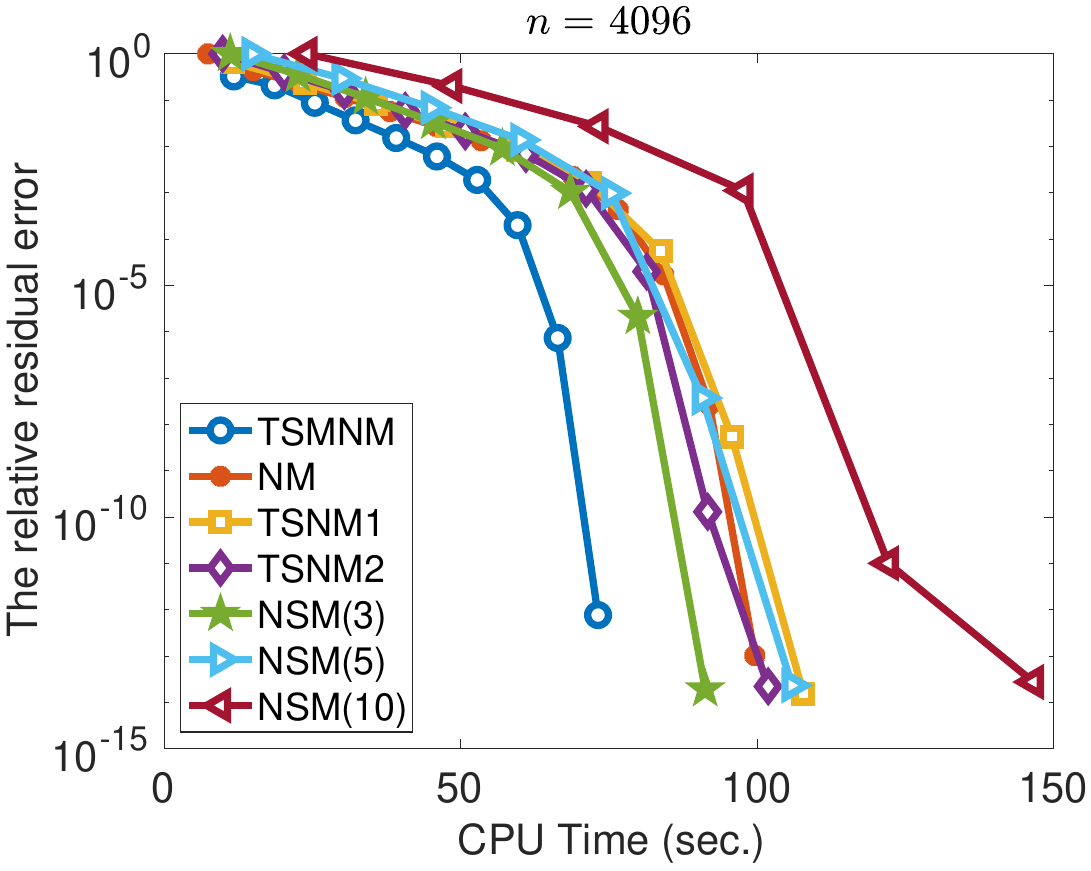}} \\
  \subfigure{\includegraphics[width=0.42\textwidth]{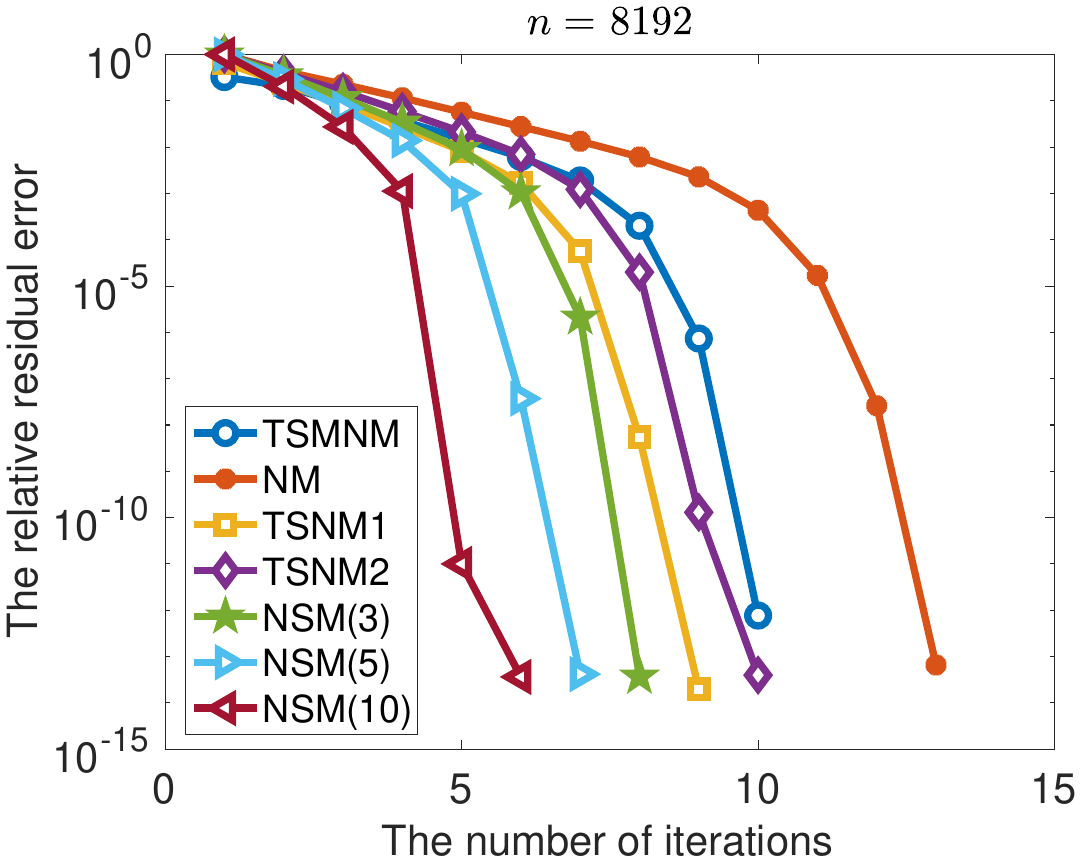}}\quad
  \subfigure{\includegraphics[width=0.42\textwidth]{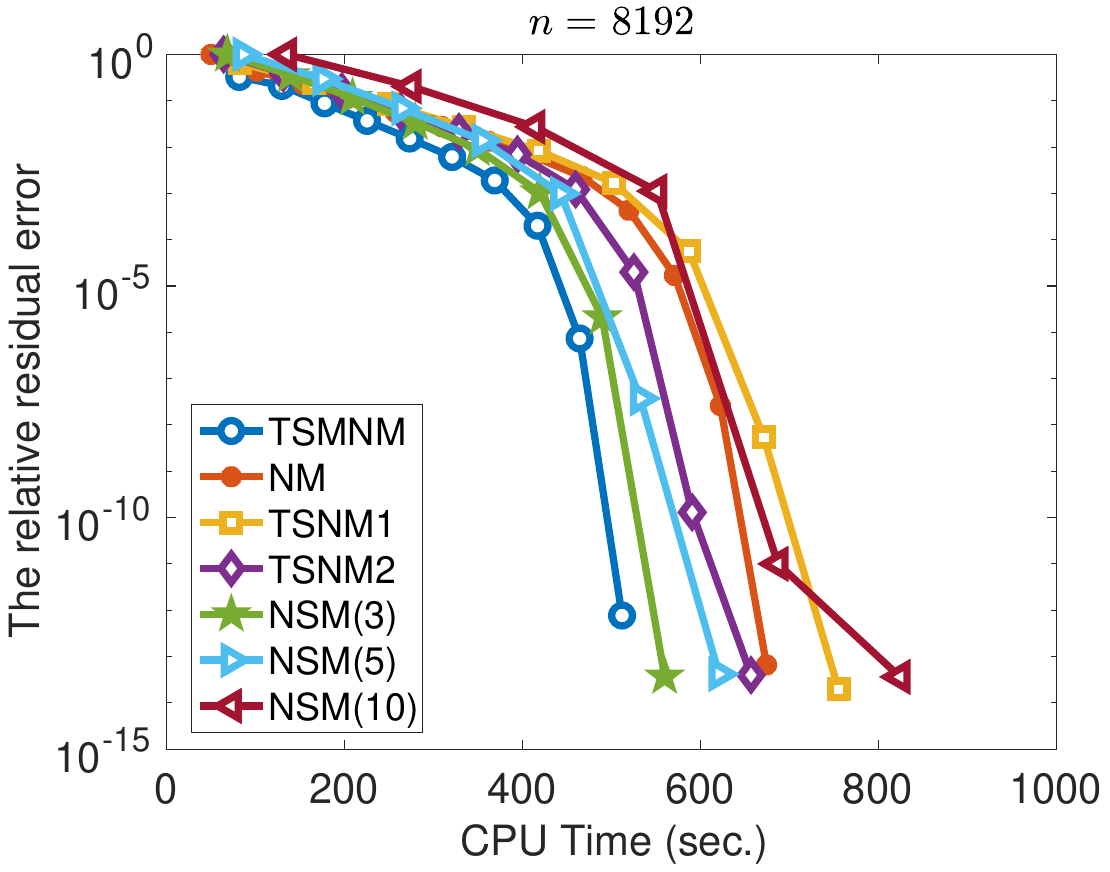}} \\
  \caption{Left: Iterations. Right: Time. Comparison of TSMNM with other Newton-type methods for $(\alpha, c) = (10^{-5},1-10^{-5})$ when the problem size $n = 4096, 8192$, respectively.}
  \label{fig:IterHistory_alphaC10-5}
\end{figure}

\begin{figure}
  \centering
  \subfigure{\includegraphics[width=0.42\textwidth]{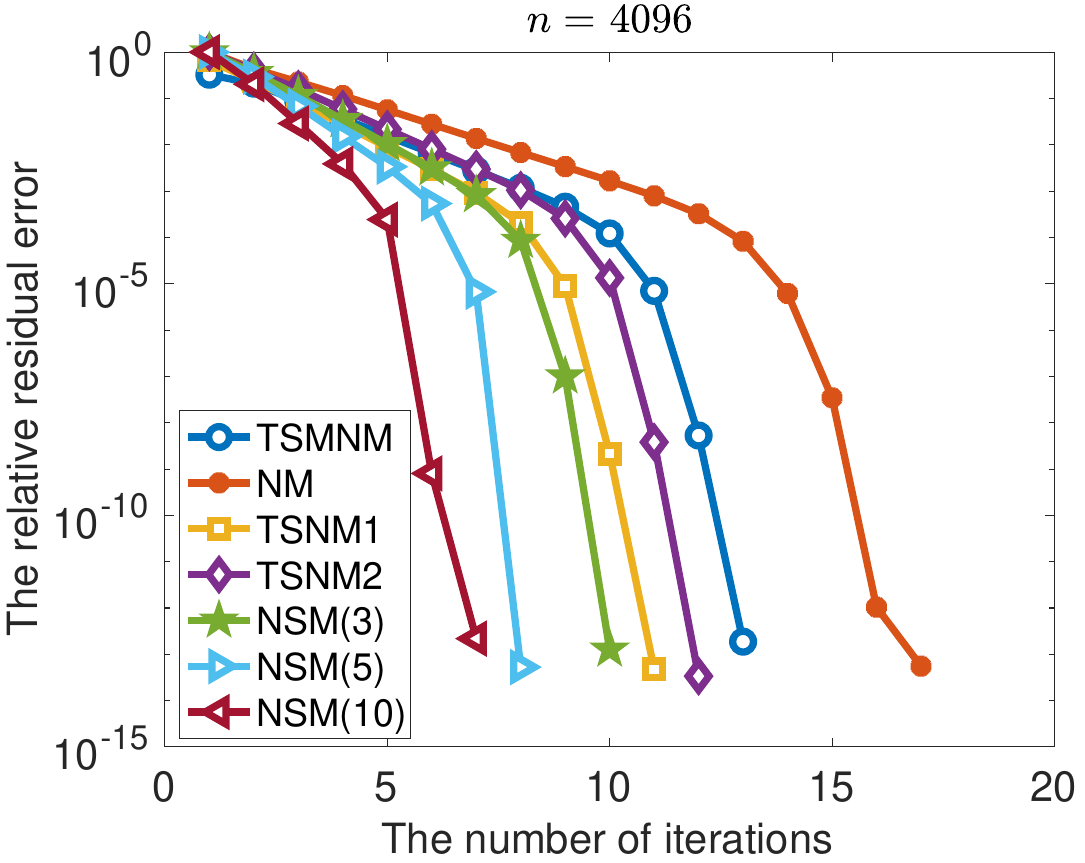}}\quad
  \subfigure{\includegraphics[width=0.42\textwidth]{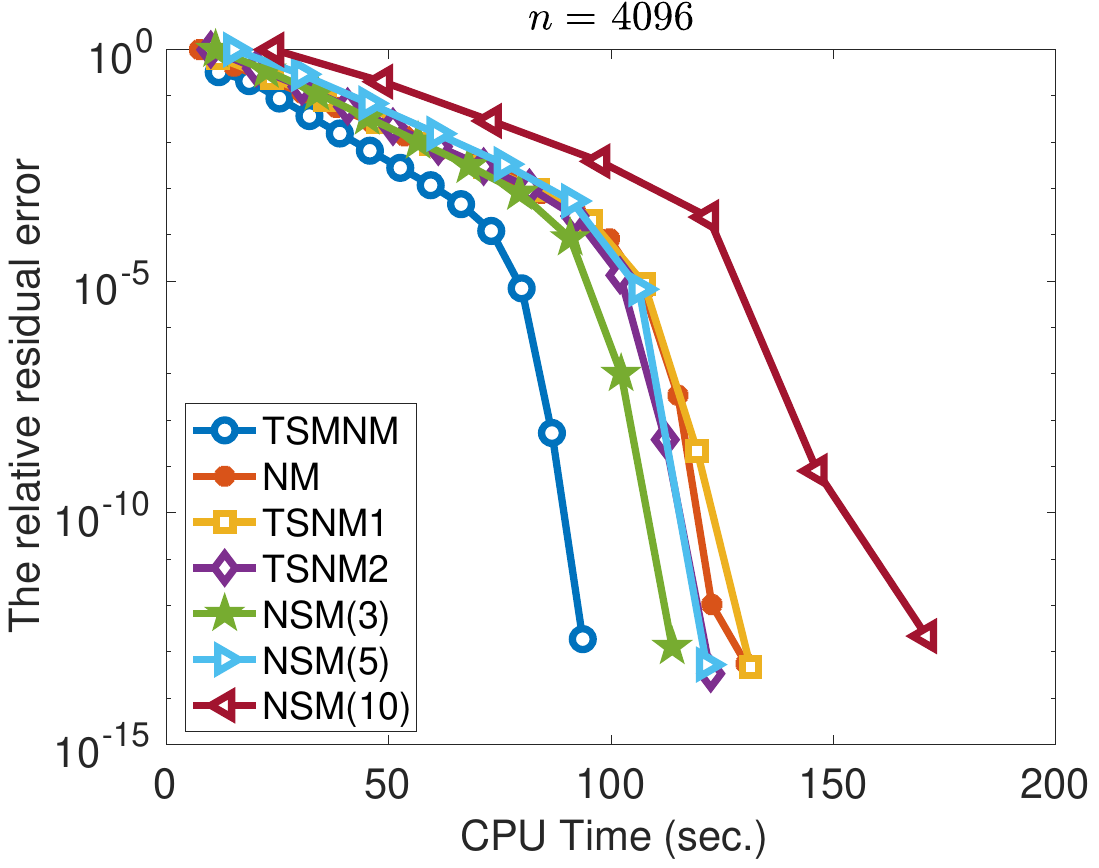}} \\
  \subfigure{\includegraphics[width=0.42\textwidth]{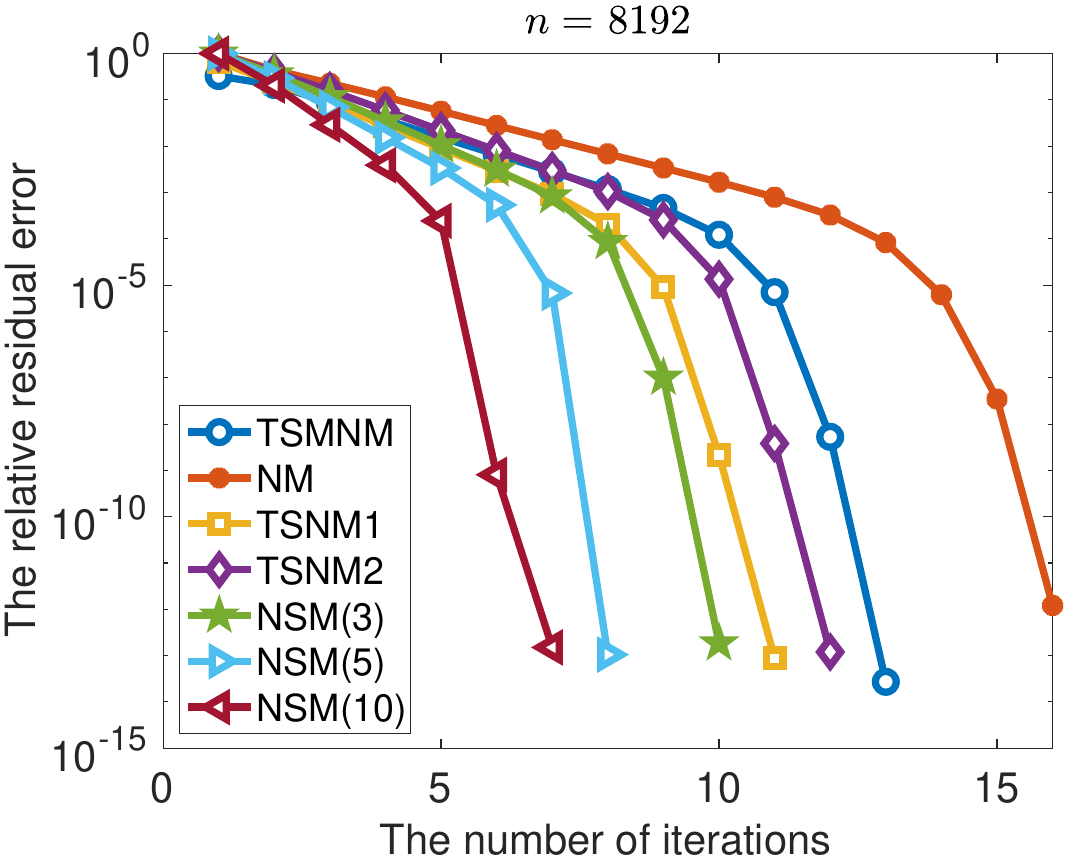}}\quad
  \subfigure{\includegraphics[width=0.42\textwidth]{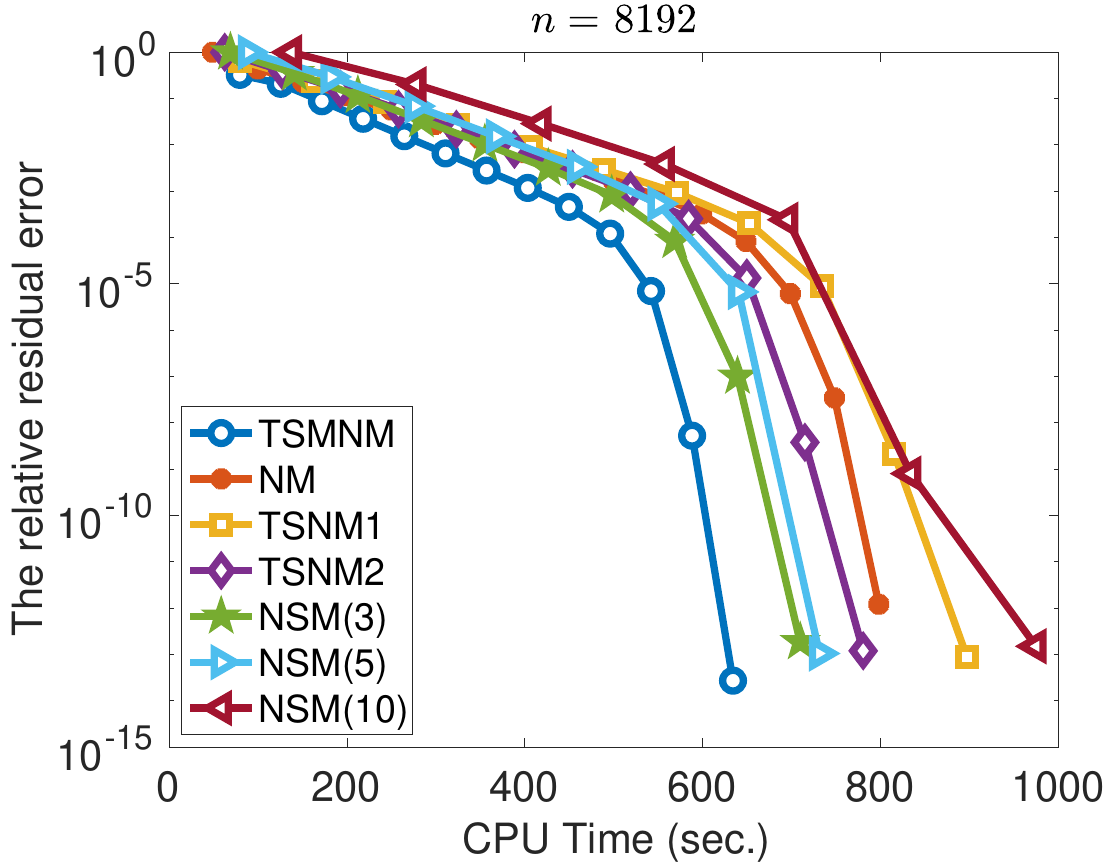}} \\
  \caption{Left: Iterations. Right: Time. Comparison of TSMNM with other Newton-type methods for $(\alpha, c) = (10^{-7},1-10^{-7})$ when the problem size $n = 4096, 8192$, respectively.}
  \label{fig:IterHistory_alphaC10-7}
\end{figure}

In conclusion, TSMNM does not require heavy computation and is more advantageous in execution time, 
especially for nearly singular and large-scale problems.
Although some Newton-type methods outperform TSMNM in terms of the number of iterations, 
TSMNM offers a balance between convergence rate, the desired accuracy and execution time, 
which makes it an effective choice for solving large-scale and nearly singular problems with limited computational resources.

\section{Conclusions}
\label{sec:Conclusions}

In this paper, we studied a two-step modified Newton method for solving 
a nonsymmetric algebraic Riccati equation arising from transport theory.
We first performed a monotone convergence analysis for the proposed method,
obtaining sufficient conditions for convergence.
We then obtained a convergence rate result for the nonsingular case, 
i.e., $\alpha \neq 0$ or $c \neq 1$.
For the singular case $\alpha = 0$ and $c = 1$, 
we presented detailed convergence analysis and error bounds for two types of singular problems.
The numerical experiments demonstrated that the proposed method is competitive with existing methods,
especially for nearly singular and large-scale problems.

\section*{Acknowledgments}

This work was supported by the Fujian Province Natural Science Foundation of China (Grant No. 2022J01896), 
the Education Research Projects for Young Teachers of Fujian Provincial Education Department (Grant No. JAT220197), 
Fujian Key Laboratory of Granular Computing and Applications, 
and Fujian Key Laboratory of Data Science and Statistics.

\section*{Statements and Declarations}

\textbf{Data availability}
The data that support the findings of this study are available from the corresponding author upon reasonable request. \newline

\noindent\textbf{Conflict of interest} 
The authors declare that they have no conflict of interest.


\end{document}